\documentclass[reqno,11pt]{amsart}
\usepackage{amsthm,amsfonts,amssymb,euscript,mathrsfs,graphics,color,amsmath,amssymb,latexsym,marginnote, hyperref}%,mathabx}
\usepackage{soul}

\usepackage{amsthm,amsfonts,amssymb,euscript,mathrsfs,graphics,color,amsmath,amssymb,latexsym,marginnote, url}
\usepackage{hyperref}

\usepackage[margin=1.4in]{geometry}
\usepackage{booktabs}
\numberwithin{equation}{section}

\newcommand{\bydef}{\stackrel{\mbox{\tiny\textnormal{\raisebox{0ex}[0ex][0ex]{def}}}}{=}}

\def\hat{\widehat}

\def\bar{\overline}
\def\R{{\mathbb R}}

\newcommand{\ba}{\bar{a}}

\newcommand{\hpi}{\hat{\pi}}

\usepackage{graphicx}
\usepackage{float} % for [H]

\def\R{{\bf R}}

\def\bar{\overline}

\def\R{\mathbb{R}}

\definecolor{bluegreen}{rgb}{0.0, 0.3, 0.9}

\newtheorem{theorem}{Theorem}[section]
\newtheorem{lemma}[theorem]{Lemma}

\newtheorem{corollary}[theorem]{Corollary}

\newtheorem{remark}[theorem]{Remark}

\title{Global Bifurcation and the Constructive Existence of Overhanging Periodic Steady Water Waves}

\author{Matthieu Cadiot}
\address{CMAP, CNRS, Ecole polytechnique, Institut Polytechnique de Paris, 91120 Palaiseau, France}
\email{matthieu.cadiot@polytechnique.edu}

\author{Susanna V. Haziot}
\address{Department of Mathematics, Fine Hall, Princeton University, Princeton, NJ 08544, USA}
\email{susanna.haziot@princeton.edu}

\begin{document}

\begin{abstract}
We provide a constructive existence proof for overhanging periodic gravity water waves with constant vorticity. By introducing a conformal mapping formulation, we parameterize the fluid domain to accommodate multi-valued surface heights, bypassing coordinate singularities that arise in traditional domain-flattening frameworks. We operate at fixed, macroscopic physical parameters and employ a global bifurcation framework, supplemented by a computer-assisted proof based on the Newton-Kantorovich theorem, to constructively prove the existence of  the branch of exact solutions. 

Crucially, this allows us to provide the first rigorous existence proof of overhanging periodic gravity water waves obtained along a finite-depth global bifurcation branch from the flat state at fixed $O(1)$ physical parameters. Moreover, our analysis confirms the existence of a global solution branch that transitions to overhanging profiles and resolves the conjecture of Constantin, Strauss, and Varvaruca \cite{constantin_global} regarding the topological termination of this branch via physical self-intersection at the trough line. 
\end{abstract}

\maketitle
\tableofcontents
	
	\section{Introduction}	

Recent years have witnessed remarkable breakthroughs in the analytical theory of extreme water waves. Notably, Chen, Walsh, and Wheeler \cite{chen2025extreme} demonstrated the existence of overturning fronts in internal waves using a domain-flattening, Dubreil-Jacotin-type formulation. Similarly, Dávila, del Pino, Musso, and Wheeler \cite{davila2026overhanging} provided a highly sophisticated proof of existence for overhanging solitary waves via singular perturbation. In the periodic setting, Hur and Wheeler \cite{hur2022overhanging} constructed overhanging and touching waves with constant vorticity by perturbing from exact zero-gravity/Crapper-type solutions for sufficiently small gravity. Their result establishes overhanging periodic gravity waves in a perturbative regime, but does not determine the geometry of the finite-depth global bifurcation branch from flat water at fixed $O(1)$ physical parameters.

While these milestone works definitively established the existence of extreme profiles in their respective contexts, their specific analytical methodologies encounter hard boundaries when attempting to explicitly track the geometry of periodic surface gravity waves. The coordinate singularity induced by domain-flattening techniques precludes the parameterization of the overhang, while the asymptotic nature of singular perturbation techniques limits their application to specific solitary-wave regimes.

By contrast, the global bifurcation framework, combined with a conformal mapping parameterization, has been successfully used to construct solution continua for periodic waves with constant vorticity, general vorticity, and stratified regimes \cite{constantin_global, haziot2021stratified, wahlen2024large}. Despite these successes, these global methods often reveal multiple potential termination alternatives, such as overturning, corner formation, or other degeneracies, leaving the specific limiting behavior of a given branch analytically unresolved. In this work, we focus on the periodic surface gravity wave problem with constant vorticity. By integrating a rigorous parameter continuation framework with a computer-assisted proof based on the Newton-Kantorovich theorem, we demonstrate that, within our specific parameter configuration, the solution branch terminates at an overhanging profile with self-intersection on the trough line. This allows us to resolve the topological termination conjecture formulated by Constantin, Strauss, and Varvaruca \cite{constantin_global}, providing the first conclusive existence proof for overhanging periodic surface gravity waves obtained along a finite-depth global bifurcation branch from the flat state at fixed macroscopic physical parameters. While other physical parameter regimes may lead to alternative limiting behaviors, such as the formation of a wave of greatest height, our constructive analysis confirms that the overhanging transition, followed by a self-intersection, is the physically realized outcome in the regime under consideration.

	\subsection{Presentation of the problem}
We consider steady, two-dimensional, periodic gravity water waves with constant vorticity $\gamma \in \mathbb{R}$, propagating with a wave speed $c$ over a flat bed. In a frame of reference moving with the wave, the fluid domain $\mathcal{D}$ and the free surface $\mathcal{S}$ are stationary. The physical coordinates are $(X,Y)$, with the impermeable bottom at $Y=0$ and the free surface at $Y = \eta(X)$. 

Let $\mathbf{U}(X,Y) = (U(X,Y), V(X,Y))$ denote the relative velocity field, $P(X,Y)$ the pressure, and $g$ the gravitational constant. The fluid dynamics are governed by the steady incompressible Euler equations:
\begin{equation}\label{euler_steady}
U_X + V_Y = 0, \qquad U U_X + V U_Y = -P_X, \qquad U V_X + V V_Y = -P_Y - g \qquad \text{in } \mathcal{D}.
\end{equation} 
On the free surface $\mathcal{S}$, the pressure is equal to the constant atmospheric pressure $P_{\text{atm}}$, and the kinematic boundary condition $\mathbf{U} \cdot \mathbf{N} = 0$ requires that $\mathcal{S}$ is a streamline. Furthermore, the velocity field satisfies the constant vorticity condition $V_X - U_Y = \gamma$ throughout $\mathcal{D}$.

By incompressibility, there exists a relative stream function $\Psi(X,Y)$ such that $U = \Psi_Y$ and $V = -\Psi_X$. To work with a purely harmonic function, we define the modified stream function $\zeta(X,Y) \bydef \Psi(X,Y) + \frac{\gamma}{2} Y^2$. Consequently, $\Delta \zeta = 0$ in $\mathcal{D}$, and the physical velocities are recovered via $U = \zeta_Y - \gamma Y$ and $V = -\zeta_X$.

Because $\mathcal{S}$ is a streamline, the kinematic condition dictates that $\zeta - \frac{\gamma}{2} Y^2 = m$ on $\mathcal{S}$, where $m$ is the constant relative mass flux. Integrating the momentum equations \eqref{euler_steady} yields Bernoulli's equation, $\frac{1}{2}(U^2 + V^2) + P + gY + \gamma \Psi = E$. Evaluating this on the free surface and defining the modified Bernoulli constant $Q \bydef 2(E - P_{\text{atm}} - \gamma m)$, the dynamic boundary condition reduces to:
\begin{equation}\label{eq: BC dynamic physical}
U^2 + V^2 = Q - 2gY \qquad \text{on } \mathcal{S}.
\end{equation}

We now map one physical period of $\mathcal{D}$ to the fixed rectangular domain $\Omega \bydef \{(x,y) \in \mathbb{R}^2 : -\pi < x < \pi, ~ -h < y < 0\}$ via a conformal transformation $X+iY = \xi(x,y) + i\eta(x,y)$. This maps the flat bottom $y=-h$ to $Y=0$, and $y=0$ to $\mathcal{S}$. By conformality, $\xi$ and $\eta$ satisfy the Cauchy--Riemann equations ($\xi_x = \eta_y$, $\xi_y = -\eta_x$) and are harmonic in $\Omega$. 

The composition of $\zeta$ with the conformal map remains harmonic in $\Omega$. On the flat boundary $y=0$, the kinematic condition becomes $\zeta - \frac{\gamma}{2}\eta^2 = m$, which differentiates to $\zeta_x = \gamma \eta \eta_x$. Using the Cauchy--Riemann equations, the relative velocity field evaluated in conformal coordinates, $(u,v) \bydef (U,V)$, is explicitly given by
\begin{equation}\label{velocity_conformal}
(u,v) = \left( \frac{\eta_y(\zeta_y - \gamma \eta \eta_y)}{\eta_x^2 + \eta_y^2}, \, \frac{-\eta_x(\zeta_y - \gamma \eta \eta_y)}{\eta_x^2 + \eta_y^2} \right).
\end{equation}
Summing their squares yields the conformal kinetic energy $u^2 + v^2 = (\zeta_y - \gamma \eta \eta_y)^2 / (\eta_x^2 + \eta_y^2)$. Substituting this into \eqref{eq: BC dynamic physical} and multiplying by the Jacobian $(\eta_x^2 + \eta_y^2)$ yields the exact conformal dynamic boundary condition. 

To physically close the system, we enforce $2\pi$-periodicity, set the mean water depth to $h$, impose a zero-drift horizontal gauge to eliminate translational invariance, and assume crest-symmetry. Collecting these constraints, the full fluid dynamics system is governed by the following elliptic boundary value problem:
\begin{subequations}\label{eq : equations for zeta and eta}
	\begin{align}
	\Delta\zeta&=0&\qquad&\text{in }\Omega, \label{eq:pde_zeta}\\
	\Delta\eta&=0&\qquad&\text{in }\Omega, \label{eq:pde_eta}\\
	\zeta-\frac{\gamma}{2}\eta^2-m&= 0&\qquad&\text{on } y=0, \label{eq:pde_kinematic}\\
	(\zeta_y-\gamma\eta\eta_y)^2-(Q-2g\eta)(\eta_x^2+\eta_y^2)&=0&\qquad&\text{on } y=0, \label{eq:pde_dynamic}\\
	\eta=0, \quad \zeta&=0&\qquad&\text{on }y=-h. \label{eq:pde_bottom}
	\end{align}
	
	\noindent We require that the fluid variables are $2\pi$-periodic and symmetric with respect to $x$:
	\begin{equation}\label{eq:symmetry_block}
	\eta(\pm x+2\pi, y) = \eta(x,y), \quad \zeta(\pm x+2\pi, y) = \zeta(x,y), \quad \xi(x+2\pi, y) = \xi(x,y) + 2\pi,
	\end{equation}
	and prescribe the physical depth and gauge conditions:
	\begin{equation}\label{eq:gauge_block}
	\frac{1}{2\pi}\int_{-\pi}^{\pi} \eta(x,0) \, dx = h, \qquad \frac{1}{2\pi}\int_{-\pi}^{\pi} (\xi(x,0)-x) \, dx = 0.
	\end{equation}
\end{subequations}

To guarantee that solutions to \eqref{eq : equations for zeta and eta} correspond to physically valid, non-degenerate waves with a single crest per period, we define an open admissible set of profiles. First, to ensure a single crest and trough while allowing for overhanging geometries, the vertical elevation must be strictly monotonic on the half-period:
\begin{equation}\label{eq : monotonicity_block}
\eta_x(x, 0) < 0 \qquad \text{for all } x \in (0, \pi).
\end{equation}
Given this monotonicity and the crest-symmetry, global injectivity of the conformal mapping is guaranteed entirely by preventing the free surface from crossing the vertical axes of symmetry \cite{constantin_global}:
\begin{equation}\label{eq:injectivity_condition}
0 < \xi(x, 0) < \pi \qquad \text{for all } x \in (0, \pi).
\end{equation} 
Finally, the mapping must remain locally invertible and the free surface must be free of stagnation points. From the dynamic boundary condition, these non-degeneracy requirements are:
\begin{equation}\label{eq : nondegeneracy_block}
\eta_x^2 + \eta_y^2 \geq \epsilon > 0 \quad \text{in } \overline{\Omega}, \qquad \text{and} \qquad Q - 2g\eta \geq \delta > 0 \quad \text{on } y=0.
\end{equation}

In the subsequent computer-assisted proof, we trace a continuous global branch of solutions to \eqref{eq : equations for zeta and eta}. A critical component of this proof consists of rigorously verifying that these strict open conditions are preserved along the entirety of the continuation curve, which eventually terminates precisely when the injectivity condition \eqref{eq:injectivity_condition} fails, resulting in a self-intersecting overhanging profile.

\subsection{Notation and function spaces}\label{sec:notation_spaces}

Because we seek periodic water waves that are symmetric with respect to the crest, we derive an equivalent zero finding problem on Fourier coefficients (cf. Section \ref{sec : single wave and fourier analysis}), for which we restrict our analysis to strictly real, even Fourier sequences. In practice, we equivalently solve for Fourier coefficients $a = (a_n)_{n\in \mathbb{Z}}$, allowing to reconstruct $\eta$ as in \eqref{eq:fourier_series_eta_and_zeta}.  We define the classical Wiener algebra $\ell^1_e$ and the weighted subspace $X_e$ corresponding to continuously differentiable profiles:
\begin{align*}
\ell^1_e &\bydef \left\{ a = (a_n)_{n \in \mathbb{Z}} : a_n \in \mathbb{R}, \,\, a_{-n} = a_n, \text{ and } \|a\|_{\ell^1} < \infty \right\}, &\quad \|a\|_{\ell^1} &\bydef  \sum_{n \in \mathbb{Z}} |a_n|, \\
X_e &\bydef \left\{ a \in \ell^1_e : \|a\|_X < \infty \right\}, &\quad \|a\|_X &\bydef  \sum_{n\in \mathbb{Z}}^\infty (1+|n|) |a_{n}|.
\end{align*}
Note that $\ell^1_e$ and $X_e$ are the even restrictions of the Banach spaces $\ell^1 = \left\{ a = (a_n)_{n \in \mathbb{Z}} : \|a\|_{\ell^1} < \infty \right\}$ and~$X = \left\{ a \in \ell^1(\mathbb{Z}) : \|a\|_X < \infty \right\}$. 
As demonstrated below in Lemma \ref{lem:Banach_algebra}, $X$ forms a Banach algebra under the discrete convolution, which is required to evaluate nonlinear boundary terms. 

Finally, because the Bernoulli flux parameter $Q$ is an unknown active variable, we define the augmented state vector $U = (Q, a)$. The exact steady water wave solutions correspond to the roots of an abstract zero-finding problem, $F(U) = 0$, which encapsulates the nonlinear boundary conditions and will be explicitly derived in Section \ref{sec : zero_finding_problem}. This problem maps the domain $H_1$ into the codomain $H_0$, where these product Banach spaces are defined as:
\begin{align*}
    H_1 &\bydef \mathbb{R} \times X_e, \qquad \text{with norm} \qquad \|U\|_{H_1} \bydef |Q| + \|a\|_X, \\
    H_0 &\bydef \mathbb{R} \times \ell^1_e, \qquad \text{with norm} \qquad \|U\|_{H_0} \bydef |Q| + \|a\|_{\ell^1}.
\end{align*}
Throughout this paper, all local isolating neighborhoods and analytic contraction mapping arguments are established strictly with respect to the $H_1$ topology. Furthermore, for any Banach spaces $X$ and $Y$, we denote the space of bounded linear operators from $X$ to $Y$ by $\mathcal{B}(X,Y)$, equipped with the standard induced operator norm, and we abbreviate $\mathcal{B}(X) \bydef \mathcal{B}(X,X)$.

\subsection{Statement of the main results}
In this section, we formally state the main existence theorems of the paper.
% We emphasize that the theorems presented herein are mathematically exact statements regarding the infinite-dimensional boundary value problem \eqref{eq : equations for zeta and eta}. While the proofs rely on computer-assisted estimates, the underlying methodology is fundamentally analytic. We utilize a topological fixed-point argument—specifically, the Newton-Kantorovich approach—within the augmented Banach space $H_1$. The computational machinery is strictly relegated to evaluating explicit, finite-dimensional bounds on the residual and the Fréchet derivative of our nonlinear operators. Because these rigorous bounds satisfy the prescribed contraction criteria, the existence of a unique, exact, and classical solution within a specific topological neighborhood is guaranteed analytically.
A central focus of our analysis is the rigorous geometric characterization of the wave profiles. By reconstructing the physical spatial coordinates via the Cauchy--Riemann boundary relation $\xi_x(x,0) =  \eta_y(x,0)$, we can formally track the geometry of the free surface. Specifically, a wave ceases to be a single-valued graph over the horizontal axis—thereby becoming physically overhanging—if and only if there exists a conformal coordinate $x^*$ such that $\xi_x(x^*,0) < 0$. 

Our first main result establishes the existence of an isolated, highly nonlinear water wave. Notably, we rigorously prove that this wave is strictly overhanging yet entirely devoid of stagnation points. Because the crest physically bulges over the trough, we classify this profile as an ``$\Omega$"-shaped wave. A representation of an approximation of the solution is provided in Figure \ref{fig : omega wave}, and a more detailed presentation of the result is given in Theorem \ref{th : existence single wave overhanging 085}.

\begin{theorem}[Existence of an isolated $\Omega$-shaped wave]\label{thm : main result single wave}
	Fix the physical parameters to the dimensionless gravitational constant $g=1$, constant vorticity $\gamma = -5$, and mean water depth $h=2$. 
	
	For the relative mass flux $m = -0.85$, there exists a unique, exact, and infinitely differentiable $2\pi$-periodic solution $(\tilde{Q}, \tilde{\eta}) \in \R \times C^\infty[-\pi,\pi]$ to the water wave equations \eqref{eq : equations for zeta and eta}. Furthermore, this exact solution strictly satisfies the non-degeneracy conditions \eqref{eq : nondegeneracy_block} and constitutes a physically valid, strictly overhanging, and injective wave.
\end{theorem}

\begin{figure}[H]
\caption{Visualization (on two periods) of the $\Omega$-shaped wave proven in Theorem \ref{thm : main result single wave}}
    \centering
    \begin{minipage}{0.8\linewidth}
        \centering
\includegraphics[width=\linewidth]{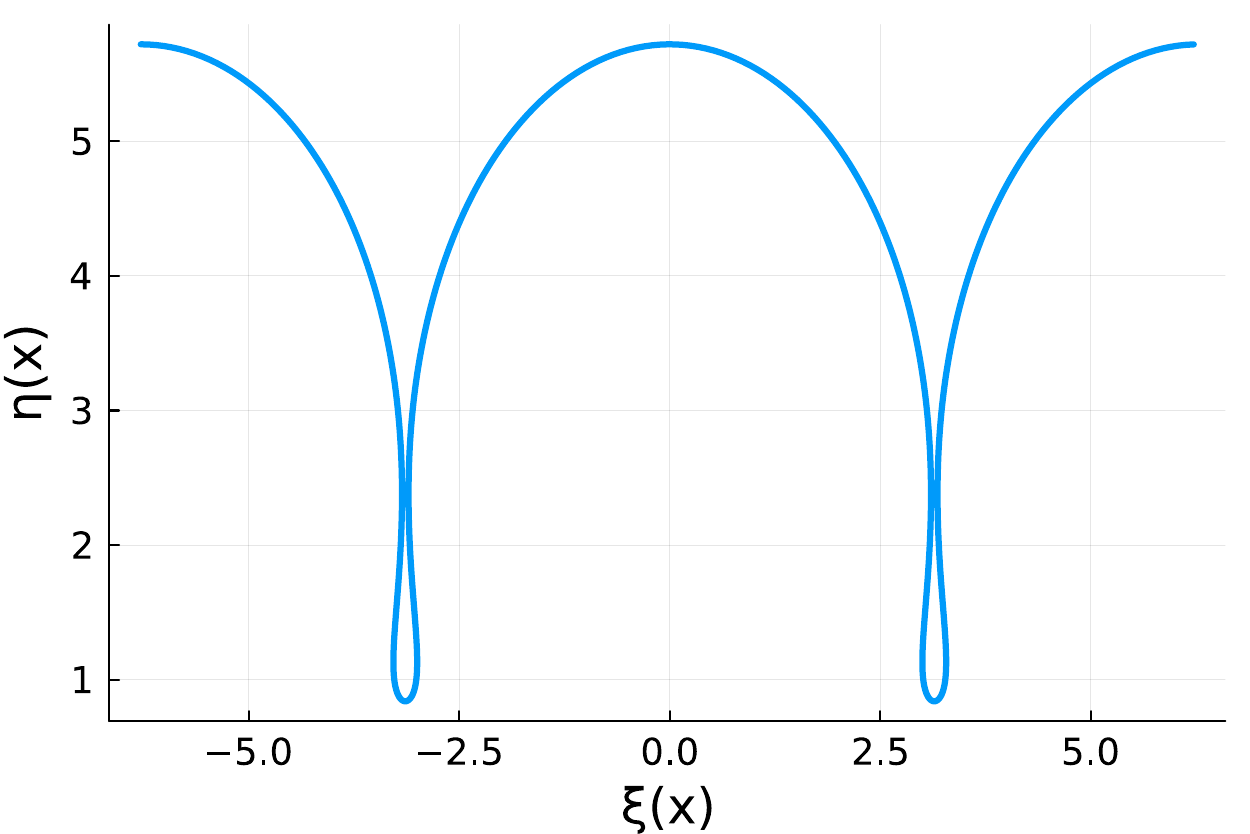}
    \end{minipage}
    \label{fig : omega wave}
\end{figure}

While isolating a single extreme wave is of significant physical interest, our analytical framework allows us to  extend this result into a global continuation argument. Our second main result constructively proves the existence of a continuous global branch of exact steady waves in $H_1$, parameterized by the mass flux $m$. Crucially, we rigorously track the topological evolution of the free surface along this branch, partially resolving the conjecture formulated in \cite{constantin_global} regarding the limiting geometries of water waves. A representation of the global branch is provided in Figure \ref{fig : branch of solutions}. A detailed presentation of the result is exposed in Section \ref{sec : existence proofs results}.

\begin{theorem}[Global bifurcation and topological transitions]\label{thm : main result global branch}
	Let $g=1$, $\gamma=-5$, and $h=2$. There exists a continuous, infinitely differentiable branch of exact $2\pi$-periodic water wave solutions, parameterized by the relative mass flux $m \in [-1, \bar{m}]$, where $\bar{m} \approx -0.061$ is the local bifurcation point. 
	
	The physical geometry of the free surface undergoes a rigorous sequence of topological transitions as the branch is traversed from the trivial state:
	\begin{enumerate}
		\item \textbf{Local Bifurcation \& Graphs:} The branch bifurcates from the flat-water state at $m = \bar{m}$. For fluxes near the bifurcation point, the exact wave profiles are strictly monotonic graphs over the horizontal axis.
		\item \textbf{Transition to Overhanging Waves:} There exists a strictly bounded critical flux $\tilde{m}_1 \in [-0.698, -0.6707]$ at which the wave profile loses global monotonicity in the horizontal coordinate. For all $m \in [-1, \tilde{m}_1)$, the waves are strictly overhanging.
		\item \textbf{The Non-Physical Limit:} As the wave becomes increasingly overhanging, there exists a second critical flux $\tilde{m}_2 \in [-0.8884, -0.8781]$ at which the physical injectivity condition \eqref{eq:injectivity_condition} fails. For $m < \tilde{m}_2$, the free surface self-intersects across the trough line, rendering the extreme end of the mathematical branch physically invalid.
	\end{enumerate}
\end{theorem}

\begin{figure}[H]
  \centering
  \includegraphics[width=0.9\linewidth]{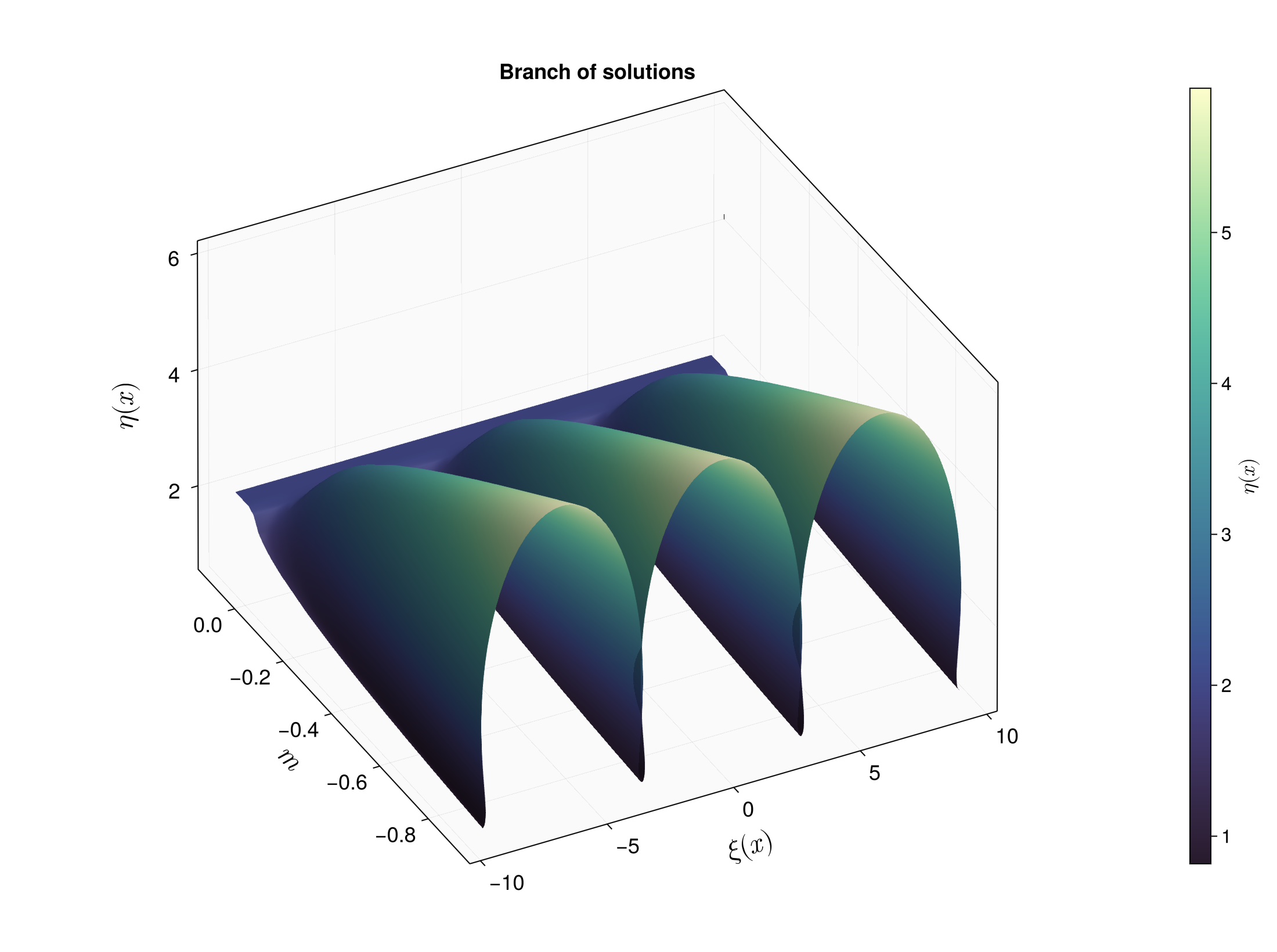}
  \caption{Visualization (on three periods) of the global branch of periodic water waves proven in Theorem \ref{thm : main result global branch}.}
  \label{fig : branch of solutions}
\end{figure}

% \mc{useful or not ?}

% \begin{remark}
%     Note that there is nothing decisive about the specific values of the parameters $\gamma$ and $h$. In principle, the presented methodology applies to any values of $\gamma$ and $h$, as long as the solutions remain smooth. 
% \end{remark}

% The remainder of this paper is dedicated to the rigorous functional analysis required to establish these theorems. In the subsequent sections, we derive the analytic operator bounds, prove the invertibility of the linearized systems, and explicitly construct the topological isolating neighborhoods that mathematically validate the global branch.

\subsection{Past work}
\subsubsection{Steady water waves}
The study of extreme water waves traces its origins to Stokes \cite{stokes1847theory}, who conjectured the existence of a branch of nontrivial periodic waves limiting to an extreme "wave of greatest height," characterized by a stagnation point at its crest and a singular interior corner angle of 120 degrees. The existence of this extreme wave was rigorously established by Toland \cite{toland1978existence}, and the 120 degree crest angle conjecture for extreme irrotational waves was subsequently confirmed by Amick, Fraenkel, and Toland \cite{amick1982stokes} alongside Plotnikov \cite{plotnikov1982justification, plotnikov1983proof}. Within the irrotational regime, the fluid geometry is highly rigid; as proved by Spielvogel \cite{spielvogel1970variational}, the horizontal velocity is strictly non-vanishing inside the fluid, ensuring that all interior streamlines remain simple graphs. However, the introduction of vorticity fundamentally alters this landscape. Numerical calculations by Simmen and Saffman \cite{simmen1985steady} for periodic waves in infinite depth with large adverse vorticity demonstrated that solution branches can terminate at "touching waves," where the free surface self-intersects tangentially, bypassing the traditional corner singularity entirely.

The possibility of overturning and overhanging wave profiles has been recognized computationally since the late 1970s. Holyer \cite{holyer1979large} provided the first computations of overhanging internal waves along an interface separating immiscible fluid layers by employing computer algebra to solve recurrence relations for high-order Fourier series coefficients. Subsequent numerical continuation studies—initiated from a flat laminar flow—further captured overhanging internal waves (Meiron and Saffman, \cite{meiron1983overhanging}; Pullin and Grimshaw, \cite{pullin1988finite}; Turner and Vanden-Broeck, \cite{turner1988broadening}). For free-surface flows, Teles da Silva and Peregrine \cite{telesdasilva1988steep} computed overhanging periodic waves with constant vorticity, predicting a striking fluid domain geometry where a nearly circular constant vorticity region is perched atop a body of nearly laminar flow. Extensive numerical investigations into these overturning profiles have continued over the following decades, including works by Dias and Vanden-Broeck \cite{dias2003internal}, Maklakov and Sharipov \cite{maklakov2018almost}, Maklakov \cite{maklakov2020note}, Hur and Wheeler \cite{hur2020exact}, Guan et al. \cite{guan2021local}, and Dyachenko and Hur \cite{dyachenko2019stokes}.

Analytically, a major breakthrough was achieved by Constantin, Varvaruca, and Strauss \cite{constantin_global}, who developed a global bifurcation framework that accommodates overturning geometries to construct families of periodic waves with constant vorticity. This approach has since been generalized to domains with arbitrary vorticity (Wahl'{e}n and Weber, \cite{wahlen2024large}), internal waves with linear density stratification (Haziot, \cite{haziot2021stratified}), solitary waves with constant vorticity (Haziot and Wheeler, \cite{haziot2023large}), and solitary waves carrying point or hollow vortices (Chen, Varholm, Walsh, and Wheeler, \cite{chen2025vortex}). Furthermore, Chen, Walsh, and Wheeler \cite{chen2025extreme} recently provided a rigorous proof that extreme hydrodynamic bores must overturn by developing novel geometric analysis techniques to exclude double-stagnation degeneracies. Despite these advancements, a rigorous proof of overturning along these global solution branches has remained elusive within the standard periodic framework, primarily due to the analytic difficulty of definitively excluding alternative geometric singularities, such as corner formation.

To bypass the challenges of tracking a global continuum, several authors have pursued direct perturbative constructions of overhanging waves. Hur and Wheeler \cite{hur2022overhanging} constructed periodic overhanging and touching waves with constant vorticity by perturbing from the explicit zero-gravity Crapper-type solutions; their result applies for sufficiently small gravitational effects and produces waves near this special integrable regime. Related constructions, including the works of Akers et al. \cite{akers2014gravity}, Cordoba et al. \cite{cordoba2016existence}, and de Carvalho \cite{decarvalho2023gravity}, similarly exploit perturbative proximity to explicit or singular limiting configurations. In the solitary-wave setting, Davila, del Pino, Musso, and Wheeler \cite{davila2026overhanging} employed a delicate gluing method to construct overhanging
constant-vorticity gravity waves in a small-gravity regime, matching a nearly circular rotating fluid region to a laminar shear flow through a thin neck. Their work provided the first rigorous verification of the geometric structure numerically predicted by Teles da Silva and Peregrine \cite{telesdasilva1988steep}, successfully matching a nearly circular constant vorticity region to a laminar flow across a thin neck. However, a fundamental limitation of these direct perturbative methods is that they operate exclusively within the weak gravity regime. The existence of such waves along a finite-depth global bifurcation branch from flat water at fixed $O(1)$ parameters remained open.

The present work addresses a different problem. Rather than perturbing from an explicit zero-gravity solution or a singular small-gravity configuration, we fix macroscopic physical parameters, specifically $g=1$, $h=2$, and $\gamma=-5$, and rigorously validate the global
bifurcation branch emanating from the flat-water state. This allows us to prove, constructively and with explicit bounds, that this particular branch transitions from classical graph profiles to overhanging profiles and ultimately loses physical injectivity through self-intersection at the trough line.

\subsubsection{Computer Assisted Proofs}

Proving the existence of solutions to nonlinear PDEs in a constructive and rigorous manner is, in general, an extremely challenging, if not impossible, task. This difficulty stems naturally from the fact that the solutions belong to infinite-dimensional spaces. Over the past decades, computer-assisted proof techniques have emerged as a powerful framework for overcoming this obstacle and establishing rigorous existence results for nonlinear equations.

Their effectiveness has been demonstrated across a broad range of applications, including the Feigenbaum conjectures \cite{MR648529}, the existence of chaos and a global attractor for the Lorenz equations \cite{MR1276767,MR1701385,MR1870856}, Wright's conjecture \cite{MR3779642}, and chaotic dynamics in the Kuramoto--Sivashinsky equation \cite{MR4113209}. In fluid dynamics, computer-assisted proofs have become extremely valuable for the rigorous validation of singular phenomena. Notable examples include splash singularities in the water wave equations \cite{MR2881486, MR3223860}, turning waves for the Muskat problem \cite{MR3215843}, finite-time blow-up for the three-dimensional Euler equations \cite{MR4977334}, and imploding solutions for three-dimensional compressible fluids \cite{MR4862915}. We refer the interested reader to the review articles \cite{nakao_numerical, gomez_cap, jb_rigorous_dynmamics, koch_computer_assisted} and the monograph \cite{plum_numerical_verif} for further background on computer-assisted proof techniques.

The objective of the present work is to combine computer-assisted methods with nonlinear analysis to rigorously construct a global branch of steady periodic gravity water waves, culminating in an overhanging wave whose profile self-intersects along the trough line.

Our proof is fundamentally grounded in functional analysis. More precisely, we apply a Newton--Kantorovich-type fixed-point theorem to the zero-finding problem
$
F(U)=0, \quad  F : H_1 \to H_0
$
where the nonlinear operator $F$ and the Banach spaces $H_0, H_1$ are defined below. By expressing the contraction mapping conditions as polynomial inequalities, involving norms of infinite-dimensional objects, the original infinite-dimensional existence problem is reduced to verifying a finite collection of explicit strict algebraic inequalities.

Each of these inequalities is established through a decomposition into a finite-dimensional component and an infinite-dimensional remainder. The finite-dimensional estimates are computed rigorously using interval arithmetic \cite{Moore_interval_analysis}, while the infinite-dimensional tail is controlled analytically.

Newton--Kantorovich methods have proved highly successful for the constructive validation of continuous branches of solutions. In particular, the uniform contraction mapping theorem (stated below as Theorem~\ref{th : radii polynomial continuation}) provides a framework for rigorously continuing solution branches. This methodology has been applied, for example, to branches of steady states in PDEs \cite{cont_equilibria_pde, cont_global_bif_diag} and to branches of homoclinic connections in ODEs \cite{cont_suspension_bridge}. In the present work, we build upon the continuation framework developed in \cite{breden_polynomial_chaos}, where solution branches are represented by Chebyshev interpolants. By exploiting the spectral accuracy of Chebyshev expansions, this approach yields highly accurate constructive existence proofs, even for long solution branches. It has already led to several notable results, including the resolution of Marchal's conjecture in the three-body problem \cite{henot_marchal} and the constructive existence of solitary waves for the nonlocal quasilinear Whitham equation \cite{whitham_cadiot}.

\subsection{Plan of the paper}

The remainder of this paper is organized as follows. Section~\ref{sec: qualitative} establishes the qualitative geometric properties of the fluid domain, specifically strict vertical monotonicity and crest-axis injectivity. This defines an open admissible set of physically valid wave profiles, analytically locking these geometric traits to streamline the functional analytic framework required for our existence proofs. Section~\ref{sec : single wave and fourier analysis} introduces the Fourier-analytic formulation of the physical problem. By passing to the Fourier domain in the horizontal conformal variable, we reduce the full two-dimensional boundary value problem to a one-dimensional algebraic zero-finding problem $F(U) = 0$ restricted to the free surface, providing the functional setting needed to track the solution continua.

Section~\ref{sec: CA bounding} develops the analytical framework for the existence proof of a single wave (fixed parameters) using a Newton-Kantorovich approach. We explicitly construct an approximate inverse for the linearization $DF(\bar{U})$ around an approximate solution $\bar{U}$ and derive quantitative bounds to guarantee the existence of an exact solution in a vicinity of $\bar{U}$. Section~\ref{sec : local bif} focuses on the analytical formulation of the local branch of solutions. By introducing a desingularized scaling operator at the pitchfork bifurcation, we mathematically prove the emergence of non-trivial, monotonic wave profiles from the flat-water state in a constructive way. Building on this, Section~\ref{sec : global bif} extends the methodology to establish the global branch. We develop a rigorous parameter continuation framework, grounded in uniform contraction mapping principles, to continuously track the solution space macroscopically as the relative mass flux varies.

Section~\ref{sec : geometric properties branch} outlines the formal verification of the geometric properties along the branch. This provides the analytical techniques needed to guarantee the preservation of smoothness and non-degeneracy, and establishes the criteria used to definitively prove precise topological transitions into overhanging and self-intersecting regimes. Section~\ref{sec : existence proofs results} then synthesizes our framework to present the main constructive existence theorems. Here, we provide the mathematical proofs for the constructive existence of an isolated extreme overhanging wave, followed by the complete constructive demonstration of the global branch's evolution from its bifurcation to its topological termination.

Finally, the appendices provide the necessary supplementary technical estimates. Appendix~\ref{sec : appendix computational} outlines the rigorous arithmetic procedures and supremum norm evaluations that formally complete the analytic proofs. Appendix~\ref{sec : appendix proofs} collects the detailed analytic derivations, including the Taylor expansions of the desingularized operators and the bounds for the difference of products required to rigorously close the Newton-Kantorovich framework.

\section{Qualitative theory}\label{sec: qualitative}
Before executing the  continuation to trace the global branch of solutions, we must first establish the foundational geometric and topological rules that govern the fluid domain. In this section, we define an open ``admissible set" of physically valid wave profiles and analytically prove a suite of qualitative properties—specifically, strict vertical monotonicity and crest-axis injectivity. By demonstrating that these geometric traits are topologically locked and cannot be lost without violating the underlying elliptic partial differential equations, we drastically reduce the computational burden on our computer-assisted proof. The rigorous numerical algorithm will ultimately only need to verify the scalar bounds for non-stagnation and trough-axis injectivity, while the pure analysis presented here automatically guarantees the structural integrity of the wave's shape.

\subsection{Monotonicity}
The cornerstone of our geometric analysis is ensuring that the wave profile does not develop interior flat spots or degenerate curvatures at the symmetry axes. In this subsection, we prove that any non-stagnating wave that is monotonically decreasing from crest to trough must be \textit{strictly} monotonic. By applying the elliptic Maximum Principle and Hopf's Boundary Lemma to the harmonic conformal vertical velocity, we show that a loss of strict monotonicity leads to a direct physical contradiction. This property acts as an analytical one-way trap that protects the shape of the wave along the global branch.
\begin{theorem}[Preservation of Strict Vertical Monotonicity]\label{thm:hw_conformal_monotonicity}
	Let $(\xi, \eta)$ be a solution to \eqref{eq : equations for zeta and eta} corresponding to a non-flat wave. Assume the mapping is non-degenerate ($\eta_x^2 + \eta_y^2 > 0$ in $\overline{\mathcal{R}}$) and satisfies the non-stagnation condition $Q - 2g\eta > 0$ on the free surface $y=0$.
	
	If the wave profile satisfies $\eta_x \leq 0$ for $x \in (0, \pi)$, then the profile is strictly monotonic on the entire half-period, satisfying:
	\begin{enumerate}
		\item[(i)] $\eta_x < 0$ for all $x \in (0, \pi)$,
		\item[(ii)] $\eta_{xx}(0,0) < 0$ (strict local maximum at the crest),
		\item[(iii)] $\eta_{xx}(\pi,0) > 0$ (strict local minimum at the trough).
	\end{enumerate}
\end{theorem}

\begin{proof}
	Let $(u,v)$ denote the physical fluid velocity evaluated in the conformal variables, as defined in \eqref{velocity_conformal}. By the physical steady Euler equations \eqref{euler_steady}, the velocity field is incompressible ($U_X + V_Y = 0$) and has constant vorticity ($V_X - U_Y = \gamma$). Consequently, the physical vertical velocity is strictly harmonic: $\Delta V = (V_X)_X + (V_Y)_Y = (\gamma + U_Y)_X + (-U_X)_Y = 0$. Because $\xi+i\eta$ is a conformal map, the composition $v(x,y)$ remains strictly harmonic in the right-half conformal domain $\mathcal{R}^+ = \{ (x,y) : 0 < x < \pi, ~ -h < y < 0 \}$.
	
By the even symmetry of the profile \eqref{eq:symmetry_block}, $v(0,y) = 0$ and $v(\pi,y) = 0$. On the flat bottom $y=-h$, the boundary conditions \eqref{eq:pde_bottom} yield $v(x,-h) = 0$. 

On the free surface $y=0$, we differentiate the kinematic boundary condition $\zeta - \frac{\gamma}{2}\eta^2 = m$ with respect to $x$ to obtain $\zeta_x - \gamma\eta\eta_x = 0$. Using the explicit physical velocity definitions $U = \zeta_Y - \gamma Y$ and $V = -\zeta_X$, we evaluate the normal cross-velocity term along the boundary. Applying the standard chain rule $\zeta_x = \zeta_X \xi_x + \zeta_Y \eta_x$, we obtain:
\begin{equation*}
	u\eta_x - v\xi_x = (\zeta_Y - \gamma\eta)\eta_x - (-\zeta_X)\xi_x = \zeta_X \xi_x + \zeta_Y \eta_x - \gamma\eta\eta_x = \zeta_x - \gamma\eta\eta_x.
\end{equation*}
Because this quantity vanishes entirely on the free surface, we recover the strict geometric tangency relation:
\begin{equation}
	v = u \frac{\eta_x}{\xi_x} \qquad \text{on } y=0.
\end{equation}
Alternatively, using the explicit velocity formula \eqref{velocity_conformal}, on the free surface we have
\begin{equation}\label{eq:v_tangent}
v = -\eta_x \frac{\zeta_y - \gamma \eta \eta_y}{\eta_x^2 + \eta_y^2}.
\end{equation}

Moreover, by the dynamic boundary condition \eqref{eq : equations for zeta and eta},
\begin{equation}
(\zeta_y - \gamma \eta \eta_y)^2 = (Q - 2g\eta)(\eta_x^2 + \eta_y^2).
\end{equation}
The non-stagnation condition $Q - 2g\eta > 0$ and the local non-degeneracy of the conformal map imply that $\zeta_y - \gamma \eta \eta_y$ is nonzero on the free surface, hence has a fixed sign on the connected interval $0 < x < \pi$. Since \(\zeta_y-\gamma\eta\eta_y\) has a fixed nonzero sign on the free surface, after possibly
replacing \(v\) by \(-v\) and relabeling, we may assume
\[
v=-\eta_x\frac{|\zeta_y-\gamma\eta\eta_y|}{\eta_x^2+\eta_y^2}
\qquad\text{on }y=0.
\]
Thus, after possibly replacing $v$ by $-v$, the assumption $\eta_x \le 0$ implies $v \ge 0$ on $y = 0$.

By the Maximum Principle, since the chosen harmonic function $v$ or $-v$ is nonnegative on all boundaries of $\mathcal{R}^+$, it must follow that $v \geq 0$ uniformly in $\overline{\mathcal{R}^+}$. Furthermore, because the wave is not entirely flat, $v > 0$ strictly in the interior.
	
	\textbf{Step 1: Strict Interior Monotonicity.} 
Assume for contradiction that $\eta_x(x_0, 0) = 0$ at an interior point $x_0 \in (0, \pi)$. By \eqref{eq:v_tangent}, this implies $v(x_0, 0) = 0$. Since $v \geq 0$ globally, this boundary point is a global minimum. By the Hopf Boundary Lemma, the outward normal derivative must be strictly negative:
\begin{equation}\label{eq:hw_hopf}
    v_y(x_0, 0) < 0.
\end{equation}

To evaluate $v_y$, we apply the multivariable chain rule to the physical vertical velocity: $v_y = V_X \xi_y + V_Y \eta_y$. At $x_0$, since $\eta_x = 0$, the Cauchy--Riemann equations dictate $\xi_y = 0$ and $\eta_y = \xi_x$. Thus, $v_y = V_Y \xi_x$. 
By the physical incompressibility condition in the steady Euler equations \eqref{euler_steady}, we have $V_Y = -U_X$. Therefore:
\begin{equation}\label{eq:vy_ux}
    v_y = -U_X \xi_x = \pm(-u_x),
\end{equation}
 where the last equality follows from the fact that $u_x = U_X \xi_x + U_Y \eta_x$, which reduces to $U_X \xi_x$ at $x_0$ because $\eta_x = 0$.

The dynamic boundary condition evaluated on the free surface \eqref{eq: BC dynamic physical} dictates that $u^2 + v^2 + 2g\eta = Q$. Differentiating this relation with respect to $x$ along the boundary yields:
\begin{equation*}
    2u u_x + 2v v_x + 2g\eta_x = 0.
\end{equation*}
Evaluating this exactly at $x_0$, where $v = 0$ and $\eta_x = 0$, the equation simplifies to $2u u_x = 0$. Because the wave is free of stagnation points, $u \neq 0$, which strictly forces $u_x = 0$. 

Substituting $u_x = 0$ into \eqref{eq:vy_ux} yields $v_y = 0$. This directly contradicts the strict negativity required by the Hopf Boundary Lemma in \eqref{eq:hw_hopf}. Thus, our assumption is false, concluding that $\eta_x < 0$ for all $x \in (0, \pi)$.
	
	\textbf{Step 2: Strict Corner Curvature.}
	At the crest $C = (0,0)$, the domain $\mathcal{R}^+$ forms a $90^\circ$ corner. The harmonic function $v$ is non-negative in $\mathcal{R}^+$ and vanishes uniformly along the vertical axis $x=0$. By Serrin's Edge Point Lemma, the directional derivative entering the domain along the inward normal (the positive $x$-direction) is strictly positive, yielding $v_x(C) > 0$.
	
	At the crest and trough, symmetry gives \(\eta_x=0\). Differentiating the preceding identity
with respect to \(x\), the terms involving derivatives of
\[
\frac{|\zeta_y-\gamma\eta\eta_y|}{\eta_x^2+\eta_y^2}
\]
vanish because they are multiplied by \(\eta_x\). Note that the absolute value is smooth since $\zeta_y-\gamma\eta\eta_y$ never vanishes.  Hence, at both symmetry points,
\[
v_x=-\eta_{xx}\frac{|\zeta_y-\gamma\eta\eta_y|}{\eta_x^2+\eta_y^2}.
\]
The prefactor is strictly positive by non-stagnation and local non-degeneracy. Serrin's edge
lemma gives \(v_x(0,0)>0\) at the crest, so \(\eta_{xx}(0,0)<0\). At the trough, the inward
direction is \(-\partial_x\), so Serrin gives \(-v_x(\pi,0)>0\), equivalently \(v_x(\pi,0)<0\).
Therefore \(\eta_{xx}(\pi,0)>0\).
\end{proof}

\subsection{Injectivity}
With strict vertical monotonicity established, we can mathematically rule out one of the two possible modes of surface self-intersection. A highly nonlinear periodic wave could theoretically fold over itself by crossing either the crest axis ($X=0$) or the trough axis ($X=\pi$). Here, we use the Cauchy--Riemann equations and the Hopf boundary lemma to prove that a strictly monotonic wave can never cross its own crest axis. This analytic guarantee effectively halves the computational requirement for proving injectivity: our numerical continuation algorithm is entirely freed from checking the left boundary, and must only monitor the horizontal distance to the trough axis.
\begin{lemma}[Injectivity at the Crest Axis]\label{lem:crest_injectivity}
	Assume the conformal mapping $\xi + i\eta$ satisfies the symmetry and boundary conditions of \eqref{eq : equations for zeta and eta}. If the free surface is strictly monotonic, meaning $\eta_x(x,0) < 0$ for all $x \in (0,\pi)$, then the wave cannot self-intersect across the crest axis $X=0$. Specifically:
	\begin{equation}
		\xi(x,0) > 0 \qquad \text{for all } x \in (0, \pi).
	\end{equation}
\end{lemma}

\begin{proof}
	Let $\mathcal{R}^+ = (0,\pi) \times (-h,0)$ be the right half of the conformal domain. The real part of the mapping, $\xi$, is harmonic in $\mathcal{R}^+$. By the assumed symmetry and the flat bottom boundary condition, we have $\xi \geq 0$ on the left ($x=0$), right ($x=\pi$), and bottom ($y=-h$) boundaries. 
	
	Suppose, for contradiction, that $\xi(x_0, 0) \leq 0$ for some $x_0 \in (0, \pi)$. Since $\xi \geq 0$ on the other three boundaries, $(x_0, 0)$ is a global minimum of $\xi$ on $\overline{\mathcal{R}^+}$. 
	
	By the Hopf boundary lemma, the outward normal derivative at this minimum must be strictly negative, yielding $\xi_y(x_0, 0) < 0$. Applying the Cauchy--Riemann equation $\xi_y = -\eta_x$, we obtain:
	\begin{equation}
		\eta_x(x_0, 0) > 0.
	\end{equation}
	This strictly contradicts the monotonicity assumption. Therefore, $\xi(x,0) > 0$ for all $x \in (0,\pi)$.
\end{proof}

\subsection{The Admissible Set and Openness}
To utilize these geometric properties in a rigorous continuation argument, we must translate them into a formal topological framework. Our computer-assisted proof provides a rigorous enclosure of the true branch of solutions within a tight mathematical neighborhood (an $\epsilon$-ball) around a numerically computed approximation. In that spirit, our physical constraints must define an \textit{open set} in the function space. In this subsection, we formally define the admissible set of non-stagnating, injective, and strictly monotonic waves, and verify that these conditions possess the necessary topological "wobble room" to survive small perturbations in the $C^2$ norm.
\begin{lemma}[Openness of the Admissible Set]\label{lem:openness}
	Let $X = C^2_{\text{even}}(\mathbb{T})$ denote the Banach space of even, $2\pi$-periodic, twice continuously differentiable functions. The admissible set $\mathcal{U} \subset X$ consisting of wave profiles $\eta$ and their conjugate horizontal mappings $\xi$ satisfying:
	\begin{enumerate}
		\item[(i)] Non-stagnation: $Q - 2g\eta(x,0) > 0$ for all $x$,
		\item[(ii)] Injectivity: $0 < \xi(x,0) < \pi$ for all $x \in (0, \pi)$,
		\item[(iii)] Strict Monotonicity: $\eta_x(x,0) < 0$ for all $x \in (0, \pi)$, with $\eta_{xx}(0,0) < 0$ and $\eta_{xx}(\pi,0) > 0$,
	\end{enumerate}
	is an open subset of $X$ in the $C^2$ topology. 
\end{lemma}

\begin{proof}
	Because the domain $[0, \pi]$ is compact, the continuous functions $(Q - 2g\eta)$ and $\xi$ are strictly bounded away from their respective critical values. Therefore, conditions (i) and (ii) are robust to sufficiently small perturbations in the $C^0$ norm. 
	
	To verify that condition (iii) defines an open set, we examine the interior and the boundaries separately. By the strict curvature assumptions at the symmetry axes, there exist constants $k_0, k_1 > 0$ such that $\eta_{xx}(0,0) = -k_0$ and $\eta_{xx}(\pi,0) = k_1$. By the continuity of the second derivative, there exists a sufficiently small $\delta > 0$ such that $\eta_{xx} \leq -k_0/2$ on $[0, \delta]$ and $\eta_{xx} \geq k_1/2$ on $[\pi-\delta, \pi]$. 
	
	On the remaining compact interior interval $[\delta, \pi-\delta]$, the strict monotonicity assumption guarantees the existence of a constant $c > 0$ such that $\eta_x \leq -c$. 
	
	Consider a perturbation $\tilde{\eta} \in X$ such that $\|\tilde{\eta}\|_{C^2} < \min\{c/2, k_0/2, k_1/2\}$. 
	On the interior interval $[\delta, \pi-\delta]$, the perturbed slope satisfies:
	\begin{equation}
		(\eta + \tilde{\eta})_x \leq -c + \frac{c}{2} < 0.
	\end{equation}
	On the boundary neighborhood $[0, \delta]$, the perturbed curvature satisfies $(\eta + \tilde{\eta})_{xx} \leq -k_0/2 + k_0/2 \leq 0$. Since the perturbed profile remains even, its slope at the crest is identically zero: $(\eta + \tilde{\eta})_x(0) = 0$. Because the slope starts at zero and is strictly decreasing on $(0, \delta]$, we have $(\eta + \tilde{\eta})_x < 0$ on $(0, \delta]$. 
	
	An identical integration argument applies to the trough neighborhood $[\pi-\delta, \pi]$, guaranteeing the slope remains strictly negative. Thus, the perturbed profile strictly satisfies condition (iii), confirming that $\mathcal{U}$ is open in the $C^2$ topology.
\end{proof}

\subsection{Global preservation}
By combining the "closedness" property provided by the Maximum Principle (Theorem \ref{thm:hw_conformal_monotonicity}) with the "openness" of the admissible set (Lemma \ref{lem:openness}), we construct a standard topological connectedness argument. We prove that if a continuous branch begins with strictly monotonic waves, it is mathematically impossible for the solutions to lose this strict geometric structure unless the branch first hits a stagnation or self-intersection limit. This corollary serves as the formal mathematical license for our rigorous numerical algorithm to completely ignore the verification of monotonicity for high-amplitude waves.
\begin{corollary}[Global Preservation of Monotonicity]\label{cor:global_monotonicity}
	Let $\mathscr{C}$ be a continuous branch of periodic water wave solutions bifurcating from small-amplitude, strictly monotonic waves. As long as the solutions along $\mathscr{C}$ strictly satisfy the non-stagnation condition ($Q - 2g\eta > 0$ on $y=0$) and the injectivity condition ($\xi(x,0) < \pi$ for $x \in (0,\pi)$), every wave profile on $\mathscr{C}$ remains strictly monotonic on the half-period. 
\end{corollary}

\begin{proof}
	Let $\mathcal{U} \subset \mathscr{C}$ denote the subset of solutions on the continuation branch whose profiles are strictly monotonic. By the asymptotic properties of the local bifurcation curve, $\mathcal{U}$ is non-empty. 
	
	By Lemma \ref{lem:openness}, the condition of strict monotonicity (including the strict corner curvatures) defines an open set in the $C^2$ topology. Therefore, $\mathcal{U}$ is relatively open in $\mathscr{C}$. 
	
	Conversely, suppose a sequence of strictly monotonic solutions in $\mathcal{U}$ converges in $C^2$ to a limiting wave on $\mathscr{C}$. By hypothesis, this limiting wave remains strictly within the bounds of non-stagnation and non-self-intersection. Because the limit of monotonic functions is monotonic ($\eta_x \leq 0$), we may apply Theorem \ref{thm:hw_conformal_monotonicity} to the limiting wave. The theorem analytically forces the limit to be strictly monotonic. Therefore, $\mathcal{U}$ contains all of its limit points within the admissible bounds, meaning $\mathcal{U}$ is relatively closed in $\mathscr{C}$.
	
	Because the continuous branch $\mathscr{C}$ is a connected topological space, the only non-empty subset that is both relatively open and relatively closed is the entire space itself. Thus, $\mathcal{U} = \mathscr{C}$, proving that strict monotonicity cannot be lost anywhere along the valid branch.
\end{proof}

\section{The Fourier-Analytic Formulation}\label{sec : single wave and fourier analysis}

\subsection{Fourier reduction}
We begin by establishing the functional analytic framework required for the rigorous numerical continuation. Our immediate goal is to reduce the full 2D boundary value problem \eqref{eq : equations for zeta and eta} to a 1D system of algebraic equations defined solely on the free surface $y=0$. We achieve this by passing to the Fourier domain in the horizontal conformal variable and analytically solving the boundary value problem in the vertical direction.

To strictly enforce the harmonicity requirements \eqref{eq:pde_zeta}--\eqref{eq:pde_eta} and the flat bottom boundary conditions \eqref{eq:pde_bottom}, we expand the conjugate functions $\eta$ and $\zeta$ using separation of variables:
\begin{equation}\label{eq:fourier_series_eta_and_zeta}
	\eta(x,y) = \sum_{n \in \mathbb{Z}} a_n e^{inx} \psi_n(y) \qquad \text{and} \qquad \zeta(x,y) = \sum_{n \in \mathbb{Z}} b_n e^{inx} \psi_n(y),
\end{equation}
where the vertical basis functions are given explicitly by
\begin{equation}
	\psi_n(y) := \begin{cases}
		\displaystyle \frac{y+h}{h} & \text{if } n = 0, \\[8pt]
		\displaystyle \frac{\sinh(|n|(y+h))}{\sinh(|n|h)} & \text{if } |n| \geq 1.
	\end{cases}
\end{equation}

By construction, these basis functions automatically satisfy the interior Laplace equations and vanish at the conformal bottom $y=-h$. Furthermore, because $\psi_n(0)=1$ for all $n \in \mathbb{Z}$, the unknowns $a_n$ and $b_n$ correspond exactly to the Fourier coefficients of the boundary traces $\eta(x,0)$ and $\zeta(x,0)$. This effectively collapses the spatial domain, allowing us to reconstruct the entire 2D fluid state from the 1D surface sequences $a = \{a_n\}$ and $b = \{b_n\}$.

The zero-mode coefficients $a_0$ and $b_0$ carry direct physical interpretations. The coefficient $a_0$ captures the mean vertical elevation of the free surface, which is rigidly fixed by the gauge constraint \eqref{eq:gauge_block} such that $a_0 = h$. Physically, the zero-mode coefficient $b_0$ captures the uniform background current relative to the moving frame. By introducing a linear vertical profile to the stream function, $b_0$ establishes the constant horizontal base flow upon which the wave kinematics and the constant vorticity $\gamma$ are superimposed.

The primary algebraic challenge in our computer-assisted proof is evaluating the nonlinear dynamic boundary condition \eqref{eq:pde_dynamic} in Fourier space. To ensure this pointwise evaluation is well-defined, we seek our profiles in the weighted Banach algebra $X_e$ introduced in Section \ref{sec:notation_spaces} (where we will establish in Lemma \ref{lem : solutions are C infinity} that solutions found within this space actually exhibit strong decay in their Fourier tails). 

While $X_e$ captures our strictly even boundary profiles, evaluating the spatial derivatives (such as $\eta_x$) naturally maps these profiles into the complimentary space of strictly odd sequences. We therefore define the corresponding closed odd subspace:
\begin{equation}\label{def : banach space X odd}
X_o \bydef \left\{ a \in \ell^1 : a_n \in i\mathbb{R}, \, a_{-n} = -a_n \text{ for all } n \in \mathbb{Z}, \text{ and } \|a\|_X < \infty \right\}. 
\end{equation}

To rigorously evaluate the nonlinear terms in the boundary conditions (such as the squared gradients), the sequence space $X$ must form a Banach algebra under discrete convolution. This property ensures that the pointwise product of two functions in our space remains within the space, thereby preserving the required regularity. We formally verify this in the following lemma. 

\begin{lemma}[Banach Algebra Property]\label{lem:Banach_algebra}
    The space $(X,*)$ is a Banach algebra under the discrete convolution operator, defined for $a,b \in X$ by:
    \begin{align*}
        (a*b)_k \bydef \sum_{n \in \mathbb{Z}} a_{k-n}b_n \qquad \text{for all } k \in \mathbb{Z}.
    \end{align*}
    Furthermore, the even and odd closed subspaces satisfy the standard parity convolution relations: $X_e * X_e \subset X_e$, $X_o * X_e \subset X_o$, and $X_o * X_o \subset X_e$.
\end{lemma}

\begin{proof}
    Let $a,b \in X$. Taking the norm of the convolution yields:
    \begin{align*}
        \|a*b\|_X &= \sum_{k \in \mathbb{Z}} (1+|k|) \left| \sum_{n \in \mathbb{Z}} a_{k-n}b_{n} \right| \\
        &\leq \sum_{k \in \mathbb{Z}} \sum_{n \in \mathbb{Z}} \frac{1+|k|}{(1+|n|)(1+|k-n|)} (1+|k-n|)|a_{k-n}| (1+|n|)|b_{n}|.
    \end{align*}
    Observe that by the standard triangle inequality $|k| = |n + (k-n)| \leq |n| + |k-n|$, we can bound the fractional weight:
    \begin{align*}
        1+|k| \leq 1+|n|+|k-n| \leq 1+|n|+|k-n|+|n||k-n| = (1+|n|)(1+|k-n|).
    \end{align*}
    Therefore, the fractional multiplier is bounded above by $1$ for all $k,n \in \mathbb{Z}$. Applying this bound and utilizing the change of variables $m = k-n$, we can cleanly decouple the sums:
    \begin{align*}
        \|a*b\|_X &\leq \sum_{k \in \mathbb{Z}} \sum_{n \in \mathbb{Z}} (1+|k-n|)|a_{k-n}| (1+|n|)|b_{n}| \\
        &= \left( \sum_{m \in \mathbb{Z}} (1+|m|)|a_{m}| \right) \left( \sum_{n \in \mathbb{Z}} (1+|n|)|b_{n}| \right) \\
        &= \|a\|_X \|b\|_X.
    \end{align*}
    This establishes the sub-multiplicative inequality $\|a*b\|_X \leq \|a\|_X \|b\|_X$, confirming that $X$ is a Banach algebra. The parity subspace relations follow immediately from the linear combinations of standard even and odd symmetries inside the convolution sum.
\end{proof}

\subsection{Operators representation in Fourier coefficients}

To evaluate the nonlinear boundary conditions—particularly the squared gradient terms in the dynamic condition \eqref{eq:pde_dynamic}—we must compute the spatial derivatives of the fluid variables exactly on the free surface $y=0$. By leveraging the analytic extension provided by \eqref{eq:fourier_series_eta_and_zeta}, we can capture the nonlocal fluid interactions within the bulk of the domain $\Omega$ through exact, purely algebraic operations in the sequence space $X$.

Differentiating the series expansion \eqref{eq:fourier_series_eta_and_zeta} term-by-term with respect to $y$ and evaluating at the surface yields:
\begin{align*}
    \partial_y\eta(x,0) = \frac{a_0}{h} +  \sum_{n \in \mathbb{Z}\setminus \{0\}} a_n e^{inx} |n| \frac{\cosh(|n|h)}{\sinh(|n|h)} = \frac{a_0}{h} + \sum_{n \in \mathbb{Z}\setminus \{0\}} a_n e^{inx} |n| \coth(|n|h).
\end{align*}
 Therefore, we define the exact multiplier sequence $(\omega_n)_{n \in \mathbb{Z}}$ as:
\begin{align}\label{def : omega k}
    \omega_n \bydef \begin{cases}
        \displaystyle \frac{1}{h} & \text{if } n=0, \\[8pt]
        |n| \coth(|n|h) & \text{if } n \neq 0.
    \end{cases} 
\end{align}
This sequence acts as the symbol for the linear operator $D_y$, defined element-wise on the Fourier coefficients by:
\begin{align}\label{eq: Dy definition}
    (D_y a)_n \bydef \omega_n a_n \qquad \text{for all } n \in \mathbb{Z}.
\end{align}
The operator $D_y$ is exactly the periodic Dirichlet-to-Neumann (DtN) map for a flat strip of depth $h$. While the DtN map is a complex, nonlocal pseudo-differential operator in the physical domain, our choice of conformal coordinates diagonalizes it, reducing the boundary derivative to a local Fourier multiplier.

Similarly, differentiating \eqref{eq:fourier_series_eta_and_zeta} with respect to $x$ at $y=0$ yields the standard horizontal derivative operator $D_x$, defined by the purely imaginary symbol:
\begin{align}\label{eq: Dx definition}
    (D_x a)_n \bydef in a_n \qquad \text{for all } n \in \mathbb{Z}.
\end{align}

Because the symbols of both $D_x$ and $D_y$ grow linearly with $|n|$ as $|n| \to \infty$, these are unbounded linear operators on $X$, reflecting the standard loss of one derivative associated with spatial differentiation. However, within our spectral computational framework, defining these exact algebraic multipliers allows us to evaluate the full boundary conditions on $\Gamma$ without any spatial discretization error.

\subsection{The Zero-finding Problem}\label{sec : zero_finding_problem}

We are now positioned to reformulate the physical boundary value problem as an algebraic zero-finding problem suitable for rigorous numerical continuation. By applying the Fourier multiplier operators $D_x$ and $D_y$ and utilizing the discrete convolution product, we map the sequence of free surface Fourier coefficients to a sequence of residuals quantifying the pointwise violation of the boundary conditions. Finding a valid physical water wave is thereby transformed into finding the zeros of a nonlinear operator $F(U) = 0$.

Evaluating the kinematic and dynamic boundary conditions \eqref{eq:pde_kinematic} and \eqref{eq:pde_dynamic} in the sequence space $X$, and representing all pointwise multiplications via the discrete convolution operator $*$, we obtain:
\begin{align}
    b &= \frac{\gamma}{2} (a * a) + m e_0, \label{eq:b_fourier} \\
    (D_y b - \gamma a * D_y a)^{*2} &- (Q - 2ga) * \left( (D_x a)^{*2} + (D_y a)^{*2} \right) = 0, \label{eq:bernoulli_fourier}
\end{align}
where $e_0 \in X$ is the standard identity sequence defined by $1$ for $n=0$ and $0$ otherwise, and the notation $(\cdot)^{*2}$ denotes the convolution square (e.g., $a^{*2} = a * a$). By substituting \eqref{eq:b_fourier} into \eqref{eq:bernoulli_fourier} to algebraically eliminate the stream function coefficients $b$, we derive a single governing equation $f(a) = 0$ for the surface elevation, where:
\begin{align}\label{def : f_of_a}
    f(a) \bydef \left(m D_y e_0 + \frac{\gamma}{2} D_y(a^{*2}) - \gamma a * D_y a\right)^{*2} - (Q - 2ga) * \left( (D_x a)^{*2} + (D_y a)^{*2} \right). 
\end{align}

To strictly enforce symmetric wave profiles, we restrict the domain of $f$ to the even subspace $X_e$. Because the multiplier operators $D_x$ and $D_y$ consume one spatial derivative (their symbols scale linearly with $|n|$), they act as bounded linear maps from $X_e$ into the standard sequence space $\ell^1_e$. Since $\ell^1_e$ is itself a Banach algebra under convolution, $f$ is a well-defined, bounded map from $X_e$ to $\ell^1_e$. This rigorous mapping between functional spaces is critical for the implementation of the Newton-Kantorovich theorem, presented in Theorem \ref{th : radii polynomial}.

Following Constantin et al.\ \cite{constantin_global}, we must also enforce the mean water depth constraint \eqref{eq:gauge_block}. In Fourier space, this integral condition drastically simplifies to the single algebraic constraint $a_0 - h = 0$. Because this introduces an additional scalar equation to our root-finding problem, we elevate the Bernoulli constant $Q$ to an unknown variable to ensure the overall system remains square and well-posed, while the physical parameters $h, \gamma,$ and $m$ remain fixed.

To solve the system simultaneously, we define the augmented state vector $U = (Q, a)$ and consider the nonlinear operator $F : H_1 \to H_0$ defined compactly by
\begin{equation}\label{def : zero finding F}
    F(U) = \begin{pmatrix}
    a_0 - h \\
    f(Q, a)
    \end{pmatrix},
\end{equation}
where $f(Q, a)$ is the exact Bernoulli boundary residual defined in \eqref{def : f_of_a}.

To facilitate our analysis, we introduce the scaling operator $\mathcal{W} : H_1 \to H_0$ defined by
\begin{align*}
    \mathcal{W}(Q,a) \bydef (Q, \hat{\mathcal{W}}a), \qquad \text{where } \quad (\hat{\mathcal{W}}a)_n \bydef (1+|n|)a_n \quad \text{for all } n \in \mathbb{Z}.
\end{align*}
By definition of the sequence spaces, this establishes the norm equivalences
\begin{align}\label{def : W and hat W}
    \|U\|_{H_1} = \|\mathcal{W}U\|_{H_0} \qquad \text{and} \qquad \|a\|_X = \|\hat{\mathcal{W}}a\|_{\ell^1} \qquad \text{for all } U \in H_1, \, a \in X_e.
\end{align}
In functional analytic terms, $\mathcal{W} : H_1 \to H_0$ and $\hat{\mathcal{W}} : X_e \to \ell^1_e$ act as isometric isomorphisms. These exact isometries will be utilized to balance the function spaces during the construction of the Fréchet derivatives in the subsequent Newton-Kantorovich proofs.

\subsection{Geometric Reconstruction and the Overhanging Criterion}

Once a rigorous sequence of Fourier coefficients $a \in X_e$ has been computed, we can explicitly reconstruct the physical coordinates of the free surface $\mathcal{S}$. Recall that the free surface is parameterized by the horizontal conformal variable $x \in [-\pi, \pi]$, such that the physical coordinates are given by the mapping $x \mapsto (\xi(x,0), \eta(x,0))$.

By evaluating the vertical conformal mapping $\eta(x,y)$ from \eqref{eq:fourier_series_eta_and_zeta} at the surface $y=0$, and exploiting the even symmetry of the space $X_e$, the vertical profile reduces to the real cosine series:
\begin{equation}\label{eq : geometric eta}
\eta(x,0) = a_0 + 2\sum_{n=1}^{\infty} a_n \cos(nx).
\end{equation}

To recover the horizontal profile $\xi(x,0)$, we utilize the Cauchy-Riemann equations evaluated at the boundary, which yield $\xi_x(x,0) = \eta_y(x,0)$. By the definition of the normal derivative operator $D_y$ in \eqref{eq: Dy definition}, this is identically the Fourier series representation of $D_y a$:
\begin{equation*}
\xi_x(x,0) = (D_y a)(x) = \sum_{n \in \mathbb{Z}} \omega_n a_n e^{inx}.
\end{equation*}
For the zero-mode ($n=0$), our multiplier is defined as $\omega_0 = \frac{1}{h}$, and the zero-finding problem \eqref{def : zero finding F} rigorously enforces the mean-depth constraint $a_0 = h$. Consequently, the mean term of the derivative evaluates to exactly $1$. Expanding the remaining terms yields the strictly real expression:
\begin{equation}\label{eq : geometric xi derivative}
\xi_x(x,0) = 1 + 2\sum_{n=1}^{\infty} n \coth(nh) a_n \cos(nx).
\end{equation}
Integrating \eqref{eq : geometric xi derivative} with respect to $x$ provides the parametric horizontal coordinate of the free surface:
\begin{equation}\label{eq : geometric xi}
\xi(x,0) = x + 2\sum_{n=1}^{\infty} a_n \coth(nh) \sin(nx).
\end{equation}

With the physical coordinates fully reconstructed, we can now formally establish the geometric criterion — introduced in Section~\ref{sec:notation_spaces} — required to rigorously identify an overhanging wave. A standard, non-overhanging water wave satisfies the condition that the free surface can be represented as a single-valued graph over the horizontal axis (i.e., $Y = f(X)$). This requires the horizontal spatial coordinate $\xi(x,0)$ to be strictly monotonically increasing with respect to the conformal variable $x$. Conversely, if this monotonicity is broken, the wave profile physically folds backward upon itself.

\begin{lemma}[Overhanging Criterion]\label{lem : overhanging criterion}
	Let $U = (Q, a) \in H_1$ be a solution to the zero-finding problem $F(U) = 0$ given in \eqref{def : zero finding F}, and let $\eta$ be given as in \eqref{eq:fourier_series_eta_and_zeta}. The physical free surface of the water wave is strictly overhanging if there exists a conformal coordinate $x^* \in [-\pi, \pi]$ such that:
	\begin{equation*}
	\eta_y(x^*,0) < 0.
	\end{equation*}
\end{lemma}

\begin{proof}
	By the Cauchy-Riemann equivalence $\xi_x(x,0) = \eta_y(x,0)$, the condition $\eta_y(x^*,0) < 0$ implies that $\xi_x(x^*, 0) < 0$. Therefore, the horizontal mapping $x \mapsto \xi(x,0)$ is not globally monotonically increasing, and the parametric curve $x \mapsto (\xi(x,0), \eta(x,0))$ cannot be expressed as a single-valued function of the horizontal coordinate. This geometric intersection of vertical cross-sections confirms the presence of an overhang.
\end{proof}

\section{Computer-assisted Bounding Framework}\label{sec: CA bounding}

\subsection{A Newton-Kantorovich Approach}\label{ssec : Newton-K single wave}

Our ultimate goal is to rigorously prove the existence of steady periodic water wave solutions to \eqref{eq : equations for zeta and eta}. In our functional analytic framework, this is equivalent to finding a true root $U \in H_1$ such that $F(U) = 0$. Our strategy relies on computing a finite-dimensional numerical approximation $\bar U = (\bar{Q},\bar{a}) \in H_1$ (satisfying the gauge constraint $\bar{a}_0 = h$ exactly, such that $F(\bar{U}) \approx 0$) and subsequently applying a Newton-Kantorovich type theorem. This allows us to rigorously prove that a true, infinite-dimensional solution exists within an explicitly calculable, small neighborhood of $\bar U$.

To transition from the continuous PDE theory to the computer-assisted proof setting, we must bridge the gap between the infinite-dimensional sequence space $X$ and the finite-dimensional arrays manipulable by a computer. We therefore introduce projection operators that decompose any sequence into a finite Galerkin projection (which the computer processes) and an infinite tail (which is handled analytically).

For a given computational truncation dimension $N \in \mathbb{N}$, we define the finite projection operator $\hat{\pi}^{\leq N} : X \to X$ as:
\begin{align*}
\left(\hat{\pi}^{\leq N} a \right)_n \bydef \begin{cases}
a_n & \text{if } |n| \leq N, \\
0 & \text{otherwise}.
\end{cases}
\end{align*}
This operator isolates the low-frequency Fourier modes, effectively defining the finite-dimensional subspace where our numerical root-finding algorithm operates.

Conversely, we define the infinite tail projection $\hat{\pi}^{>N} : X \to X$, which captures the analytic remainder, as:
\begin{align*}
\left(\hat{\pi}^{>N} a \right)_n \bydef \begin{cases}
0 & \text{if } |n| \leq N, \\
a_n & \text{otherwise}.
\end{cases}
\end{align*}

Because subsequent analytic bounds on the convolution of these infinite tails require distinguishing between the positive and negative high-frequency spectra, we further decompose the tail operator into $\hat{\pi}^{n>N} : X \to X$ and $\hat{\pi}^{n<-N} : X \to X$, defined respectively by:
\begin{align*}
\left(\hat{\pi}^{n > N} a \right)_n \bydef \begin{cases}
a_n & \text{if } n > N, \\
0 & \text{otherwise},
\end{cases}
\qquad \text{and} \qquad
\left(\hat{\pi}^{n < -N} a \right)_n \bydef \begin{cases}
a_n & \text{if } n < -N, \\
0 & \text{otherwise}.
\end{cases}
\end{align*}

We naturally extend these sequence projections to the augmented state vector $U = (Q, a)$. Because the Bernoulli constant $Q$ is a finite scalar parameter, it is entirely encapsulated within the finite-dimensional subspace. Thus, we define the full augmented projections $\pi^{\leq N}$ and $\pi^{>N}$ on the product spaces (either $H_1$ or $H_0$) by:
\begin{align*}
\pi^{\leq N} U \bydef \begin{pmatrix} Q \\ \hat{\pi}^{\leq N} a \end{pmatrix} 
\qquad \text{and} \qquad 
\pi^{> N} U \bydef \begin{pmatrix} 0 \\ \hat{\pi}^{> N} a \end{pmatrix}.
\end{align*}
The directional high-frequency projections $\pi^{n > N}$ and $\pi^{n < -N}$ act analogously, mapping the scalar $Q$-component identically to zero.

Let $K_0 \in \mathbb{N}$ denote the computational truncation dimension for our numerical candidate $\bar{U} = (\bar{Q}, \bar{a})$. By construction, the numerical approximation is strictly supported on the finite subset of Fourier frequencies $|n| \leq K_0$. Therefore, $\bar{U}$ satisfies:
\begin{equation}\label{eq : approx sol finite dim}
\bar{U} = \pi^{\leq K_0}\bar{U}, \qquad \text{with the exact gauge constraint} \qquad \bar{a}_0 = h.
\end{equation}

To bridge the gap between this finite-dimensional numerical approximation and a true infinite-dimensional solution, we employ a Newton-Kantorovich approach. Let $A : H_0 \to H_1$ be an injective bounded linear operator strategically constructed to act as an approximate inverse to the Fréchet derivative $DF(\bar{U})$. We seek to determine a rigorous isolation radius $r > 0$ such that the Newton-like operator:
\begin{equation}\label{def : fixed point op T}
T(U) \bydef U - A F(U)
\end{equation}
is a well-defined contraction mapping from the closed topological ball $B_r(\bar{U}) \subset H_1$ into itself. 

By the Banach Fixed Point Theorem, if $T$ maps $B_r(\bar{U})$ into itself and is strictly contractive, then it possesses a unique fixed point $U^*$ within that ball. Furthermore, because the approximate inverse $A$ is injective, the fixed-point condition $T(U^*) = U^*$ directly implies $F(U^*) = 0$, rigorously establishing the existence of a true periodic water wave. To computationally verify these functional analytic conditions, we employ the standard Newton-Kantorovich framework (see, e.g., \cite{ArKoTe05, DaLeMi07, NakPluWat19}), which elegantly reduces the infinite-dimensional topological requirements to a finite set of inequalities.

\begin{theorem}[The Newton-Kantorovich Theorem]\label{th : radii polynomial}
	Let $(X_1,\|\cdot\|_{X_1})$ and $(X_2,\|\cdot\|_{X_2})$ be Banach spaces, and let $F : X_1 \to X_2$ be Fréchet differentiable. Moreover, let $\bar{u} \in X_1$ be an approximate solution, and let ${A} : X_2 \to X_1$ be an injective bounded linear operator. 
	
	Assume there exist non-negative constants $Y$ and $Z_1$, and a non-negative monotonically non-decreasing function $Z_2 : (0, \infty) \to [0,\infty)$, such that:
	\begin{equation}\label{eq : bounds general radii polynomial}
	\begin{aligned}
	\|{A}{F}(\bar{u})\|_{X_1} & \le {Y}, \\
	\|I - {A}D{F}(\bar{u})\|_{\mathcal{B}(X_1)} &\le {Z}_1,\\
	\|{A}\left({D}{F}(\bar{u} + h) - D{F}(\bar{u})\right)\|_{\mathcal{B}(X_1)} &\le {Z}_2(r)r, \qquad \text{for all } h \in {B_r(0)} \subset X_1 \text{ and all } r>0.
	\end{aligned}  
	\end{equation}
	If there exists a radius $r>0$ such that the radii polynomial conditions are satisfied:
	\begin{equation}\label{condition radii polynomial}
	\frac{1}{2}{Z}_2(r)r^2 - (1-{Z}_1)r + {Y} < 0 \qquad \text{and} \qquad {Z}_1 + {Z}_2(r)r < 1,
	\end{equation}
	then there exists a unique $\tilde{u} \in {B_r}(\bar{u}) \subset X_1$ such that ${F}(\tilde{u})=0$. 
\end{theorem}

This theorem provides the rigorous functional analytic mechanism to elevate our numerical candidate $\bar{u}$ to an exact, true solution. The proof strategy relies on computing three distinct analytical bounds that characterize the local geometry of the operator $T(u) = u - AF(u)$. 

First, the residual bound $Y$ provides a quantitative measure of the quality of our initial numerical approximation by evaluating how far $\bar{u}$ is from being a true solution. Second, the linear bound $Z_1$ evaluates the quality of $A$ as a quasi-inverse of the Fréchet derivative. By strictly requiring $Z_1 < 1$ (implied by \eqref{condition radii polynomial}), we ensure that the Jacobian is invertible at $\bar{u}$ and that the system is well-conditioned. Finally, the nonlinear curvature bound $Z_2(r)$ captures the local Lipschitz behavior of the derivative, measuring how much $DF$ deviates as one perturbs the state away from the numerical candidate.

These components are synthesized into the so-called radii polynomial conditions \eqref{condition radii polynomial}. The existence of an exact solution is mathematically guaranteed for any radius $r$ where this bounding curve evaluates to a strictly negative number, proving that $T$ maps the closed ball into itself. Simultaneously, the second strict inequality in \eqref{condition radii polynomial} uniformly bounds the Lipschitz constant of $T$ below 1, triggering the Banach Fixed Point Theorem to guarantee that the solution is strictly unique within that neighborhood.

Our objective for the remainder of this section is to explicitly construct the approximate inverse $A$ for our water wave system, and to derive rigorous, computable formulas for the bounds $Y, Z_1,$ and $Z_2(r)$.

\subsection{Construction of the Approximate Inverse}\label{ssec : approximate inverse single wave}

Given the numerical approximation $\bar{U} = (\bar{Q}, \bar{a}) \in H_1$ satisfying \eqref{eq : approx sol finite dim}, our objective is to explicitly construct an injective bounded linear operator $A : H_0 \to H_1$ that acts as an approximate inverse to the Fréchet derivative $DF(\bar{U})$.

Given $a \in X$, we introduce the discrete convolution operator $M_a$ as 
\begin{align}
M_a b \bydef a * b \qquad \text{for all } b \in X. 
\end{align}
By the Banach algebra property established in Lemma \ref{lem:Banach_algebra}, $M_a$ is a bounded linear operator whose induced operator norm satisfies the sub-multiplicative bound
\begin{align}\label{eq : multiplication equality}
\|M_a\|_{\mathcal{B}(X)} = \|a\|_X.
\end{align} 
Representing nonlinear pointwise products as linear mappings on $X$ is the foundational prerequisite for systematically deriving the Jacobian of our system.

Taking the directional Fréchet derivative of $f(Q, a)$ with respect to the surface profile $a$, evaluated at $\bar{U}$, we can decompose the operator $D_a f(\bar{U})$ into four primary linear operators that isolate the geometric and physical interactions of the fluid variables:
\begin{align}\label{def : DF}
D_a f(\bar{U}) = M_{v_1} \mathcal{C}[{\bar{a}}] + M_{v_2} + M_{v_3} D_x + M_{v_4} D_y.
\end{align}
Here, $\mathcal{C}[\bar{a}] : X \to X$ isolates the nonlinear Dirichlet-to-Neumann commutator structure, defined by
\begin{equation}
\mathcal{C}[\bar{a}] \bydef D_y M_{\bar{a}} - M_{\bar{a}} D_y - M_{D_y \bar{a}},
\end{equation}
and the multiplier sequences $v_1, v_2, v_3, v_4 \in X$ are evaluated directly from the steady state variables
\begin{equation}\label{eq : functions v1 to v4 and C}
\begin{aligned}
v_1 &\bydef 2\gamma \left(m D_y e_0 + \frac{\gamma}{2} D_y(\bar{a}^{*2}) - \gamma \bar{a} * D_y\bar{a}\right), \\
v_2 &\bydef 2g \left( (D_x\bar{a})^{*2} + (D_y\bar{a})^{*2} \right), \\
v_3 &\bydef -2(\bar{Q} - 2g \bar{a}) * D_x \bar{a}, \\
v_4 &\bydef -2(\bar{Q} - 2g \bar{a}) * D_y \bar{a}.
\end{aligned}
\end{equation}

The first term, $M_{v_1} \mathcal{C}[\bar{a}]$, captures the linearized influence of vorticity on the dynamic pressure balance at the free surface. The operator $\mathcal{C}[\bar{a}]$ acts as a commutator that reconciles the vertical Dirichlet-to-Neumann operator $D_y$ with the surface multiplication operator $M_{\bar{a}}$. Physically, this commutator accounts for the conformal distortion of the underlying background shear flow. When the wave profile is perturbed, the rotational current is locally compressed or stretched beneath the crests and troughs. The expression $D_y M_{\bar{a}} - M_{\bar{a}} D_y$ precisely captures this geometric deformation of the stream function, while the correction term $M_{D_y \bar{a}}$ accounts for the local rotational tilt of the fluid particles exactly at the boundary.

The multiplier sequence $v_1$ corresponds to the relative tangential velocity of the fluid at the surface. Where this relative velocity is large—indicating the water is moving rapidly relative to the wave frame—even infinitesimal perturbations to the wave's shape induce pronounced dynamic pressure variations, as the fast-moving shear layer is forcefully redirected by the altered surface slope.

The remaining terms govern the classical energy balance of the wave. The bounded multiplier $M_{v_2}$ isolates the variation in potential energy, effectively representing the hydrostatic restoring force of gravity acting on the perturbed surface. Finally, the unbounded operators $M_{v_3} D_x$ and $M_{v_4} D_y$ capture the kinetic energy variations dictated by the perturbed horizontal and vertical velocity gradients, respectively.

Since our numerical approximation $\bar{a}$, satisfying \eqref{eq : approx sol finite dim}, is a finite trigonometric polynomial, the multiplier sequences $v_i$ are also finite trigonometric polynomials. Consequently, the associated pointwise multiplication operators $M_{v_i}$ are strictly bounded. It remains to prove that the nonlinear commutator $\mathcal{C}[\bar{a}]$ is also a bounded linear operator on $X$, a non-trivial requirement given the presence of the unbounded spatial derivative $D_y$.

\begin{lemma}\label{lem : operator C}
	Let the diagonal operator $\mathcal{D} : X \to \ell^1$ be defined by the multiplier sequence:
	\begin{align*}
	(\mathcal{D}a)_k = \mathcal{D}_k a_k \qquad \text{where} \qquad \mathcal{D}_k \bydef |k| - \omega_k,
	\end{align*}
	and $\omega_k$ is the Dirichlet-to-Neumann symbol given in \eqref{def : omega k}.
	If the tail truncation threshold satisfies $K_1 \geq 2K_0$, then we have the rigorous bound:
	\begin{align*}
	\|\mathcal{C}[\bar{a}]\hat{\pi}^{>K_1} - M_{\mathcal{D}\bar{a}}\hat{\pi}^{>K_1}\|_{\mathcal{B}(X)} \leq \frac{4K_0}{1-e^{-2hK_0}}e^{-2hK_0} \|\bar{a}\|_X.
	\end{align*}
	In particular, $\mathcal{C}[\bar{a}] : X \to X$ is a bounded linear operator.
\end{lemma}

\begin{remark}[Asymptotic Behavior and the Deep-Water Limit]
	This lemma demonstrates that the commutator $\mathcal{C}[\bar{a}]$ asymptotically behaves as a bounded multiplication operator in the high-frequency limit. The physical mechanism driving this is the deep-water limit of short waves. As the frequency $|k| \to \infty$, the wave scale becomes infinitesimally small compared to the fluid depth $h$, causing the symbol $\coth(|k|h)$ to approach $1$ exponentially fast. 
	
	Physically, this indicates that high-frequency surface ripples no longer interact with the rigid sea floor; the non-local Dirichlet-to-Neumann map $D_y$ collapses into the local horizontal derivative $|D_x|$. By carefully canceling these unbounded leading-order terms, the lemma proves that the remaining geometric deformation of the background shear flow is completely localized and bounded.
\end{remark}

\begin{remark}[The Truncation Threshold $K_1$]
	The choice of the secondary truncation threshold $K_1 \geq 2K_0$ is critical for isolating this asymptotic behavior. Because the numerical approximation $\bar{a}$ has a maximum frequency support of $K_0$, discrete convolution with the high-frequency tail (modes where $|k| > K_1$) shifts frequencies by at most $K_0$. 
	
	Setting $K_1 \geq 2K_0$ establishes a strict spectral "buffer zone", guaranteeing that nonlinear interactions involving the infinite-dimensional tail can only produce frequencies $|k| > K_0$. This completely prevents the asymptotic, infinitely-deep-water dynamics of the tail from artificially polluting the low-frequency core of the numerical solution, maintaining the strict validity of the analytic bounds.
\end{remark}

\begin{proof}
	Let $u \in X$. From the definition $\mathcal{C}[\bar{a}] = D_y M_{\bar{a}} - M_{\bar{a}} D_y - M_{D_y \bar{a}}$, the $k$-th Fourier coefficient of the residual operator is
	\begin{align*}
	\left( \mathcal{C}[\bar{a}]\hat{\pi}^{>K_1}u - M_{\mathcal{D}\bar{a}}\hat{\pi}^{>K_1}u \right)_k = \sum_{|n| > K_1} \left( \omega_k - \omega_n - \omega_{k-n} - \mathcal{D}_{k-n} \right) \bar{a}_{k-n} u_n.
	\end{align*}
	Since $\bar{a}_n = 0$ for $|n| > K_0$ (see\ \eqref{eq : approx sol finite dim}), non-zero terms require $|k-n| \leq K_0$. Consequently, $|k| = |n + (k-n)| \geq |n| - |k-n| \geq K_1 - K_0$. Splitting the sequence norm via the triangle inequality yields
	$$\|\mathcal{C}[\bar{a}]\hat{\pi}^{>K_1}u - M_{\mathcal{D}\bar{a}}\hat{\pi}^{>K_1}u\|_X \leq \Sigma_{\text{int}} + \Sigma_{\text{res}},$$ 
	which we bound in two steps.
	
	First, using the algebraic identity $||n| - \omega_n| \leq \frac{2K_0}{1-e^{-2hK_0}}e^{-2hK_0}$ for $|n| \geq K_0$, the finite-depth discrepancy is bounded by
	\begin{align*}
	\Sigma_{\text{res}} &\bydef \sum_{|k| \geq K_1-K_0} (1+|k|) \left| \sum_{|n| > K_1} \left( ||k|-\omega_k| + ||n|-\omega_n| \right) \bar{a}_{k-n} u_n \right| \\
	&\leq \frac{4K_0}{1-e^{-2hK_0}}e^{-2hK_0} \|\bar{a}\|_X \|u\|_X.
	\end{align*}
	
	Second, for the high-frequency interaction, the threshold $K_1 \geq 2K_0$ ensures that the condition $|n| > K_1$ coupled with $|k-n| \leq K_0$ implies $\text{sgn}(k) = \text{sgn}(n)$. This gives $|k| - |n| = \text{sgn}(n)(k-n)$. Substituting $\mathcal{D}_k = |k| - \omega_k$ results in identical cancellation
	\begin{align*}
	\Sigma_{\text{int}} &\bydef \sum_{|k| \geq K_1-K_0} (1+|k|) \left| \sum_{|n| > K_1} \left( |k| - |n| - \omega_{k-n} - \mathcal{D}_{k-n} \right) \bar{a}_{k-n} u_n \right| = 0.
	\end{align*}
	
	Since $\mathcal{C}[\bar{a}]\hat{\pi}^{\leq K_1}$ is a finite-rank (and thereby bounded) linear operator, and the remaining tail $\mathcal{C}[\bar{a}]\hat{\pi}^{>K_1}$ is bounded by the multiplication operator $M_{\mathcal{D}\bar{a}}$ plus the exponentially small residual $\Sigma_{\text{res}}$, we conclude that $\mathcal{C}[\bar{a}] \in \mathcal{B}(X)$.
\end{proof}

To facilitate the explicit construction of the approximate inverse, we introduce the asymptotic vertical derivative operator $\hat{D}_y$, defined element-wise by the diagonal multiplier
\begin{align*}
(\hat{D}_y a)_k \bydef \begin{cases}
\displaystyle \frac{1}{h} a_k & \text{if } k = 0, \\[8pt]
|k| a_k & \text{if } |k| \geq 1.
\end{cases}
\end{align*}
The operator $\hat{D}_y$ captures the high-frequency limit of the Dirichlet-to-Neumann map. Physically, it represents the vertical derivative for a fluid of infinite depth, whereas the exact operator $D_y$ retains the boundary effects of the rigid ocean floor. As $|k| \to \infty$, the exact symbol $\omega_k = |k| \coth(|k|h)$ and its asymptotic counterpart $|k|$ become exponentially close, with their discrepancy decaying at the rate $\mathcal{O}(e^{-2|k|h})$ (as quantified previously in Lemma \ref{lem : operator C}).

By artificially assigning the zero-mode symbol as $1/h$ (matching $\omega_0$), we ensure that $\hat{D}_y$ is boundedly invertible. This allows the high-frequency tail of the approximate inverse to be explicitly constructed using the simple diagonal sequence $1/|k|$. This strategic substitution streamlines the computation of the $Z_1$ linear bound, as the error introduced by swapping the exact finite-depth physics for the asymptotic infinite-depth physics is exponentially suppressed in the tail.

We now construct the analytical tail of the approximate inverse, denoted by $\hat{A}_\infty$. The subsequent lemma establishes that in the high-frequency regime, the linearized operator is strictly dominated by its principal part, enabling us to bound the asymptotic inversion error by a contraction constant $\delta < 1$. 

Analytically, $\hat{A}_\infty$ is constructed by preconditioning the bounded inverse operator $\hat{D}_y^{-1}$ with sequence multiplication operators $M_w$ and $M_{w^*}$. These weights are explicitly designed to neutralize the kinetic energy coefficients $v_3$ and $v_4$ appearing in the Fréchet derivative of Bernoulli's equation. 

The contraction constant $\delta$ quantitatively bounds the defect of $\hat{A}_\infty$ as a true right inverse. It comprises two distinct error sources: the residual $\|w * (iv_3 + v_4) - e_0\|_{\ell^1}$, which measures the algebraic imperfection of the chosen weight $w$ in exactly canceling the physical velocity coefficients, and an exponentially decaying $\mathcal{O}(e^{-2hK_1})$ term, which captures the finite-depth truncation error (the geometric discrepancy between $D_y$ and $\hat{D}_y$).

\begin{lemma}[Asymptotic Approximate Inverse]\label{lem : A infinity properties}
    Let the truncation threshold satisfy $K_1 \geq 2K_0$. Let $w \in X$ be a finite  sequence such that $w = \hat{\pi}^{\leq K_0}w$, and let $w^*$ denote its complex conjugate. We define the high-frequency approximate inverse $\hat{A}_\infty : \ell^1 \to X$ as
    \begin{equation}\label{def : A infinity}
        \hat{A}_\infty \bydef \hat{\pi}^{k < -K_1} \hat{D}_y^{-1} M_{w^*} + \hat{\pi}^{k > K_1} \hat{D}_y^{-1} M_w.
    \end{equation}
    Furthermore, we define the scalar constant $\delta > 0$ by
    \begin{equation}
        \delta \bydef \frac{K_1+2}{K_1+1} \left( \|w * (i v_3 + v_4) - e_0\|_{\ell^1} + \|w * v_4\|_{\ell^1} \frac{2e^{-2hK_1}}{1-e^{-2hK_1}} \right).
    \end{equation}
    Then, the following operator norm bound holds on the infinite tail
    \begin{equation}
        \left\| \hat{A}_\infty \left( M_{v_3} D_x + M_{v_4} D_y \right) \hat{\pi}^{>K_1} - \hat{\pi}^{>K_1} \right\|_{\mathcal{B}(X)} \leq \delta.
    \end{equation}
    In particular, if $\delta < 1$, then by a standard Neumann series argument, the principal linear operator $(M_{v_3} D_x + M_{v_4} D_y)\hat{\pi}^{>K_1} : \hat{\pi}^{>K_1}X \to \ell^1$ is bounded below and possesses a bounded left inverse.
\end{lemma}

\begin{proof}
Let $B \bydef \left(M_{v_3} D_x + M_{v_4} \hat{D}_y\right) \hat{\pi}^{k>K_1}$. Because $D_x$ and $\hat{D}_y$ act diagonally on the positive tail, we observe the exact identity $D_x\hat{\pi}^{k>K_1} = i\hat{D}_y\hat{\pi}^{k>K_1}$. This implies that:
\begin{align*}
    B = \left(iM_{v_3} \hat{D}_y + M_{v_4} \hat{D}_y\right) \hat{\pi}^{k>K_1} = M_{iv_3+v_4} \hat{D}_y\hat{\pi}^{k>K_1}.
\end{align*}
Physically, this identity reflects the deep-water limit where particle orbits become strictly circular, perfectly coupling the horizontal and vertical velocity components via a Cauchy–Riemann type relation. Consequently, in the high-frequency regime, the 2D advective dynamics collapse into a single complex multiplication operator, illustrating that surface perturbations at these microscopic scales are effectively decoupled from the bottom boundary.

We now factor the principal operator difference to isolate the algebraic discrepancy between the weight $w$ and the combined velocity sequence $iv_3 + v_4$:
\begin{align*}
    \|\hat{\pi}^{k>K_1}\hat{D}_y^{-1}M_w B - \hat{\pi}^{k>K_1}\|_{\mathcal{B}(X)} 
    &= \|\hat{\pi}^{k>K_1}\hat{D}_y^{-1}M_w M_{iv_3+v_4} \hat{D}_y \hat{\pi}^{k>K_1} - \hat{\pi}^{k>K_1}\|_{\mathcal{B}(X)} \\
    &= \|\hat{\pi}^{k>K_1}\hat{D}_y^{-1}(M_w M_{iv_3+v_4} - I) \hat{D}_y \hat{\pi}^{k>K_1}\|_{\mathcal{B}(X)}.
\end{align*}
This transformation decouples the unbounded derivative growth from the convolution error. To bound this operator norm in $X$, we recall from \eqref{def : W and hat W} that $\hat{\mathcal{W}} : X \to \ell^1$ is an isometric isomorphism. Because $\hat{\mathcal{W}}$, $\hat{D}_y$, and $\hat{\pi}^{k>K_1}$ are all diagonal multiplier operators, they mutually commute. Commuting $\hat{\mathcal{W}}$ through $\hat{D}_y^{-1}$ produces the scalar sequence ratio $\frac{1+|k|}{|k|}$, which attains its strict maximum of $\frac{K_1+2}{K_1+1}$ at the truncation boundary $|k| = K_1 + 1$. Using the algebra property \eqref{eq : multiplication equality}, we obtain:
\begin{align*}
    \|\hat{\pi}^{k>K_1}\hat{D}_y^{-1}(M_w M_{iv_3+v_4} - I) \hat{D}_y \hat{\pi}^{k>K_1}\|_{\mathcal{B}(X)} 
    &= \|\hat{\pi}^{k>K_1}\hat{\mathcal{W}}\hat{D}_y^{-1}(M_w M_{iv_3+v_4} - I) \hat{D}_y \hat{\mathcal{W}}^{-1}\hat{\pi}^{k>K_1}\|_{\mathcal{B}(\ell^1)}\\
    &\leq \frac{K_1+2}{K_1+1} \|M_w M_{iv_3+v_4} - I\|_{\mathcal{B}(\ell^1)} \\
    &\leq \frac{K_1+2}{K_1+1} \|w*(iv_3+v_4) - e_0\|_{\ell^1}.
\end{align*}

We now evaluate the total inversion error by reintroducing the exact finite-depth operator $D_y$. Using the triangle inequality, we decompose the error into the algebraic imperfection (bounded above) and the geometric modeling error $\hat{D}_y - D_y$:
\begin{align*}
    &\|\hat{\pi}^{k>K_1}\hat{D}_y^{-1}M_w \left(M_{v_3} D_x + M_{v_4} D_y\right) \hat{\pi}^{k>K_1} - \hat{\pi}^{k>K_1}\|_{\mathcal{B}(X)} \\
    &\leq \|\hat{\pi}^{k>K_1}\hat{D}_y^{-1}M_w B - \hat{\pi}^{k>K_1}\|_{\mathcal{B}(X)} + \|\hat{\pi}^{k>K_1}\hat{D}_y^{-1}M_w \left(B - M_{v_3} D_x - M_{v_4} D_y\right) \hat{\pi}^{k>K_1} \|_{\mathcal{B}(X)} \\
    &\leq \frac{K_1+2}{K_1+1} \|w*(iv_3+v_4) - e_0\|_{\ell^1} + \|\hat{\pi}^{k>K_1}\hat{\mathcal{W}}\hat{D}_y^{-1}M_w M_{v_4}\left(\hat{D}_y - D_y\right)\hat{\mathcal{W}}^{-1} \hat{\pi}^{k>K_1} \|_{\mathcal{B}(\ell^1)} \\
    &\leq \frac{K_1+2}{K_1+1} \|w*(iv_3+v_4) - e_0\|_{\ell^1} + \frac{K_1+2}{K_1+1} \|w*v_4\|_{\ell^1} \|\left(\hat{D}_y - D_y\right)\hat{\mathcal{W}}^{-1} \hat{\pi}^{k>K_1} \|_{\mathcal{B}(\ell^1)}.
\end{align*}
Recalling from Lemma \ref{lem : operator C} that $\left||k| - \omega_k\right| \leq \frac{2|k|}{1-e^{-2h|k|}}e^{-2h|k|}$ for all $k \neq 0$, the application of $\hat{\mathcal{W}}^{-1}$ divides out the leading $|k|$ growth, leaving the exponentially suppressed tail bound:
\begin{equation*}
    \|\left(\hat{D}_y - D_y\right)\hat{\mathcal{W}}^{-1} \hat{\pi}^{k>K_1} \|_{\mathcal{B}(\ell^1)} \leq \frac{2}{1-e^{-2hK_1}}e^{-2hK_1}.
\end{equation*}
Combining these estimates yields the exact definition of $\delta$ for the positive tail.

We now verify that the negative frequency tail obeys an identical bound. On the negative branch, $D_x\hat{\pi}^{k<-K_1} = -i\hat{D}_y\hat{\pi}^{k<-K_1}$, reflecting the physical reversal of particle rotation from counter-clockwise to clockwise. Consequently, the principal operator becomes $M_{-iv_3+v_4}\hat{D}_y$. Because $v_4 \in X_e$ (even, real coefficients) and $v_3 \in X_o$ (odd, purely imaginary coefficients) are derived from real-valued fluid variables, their convolutions satisfy strict complex conjugate symmetries:
\begin{align*}
    \|w*(iv_3+v_4) - e_0\|_{\ell^1} = \|w^* * (-iv_3+v_4) - e_0\|_{\ell^1} \qquad \text{and} \qquad \|w*v_4\|_{\ell^1} = \|w^* * v_4\|_{\ell^1}. 
\end{align*}
Exploiting this conjugate symmetry, an identical functional argument bounds the negative tail error by the same scalar $\delta$.

Finally, we demonstrate that these signed frequency components decouple completely, allowing the total operator norm to be evaluated as a block-diagonal maximum. Defining $P \bydef M_{v_3} D_x + M_{v_4} D_y$, we have:
\begin{align*}
    \|\hat{A}_\infty P\hat{\pi}^{>K_1} - \hat{\pi}^{>K_1}\|_{\mathcal{B}(X)} = \max\left\{ \|\hat{A}_\infty P\hat{\pi}^{k>K_1} - \hat{\pi}^{k>K_1}\|_{\mathcal{B}(X)}, \, \|\hat{A}_\infty P\hat{\pi}^{k<-K_1} - \hat{\pi}^{k<-K_1}\|_{\mathcal{B}(X)}\right\}. 
\end{align*}
Since the velocity coefficients $v_3$ and $v_4$ are derived from quadratic interactions of the numerical state $\bar{a} = \hat{\pi}^{\leq K_0}\bar{a}$, their Fourier support is strictly confined to $|k| \leq 2K_0$. Because the truncation threshold is set to $K_1 \geq 2K_0$, convolution with $v_3$ or $v_4$ cannot shift a frequency across the origin. Therefore:
\begin{align*}
    (M_{v_3} D_x + M_{v_4} D_y)\hat{\pi}^{k>K_1} &= \hat{\pi}^{k>0}(M_{v_3} D_x + M_{v_4} D_y)\hat{\pi}^{k>K_1}, \\
    (M_{v_3} D_x + M_{v_4} D_y)\hat{\pi}^{k<-K_1} &= \hat{\pi}^{k<0}(M_{v_3} D_x + M_{v_4} D_y)\hat{\pi}^{k<-K_1}. 
\end{align*}
Similarly, because $w = \hat{\pi}^{\leq K_0}w$, the inverse operator naturally localizes as
$$\hat{A}_\infty = \hat{\pi}^{k<-K_1}\hat{D}_y^{-1}M_{w^*} \hat{\pi}^{k<0} + \hat{\pi}^{k>K_1}\hat{D}_y^{-1}M_{w} \hat{\pi}^{k>0}.$$

The orthogonality of the half-space projections $\hat{\pi}^{k<0}$ and $\hat{\pi}^{k>0}$ perfectly annihilates the cross-terms in the composition $\hat{A}_\infty P$, confining the operator algebraically to independent, block-diagonal subspaces. This geometric decoupling, combined with the uniform bound $\max \leq \delta < 1$, rigorously guarantees the convergence of the Neumann series, thereby establishing the existence of a bounded left inverse on the infinite tail.
\end{proof}

We now construct the global approximate inverse $A: H_0 \to H_1$. This construction relies on a standard spectral splitting strategy: a rigorous numerical approximation for the low-frequency core and an analytical operator for the infinite-dimensional high-frequency tail.

For the high-frequency regime, we use the asymptotic structure established in Lemma \ref{lem : A infinity properties}. We numerically construct a  sequence $w$ with finite support $K_0$, optimized to minimize the convolution defect $\|w*(iv_3 +v_4)-e_0\|_{\ell^1}$. Effectively, $w$ acts as a localized algebraic preconditioner that forces the principal advective dynamics as close to the identity as possible. This enables the definition of the bounded tail operator $\hat{A}_\infty : \ell^1 \to X$ as in \eqref{def : A infinity}. 

Note that the restriction $\hat{A}_\infty : \ell^1_e \to X_e$ is strictly well-defined. Because the underlying velocity fields ($v_3$ and $v_4$) possess definite parity, the tail operator preserves the even symmetry of the wave profile, allowing us to rigorously restrict our analysis to the subspace of symmetric solutions. We embed this tail operator into the full augmented space via $A_\infty : H_0 \to H_1$, defined by:
\begin{align}\label{eq : def A infinity full space}
A_\infty U \bydef \begin{pmatrix}
0\\
\hat{A}_\infty a
\end{pmatrix}
\qquad \text{for all } U = (Q,a) \in H_0.
\end{align}

To synthesize the global approximate inverse $A : H_0 \to H_1$, we employ a block-matrix Schur complement decomposition, defining:
\begin{align}\label{def : approx inverse}
A \bydef A_K - A_K DF(\bar{U})A_\infty + A_\infty,
\end{align}
where $A_K = \pi^{\leq K_1} A_K \pi^{\leq K_1}$ is a finite-rank operator. Analytically, $A_K$ serves as an approximate inverse of the finite-dimensional Galerkin block $\pi^{\leq K_1} \left( DF(\bar{U}) - DF(\bar{U}) A_\infty DF(\bar{U}) \right)\pi^{\leq K_1}$. In practice, $A_K$ is computed numerically and its action is enclosed rigorously using interval arithmetic \cite{julia_interval, Moore_interval_analysis}. 

We deliberately choose this Schur complement structure to correct for the spectral ``leakage" that occurs when low-frequency and high-frequency modes interact. The cross-term $-A_K DF(\bar{U})A_\infty$ explicitly absorbs the off-diagonal error generated when the analytical tail operator interacts with the low-frequency numerical core. In what follows, we rigorously justify this choice by computing the explicit bounds for the residual operator $I -A DF(\bar{U})$.
\subsection{Analytic Bounds for the Residuals and Derivative}\label{ssec : computation bounds single wave}
With the global approximate inverse $A$ defined above in \eqref{def : approx inverse}, we now quantify its accuracy. Specifically, we compute the $Z_1$ linear bound, which provides a rigorous estimate of the residual error $\|I - ADF(\bar{U})\|_{\mathcal{B}(H_1)}$. This is achieved by partitioning the operator into finite-dimensional  blocks and an analytically controlled infinite-dimensional tail.

To systematically construct this bound, the following lemma introduces scalar constants for each component of the partitioned operator. $Z_\infty$ bounds the asymptotic tail error, encapsulating both the convolution defect of the algebraic weight $w$ and the geometric finite-depth approximation. $Z_{10}$ measures the core Galerkin matrix residual, quantifying the finite-dimensional inversion error on the numerical block. Finally, $Z_{11}$ and $Z_{12}$ capture the spectral ``leakage'' (the cross-terms): $Z_{11}$ bounds the defect introduced when the high-frequency tail operator evaluates the finite-dimensional numerical state, while $Z_{12}$ quantifies the action of the finite-rank inverse $A_K$ on the infinite-dimensional tail.

\begin{lemma}\label{lem : Z1 bound}
    Let $\delta$ be given as in Lemma \ref{lem : A infinity properties}, and let $Z_\infty, Z_{10}, Z_{11}, Z_{12}$ be strictly positive constants satisfying the inequalities:
    \begin{align*}
        Z_\infty &\geq \frac{\|w*(v_1*(\mathcal{D}\bar{a}) + v_2)\|_{\ell^1}}{K_1+1} + \frac{K_1+2}{K_1+1} \|w*v_1\|_{\ell^1}\frac{4K_0e^{-2hK_0}}{1-e^{-2hK_0}} \|\bar{a}\|_X + \delta, \\
        Z_{10} &\geq \|\pi^{\leq K_1} - A_K \pi^{\leq K_1} \left( DF(\bar{U}) - DF(\bar{U})A_\infty DF(\bar{U}) \right)\pi^{\leq K_1}\|_{\mathcal{B}(H_1)}, \\
        Z_{11} &\geq \|(\hat{\pi}^{k \leq K_1 + 4K_0} - \hat{\pi}^{\leq K_1}) \hat{D}_y^{-1} M_w \hat{\pi}^{\leq K_1 + 2K_0} D_a F(\bar{U})\hat{\pi}^{\leq K_1}\|_{\mathcal{B}(X_e)}, \\
        Z_{12} &\geq \|A_K DF(\bar{U})\pi^{>K_1}\|_{\mathcal{B}(H_1)}.
    \end{align*}
    Define the global uniform bound $Z_1$ as:
    \begin{align}
        Z_1 \bydef \max\left\{Z_{10}+Z_{11}, ~ (1+Z_{12})Z_\infty\right\}.
    \end{align}
    Then, the true residual is strictly bounded by $\|I - ADF(\bar{U})\|_{\mathcal{B}(H_1)} \leq Z_1$.
\end{lemma}

\begin{proof}
	Let us define the tail residual operator $N_\infty : H_1 \to H_1$ by:
	\begin{align}\label{def : N}
	N_\infty \bydef A_\infty DF(\bar{U})\pi^{>K_1} - \pi^{>K_1}.
	\end{align} 
	Our first goal is to compute a rigorous upper bound for the operator norm $\|N_\infty\|_{\mathcal{B}(H_1)}$.
	
	Recalling the definition of $DF(\bar{U})$ in \eqref{def : DF} and the embedding \eqref{eq : def A infinity full space}, we use Lemma \ref{lem : A infinity properties} to bound the principal part
	\begin{align*}
	\|A_\infty DF(\bar{U})\pi^{>K_1} - \pi^{>K_1}\|_{\mathcal{B}(H_1)} 
	&\leq \|\hat{A}_\infty(M_{v_3} D_x + M_{v_4} D_y)\hat{\pi}^{>K_1} - \hat{\pi}^{>K_1}\|_{\mathcal{B}(X)} + \|\hat{A}_\infty(M_{v_1} \mathcal{C}[\bar{a}] + M_{v_2}) \hat{\pi}^{>K_1}\|_{\mathcal{B}(X)} \\
	&\leq \delta + \|\hat{A}_\infty(M_{v_1} \mathcal{C}[\bar{a}] + M_{v_2}) \hat{\pi}^{>K_1}\|_{\mathcal{B}(X)}.
	\end{align*}
	Applying Lemma \ref{lem : operator C}, which establishes that the high-frequency commutator $\mathcal{C}[\bar{a}]$ is governed by the multiplier $M_{\mathcal{D}\bar{a}}$ up to an exponentially small finite-depth discrepancy, we obtain
	\begin{align*}
	\|\hat{A}_\infty(M_{v_1} \mathcal{C}[\bar{a}] + M_{v_2}) \hat{\pi}^{>K_1}\|_{\mathcal{B}(X)} 
	&\leq \|\hat{A}_\infty(M_{v_1} M_{\mathcal{D}\bar{a}} + M_{v_2}) \hat{\pi}^{>K_1}\|_{\mathcal{B}(X)} + \|\hat{A}_\infty M_{v_1}\|_{\mathcal{B}(X)} \frac{4K_0e^{-2hK_0}}{1-e^{-2hK_0}} \|\bar{a}\|_X \\
	&= \|\hat{A}_\infty M_{v_1*(\mathcal{D}\bar{a}) + v_2} \hat{\pi}^{>K_1}\|_{\mathcal{B}(X)} + \|\hat{A}_\infty M_{v_1}\|_{\mathcal{B}(X)} \frac{4K_0e^{-2hK_0}}{1-e^{-2hK_0}} \|\bar{a}\|_X.
	\end{align*}
	Since $\mathcal{D}\bar{a}$ has Fourier support confined to $K_0$, and both $v_1$ and $v_2$ are supported on $2K_0$, their convolution $v_1 * (\mathcal{D}\bar{a})$ is strictly supported on $3K_0$. We can therefore write
	\begin{align*}
	v_1*(\mathcal{D}\bar{a}) + v_2 = \hat{\pi}^{\leq 3K_0}(v_1*(\mathcal{D}\bar{a}) + v_2).
	\end{align*}
	When the multiplication operator $M_{v_1*(\mathcal{D}\bar{a}) + v_2}$ acts on the high-frequency subspace $\hat{\pi}^{>K_1}$, the lowest frequency it can generate is $K_1-3K_0$. Recalling the definition of $\hat{A}_\infty$ in \eqref{def : A infinity}, we deduce
	\begin{align*}
	\hat{A}_\infty M_{v_1*(\mathcal{D}\bar{a}) + v_2} \hat{\pi}^{k>K_1} 
	&= \left(\hat{\pi}^{k>K_1}\hat{D}_y^{-1}M_{w} + \hat{\pi}^{k<-K_1}\hat{D}_y^{-1}M_{w^*}\right) \hat{\pi}^{k>K_1-3K_0} M_{v_1*(\mathcal{D}\bar{a}) + v_2} \hat{\pi}^{>K_1} \\
	&= \hat{\pi}^{k>K_1}\hat{D}_y^{-1}M_{w}\hat{\pi}^{k>K_1-3K_0} M_{v_1*(\mathcal{D}\bar{a}) + v_2} \hat{\pi}^{>K_1}.
	\end{align*}
	This equality holds because the output range $\hat{\pi}^{k>K_1-3K_0}$ is strictly disjoint from the domain of the negative branch $\hat{\pi}^{k<-K_1}\hat{D}_y^{-1}M_{w^*}$. Specifically, since $w^*$ is supported on $K_0$, the negative branch only interacts with modes $k < -K_1 + K_0$. The condition $K_1 \geq 2K_0$ ensures $K_1 - 3K_0 \geq -K_1 + K_0$, precluding any spectral overlap.
	
	Recalling \eqref{def : W and hat W} and the fact that the diagonal operator $\hat{\mathcal{W}}$ commutes with the projection $\hat{\pi}^{k>K_1}$, we compute the operator norm in $\ell^1$
	\begin{align*}
	\|\hat{A}_\infty M_{v_1*(\mathcal{D}\bar{a}) + v_2} \hat{\pi}^{>K_1}\|_{\mathcal{B}(X)} 
	&= \|\hat{\pi}^{k>K_1}\hat{\mathcal{W}}\hat{D}_y^{-1} M_{w} M_{v_1*(\mathcal{D}\bar{a}) + v_2} \hat{\mathcal{W}}^{-1}\hat{\pi}^{k>K_1}\|_{\mathcal{B}(\ell^1)} \\
	&\leq \|\hat{\pi}^{k>K_1}\hat{\mathcal{W}}\hat{D}_y^{-1}\|_{\mathcal{B}(\ell^1)} \| M_{w} M_{v_1*(\mathcal{D}\bar{a}) + v_2}\|_{\mathcal{B}(\ell^1)} \| \hat{\mathcal{W}}^{-1}\hat{\pi}^{k>K_1}\|_{\mathcal{B}(\ell^1)}.
	\end{align*}
	Because $\hat{\mathcal{W}}\hat{D}_y^{-1}$ and $\hat{\mathcal{W}}^{-1}$ are diagonal, their operator norms are simply the supremum of their multipliers over the tail
	\begin{align*}
	\|\hat{\pi}^{k>K_1}\hat{\mathcal{W}}\hat{D}_y^{-1}\|_{\mathcal{B}(\ell^1)} = \sup_{k \geq K_1 +1} \frac{|k|+1}{|k|} = \frac{K_1+2}{K_1 +1}, \qquad \text{and} \qquad
	\|\hat{\mathcal{W}}^{-1}\hat{\pi}^{k>K_1}\|_{\mathcal{B}(\ell^1)} = \sup_{k \geq K_1 +1} \frac{1}{|k|+1} = \frac{1}{K_1 +2}.
	\end{align*}
	Using the Banach algebra property from Lemma \ref{lem:Banach_algebra}, this simplifies to
	\begin{align*}
	\|\hat{A}_\infty M_{v_1*(\mathcal{D}\bar{a}) + v_2} \hat{\pi}^{>K_1}\|_{\mathcal{B}(X)} \leq \frac{1}{K_1+1} \| M_{w} M_{v_1*(\mathcal{D}\bar{a}) + v_2}\|_{\mathcal{B}(\ell^1)} = \frac{\|w*(v_1*(\mathcal{D}\bar{a}) + v_2)\|_{\ell^1}}{K_1+1}. 
	\end{align*}
	Similarly, bounding the second term yields $\|A_\infty M_{v_1}\|_{\mathcal{B}(X)} \leq \frac{K_1+2}{K_1+1} \|w*v_1\|_{\ell^1}$. Combining all bounds, we confirm
	\begin{align*}
	\|N_\infty\|_{\mathcal{B}(H_1)} \leq \frac{\|w*(v_1*(\mathcal{D}\bar{a}) + v_2)\|_{\ell^1}}{K_1+1} + \frac{K_1+2}{K_1+1} \|w*v_1\|_{\ell^1}\frac{4K_0e^{-2hK_0}}{1-e^{-2hK_0}} \|\bar{a}\|_X + \delta \leq Z_\infty.
	\end{align*}
	
	Finally, we construct the global bound. By decomposing the augmented space into low-frequency and high-frequency subspaces, $H_1 = \pi^{\leq K_1} H_1 \oplus \pi^{> K_1} H_1$, the operator $1$-norm of the full residual is bounded by the maximum of its block-wise components
	\begin{align*}
	\|I - ADF(\bar{U})\|_{\mathcal{B}(H_1)} = \max\left\{ \|\pi^{\leq K_1} \left(I - A DF(\bar{U})\right)\pi^{\leq K_1} \|_{\mathcal{B}(H_1)}, \, \|\pi^{> K_1} \left(I - A DF(\bar{U})\right)\pi^{> K_1}\|_{\mathcal{B}(H_1)} \right\}.
	\end{align*}
	For the finite-dimensional Galerkin block, we substitute the definition of $A$ from \eqref{def : approx inverse} and isolate the leakage
	\begin{align*}
	\|\pi^{\leq K_1} \left(I - A DF(\bar{U})\right)\pi^{\leq K_1}\|_{\mathcal{B}(H_1)} 
	&\leq \|\pi^{\leq K_1} - A_K \pi^{\leq K_1} \left( DF(\bar{U}) - DF(\bar{U})A_\infty DF(\bar{U}) \right)\pi^{\leq K_1}\|_{\mathcal{B}(H_1)} \\
	&\qquad + \|\hat{\pi}^{k > K_1} \hat{D}_y^{-1} M_w D_a F(\bar{U})\hat{\pi}^{\leq K_1}\|_{\mathcal{B}(X)} \\
	&= Z_{10} + \|(\hat{\pi}^{\leq K_1 + 4K_0} - \hat{\pi}^{\leq K_1}) \hat{D}_y^{-1} M_w \hat{\pi}^{\leq K_1 + 2K_0} D_a F(\bar{U})\hat{\pi}^{\leq K_1}\|_{\mathcal{B}(X)} \\
	&= Z_{10} + Z_{11}.
	\end{align*}
	In the third line, the infinite tail projection $\hat{\pi}^{> K_1}$ is reduced to the finite window $\hat{\pi}^{K_1 < k \leq K_1 + 4K_0}$. Because the operator chain $M_w D_a F(\bar{U})$ consists exclusively of convolutions with finite-support sequences, low-frequency modes can only propagate a finite spectral distance into the high-frequency tail, ensuring $Z_{11}$ is strictly computable.
	
	For the high-frequency residual, we evaluate how the approximate inverse acts on the infinite-dimensional tail by using the exact algebraic identity $A_\infty DF(\bar{U})\pi^{> K_1} = \pi^{>K_1} + N_\infty$. This precisely cancels the leading-order identity term
	\begin{align*}
	\|\pi^{> K_1} \left(I - A DF(\bar{U})\right)\pi^{> K_1}\|_{\mathcal{B}(H_1)}
	&= \|\pi^{> K_1} - A_K DF(\bar{U})\pi^{> K_1} + A_K DF(\bar{U}) A_\infty DF(\bar{U})\pi^{> K_1}\\ &\qquad\qquad\qquad\qquad - A_\infty DF(\bar{U})\pi^{> K_1} \|_{\mathcal{B}(H_1)} \\
	&= \|-N_\infty + A_K DF(\bar{U})N_\infty\|_{\mathcal{B}(H_1)} \\
	&\leq (1 + Z_{12})Z_\infty.
	\end{align*}
	Since both diagonal blocks are appropriately bounded by the prescribed constants, taking their maximum validates the global estimate $Z_1$.
\end{proof}

Having established the $Z_1$ bound for the approximate inverse $A$, we now quantify the accuracy of our candidate $\bar{U}$. The following lemma establishes an upper bound $Y$ for the residual of the approximate solution in the $H_1$ norm.

\begin{lemma}\label{lem : Y bound single wave}
	Defining the constant $Y > 0$ as
	\begin{align*}
	Y \bydef \|A_K(F(\bar{U}) - DF(\bar{U})A_\infty F(\bar{U}))\|_{H_1} + \frac{K_1+2}{K_1+1} \left( \|\hat{\pi}^{k>K_1}(w * f(\bar{U}))\|_{\ell^1} + \|\hat{\pi}^{k<-K_1}(w^* * f(\bar{U}))\|_{\ell^1} \right),
	\end{align*}
	we obtain the residual bound $\|AF(\bar{U})\|_{H_1} \leq Y.$
\end{lemma}

\begin{proof}
	Recall the definition of the approximate inverse $A = A_K(I - DF(\bar{U})A_\infty) + A_\infty$. Applying this to the nonlinear residual $F(\bar{U})$, we obtain the decomposition
	\begin{align*}
	    \|AF(\bar{U})\|_{H_1} \leq \|A_K(I - DF(\bar{U})A_\infty) F(\bar{U})\|_{H_1} + \|A_\infty F(\bar{U})\|_{H_1}.
	\end{align*}
	To bound the second term, we use the definition of $\hat{A}_\infty$ in \eqref{def : A infinity}. Since $\hat{D}_y^{-1}$ acts as the diagonal multiplier $1/|k|$ for $|k| \geq 1$, and $\hat{\mathcal{W}}$ is the weight $(1+|k|)$, the $H_1$ norm (which includes the $|k|$ growth) simplifies via the supremum of the ratio $\frac{1+|k|}{|k|}$. Evaluating this at the boundary $|k| = K_1+1$, and exploiting the disjoint supports of the positive and negative branches, we obtain
	\begin{align*}
	    \|A_\infty F(\bar{U})\|_{H_1} 
	    &\leq \sup_{|k| > K_1} \left( \frac{1+|k|}{|k|} \right) \left( \|\hat{\pi}^{k>K_1}(w * f(\bar{U}))\|_{\ell^1} + \|\hat{\pi}^{k<-K_1}(w^* * f(\bar{U}))\|_{\ell^1} \right) \\
	    &= \frac{K_1+2}{K_1+1} \left( \|\hat{\pi}^{k>K_1}(w * f(\bar{U}))\|_{\ell^1} + \|\hat{\pi}^{k<-K_1}(w^* * f(\bar{U}))\|_{\ell^1} \right).
	\end{align*}
	Combining this with the finite-dimensional residual bound completes the proof.
\end{proof}

\begin{lemma}[The Nonlinear $Z_2$ Bound]\label{lem : Z2 single wave}
	Let $r >0$ and consider the closed ball $B_r(\bar{U}) \subset H_1$. Let $A : H_0 \to H_1$ be the global approximate inverse defined in \eqref{def : approx inverse}. Suppose $Z_{D^2}$ and $Z_{D^3}(r)$ are strictly positive constants satisfying the uniform bounds:
	\begin{align*}
	Z_{D^2} &\geq 2\gamma^2\max\left\{1, \, \frac{1}{h^2}\right\}\left(2\|\bar{a}\|_{X} + \|\bar{a}\|_{\ell^1}\right)^2 + 6\max\left\{1, \, \frac{1}{h}\right\}|\gamma| \left\|mD_y e_0 + \frac{\gamma}{2} D_y(\bar{a}^2) - \gamma \bar{a}*(D_y\bar{a})\right\|_{\ell^1}\\
	&\qquad + \left(1 + \max\left\{1, \, \frac{1}{h^2}\right\}\right)\left( 4(1 + 2g) \|\bar{a}\|_X + 2\|\bar{Q}e_0 - 2g\bar{a}\|_{\ell^1}\right), \\[8pt]
	Z_{D^3}(r) &\geq 27\frac{|\gamma|^2}{h^2}(\|\bar{a}\|_X +r) + 4g\left(1 + \max\left\{1, \, \frac{1}{h^2}\right\}\right)\\
	&\qquad + (|\bar{Q}| + 2g)\left(1 + \max\left\{1, \, \frac{1}{h^2}\right\}\right) + 3(|\bar{Q}| +r)\left(1 + \max\left\{1, \, \frac{1}{h^2}\right\}\right).
	\end{align*}
	Recalling the  constant $Z_{12}$ introduced in Lemma \ref{lem : Z1 bound}, we define the constant $Z_2(r)$ as
	\begin{align*}
	Z_2(r) \bydef \max\left\{\|\mathcal{W}(A_K - A_K DF(\bar{U})A_\infty + A_\infty)\pi^{\leq K_1} \|_{\mathcal{B}(\ell^1)}, \, (1+Z_{12})\frac{K_1+2}{K_1+1} \|w\|_{\ell^1}\right\} \Big(Z_{D^2} + Z_{D^3}(r) \cdot r\Big).
	\end{align*}
	Then, for any perturbation $V \in H_1$ such that $\|V\|_{H_1} \leq r$, the operator norm of the Fréchet derivative difference is bounded by
	\begin{equation}
	\|A(DF(\bar{U}+V) - DF(\bar{U}))\|_{\mathcal{B}(H_1)} \leq Z_2(r) \cdot r.
	\end{equation}
\end{lemma}

\begin{proof}
	First, we compute explicit formulas for $D^2F$ and $D^3F$. Recall from \eqref{def : zero finding F} that $F(U) = \begin{pmatrix}
	a_0 - h\\
	f(U)
	\end{pmatrix}.$ Because the first component of $F$ is affine with respect to the fluid depth $h$, it vanishes upon differentiation and does not contribute to $D^2F$ or $D^3F$. Now, observe that 
	\begin{equation}
	f(U) = g(a)^2 - p(Q,a) \qquad \text{for all } a \in X,
	\end{equation}
	where the intermediate mappings $g(a)$ and $p(Q,a)$ are defined as 
	\begin{align*}
	g(a) = m D_y e_0 + \frac{\gamma}{2} D_y(a^2) - \gamma a*(D_ya) \qquad \text{and} \qquad 
	p(Q,a) = (Q-2ga)*\left((D_xa)^2+(D_ya)^2\right),
	\end{align*}
	respectively. Computing the derivatives of these maps yields the desired result. 
	
	We begin by considering $g(a)^2$. By the chain rule, we have
	\begin{align*}
	D (g(a)^2) u &= 2 g(a) Dg(a) u, \\
	D^2 (g(a)^2) (u,v) &= 2 Dg(a)v Dg(a) u + 2 g(a) D^2g(a) (u,v).
	\end{align*}
	Similarly, differentiating a third time yields
	\begin{align*}
	D^3 (g(a)^2) (u,v,w) &= 2 D^2g(a)(v,w) Dg(a) u + 2 D^2g(a)(u,w) Dg(a) v \\
	&\qquad + 2 Dg(a)w D^2g(a) (u,v) + 2 g(a) D^3g(a) (u,v,w).
	\end{align*}
	Notice, however, the explicit derivatives of $g(a)$
	\begin{equation}\label{eq: derivatives of g}
	\begin{aligned}
	Dg(a)u &= \gamma D_y(a*u) - \gamma u*(D_ya) - \gamma a*(D_yu), \\
	D^2g(a)(u,v) &= \gamma D_y(v*u) - \gamma u*(D_yv) - \gamma v*(D_yu), \\
	D^3g(a)(u,v,w) &= 0.
	\end{aligned}
	\end{equation}
	The final term identically vanishes because $g$ is purely quadratic. This yields a substantial simplification: $D^3 (g(a)^2)$ depends exclusively on the interactions of $g$, $Dg$, and $D^2g$. In fact, all three terms in the expansion of $D^3(g(a)^2)$ consist of a product of $Dg$ and $D^2g$.
	
	From the definition of the norm on $X$, we observe the derivative bounds:
	\begin{align}\label{eq : estimates Dx and Dy}
	\|D_xv\|_{\ell^1} \leq \|v\|_X \qquad \text{and} \qquad \|D_yv\|_{\ell^1} \leq \max\left\{1, \,\frac{1}{h}\right\}\|v\|_X \qquad \text{for all } v \in X_e.
	\end{align}
	Combining Lemma \ref{lem:Banach_algebra} with the bounds \eqref{eq : estimates Dx and Dy} for the expressions in \eqref{eq: derivatives of g}, we obtain:
	\begin{equation}
	\|Dg(a)u\|_{\ell^1} \leq 3|\gamma| \max\left\{1, \,\frac{1}{h^2}\right\}(\|\bar{a}\|_X+r) \qquad \text{and} \qquad \|D^2g(a)(u,v)\|_{\ell^1} \leq 3|\gamma|\max\left\{1, \,\frac{1}{h^2}\right\},
	\end{equation}
	for all unit vectors $u,v,w \in B_1(0) \subset X$. Consequently, each component in the expansion of $D^3(g(a)^2)$ is bounded by $18|\gamma|^2 \max\left\{1, \,\frac{1}{h^2}\right\}(\|\bar{a}\|_X+r)$, which yields
	\begin{align}\label{eq : estimate 3rd deriv g}
	\sup_{a \in B_r(\bar{a})} \sup_{u,v,w \in B_1(0)} \|D^3 (g(a)^2) (u,v,w)\|_{\ell^1} \leq 54|\gamma|^2\max\left\{1, \,\frac{1}{h^2}\right\}(\|\bar{a}\|_X +r).  
	\end{align}
	
	Let us now consider $p(Q,a)$. Differentiating with respect to the sequence argument $a$, we find
	\begin{equation}\label{eq : derivates of p}
	\begin{aligned}
	D_a p(Q,a)u &= - 2g\big((D_xa)^2 + (D_ya)^2\big)*u + 2(Q-2g a)\big((D_xa) * (D_xu) + (D_ya)*(D_yu)\big), \\
	D_a^2 p(Q,a)(u,v) &= - 4g\big((D_xa)*(D_xv) + (D_ya)*(D_yv)\big)*u - 4gv*\big((D_xa) * (D_xu) + (D_ya)*(D_yu)\big) \\
	&\qquad + 2(Q-2g a)\big((D_xv) * (D_xu) + (D_yv)*(D_yu)\big), \\
	D_a^3 p(Q,a)(u,v,w) &= - 4g\big((D_xw)*(D_xv) + (D_yw)*(D_yv)\big)*u - 4gv*\big((D_xw) * (D_xu) + (D_yw)*(D_yu)\big) \\
	&\qquad + 2(Q-2g w)\big((D_xv) * (D_xu) + (D_yv)*(D_yu)\big).
	\end{aligned}
	\end{equation}
	Following identical logic as above, we apply Lemma \ref{lem:Banach_algebra} and the derivative bounds \eqref{eq : estimates Dx and Dy} to obtain
	\begin{align}\label{eq : estimate 3rd deriv p with a}
	\sup_{a \in B_r(\bar{a})} \sup_{u,v,w \in B_1(0)} \|D_a^3p(Q,a)(u,v,w)\|_{\ell^1} \leq 8g\left(1 + \max\left\{1, \,\frac{1}{h^2}\right\}\right) + 2(|Q| + 2g)\left(1 + \max\left\{1, \,\frac{1}{h^2}\right\}\right).
	\end{align}
	Additionally, computing the mixed derivatives and derivatives with respect to the scalar $Q$ yields
	\begin{equation}\label{eq : second derivatives of p second}
	\begin{aligned}
	D_Q p(Q,a) \beta &= \beta\big((D_xa)^2+(D_ya)^2\big), \\
	D_Q^2 p(Q,a) (\alpha, \beta) &= 0, \\
	D_Q D_a p(Q,a) (\beta,u) &= 2\beta\big((D_xa)*(D_xu)+(D_ya)*(D_yu)\big), \\
	D_Q D_a^2 p(Q,a) (\beta,u,v) &= 2\beta\big((D_xv)*(D_xu)+(D_yv)*(D_yu)\big).
	\end{aligned}
	\end{equation}
	Consequently, we bound the highest-order mixed derivative as
	\begin{align}\label{eq : estimate third derivative p with Q}
	\sup_{(Q,a) \in B_r(\bar{U})} \sup_{|\beta| \leq 1} \sup_{u,v \in B_1(0)} \|D_Q D_a^2 p(Q,a) (\beta,u,v)\|_{\ell^1} \leq 2(|\bar{Q}| +r)\left(1 + \max\left\{1, \,\frac{1}{h^2}\right\}\right).
	\end{align}
	
	Now, let $V \in H_1$ be a perturbation such that $\|V\|_{H_1} \leq r$. Applying the mean value inequality for Banach spaces, and inserting the bounds from \eqref{eq : estimate 3rd deriv g}, \eqref{eq : estimate 3rd deriv p with a}, and \eqref{eq : estimate third derivative p with Q}, we establish the residual expansion bound
    \small{
	\begin{align*}
	\|A(DF(\bar{U}+V) - DF(\bar{U}))\|_{\mathcal{B}(H_1)} 
	&\leq \sup_{u,v \in B_1(0)}\|AD^2F(\bar{U})(u,v)\|_{H_1} \cdot r + \frac{1}{2}\sup_{U \in B_r(\bar{U})}\sup_{u,v,w \in B_1(0)}\|AD^3F(U)(u,v,w)\|_{H_1} \cdot r^2\\
	&\leq \sup_{u,v \in B_1(0)}\|AD^2F(\bar{U})(u,v)\|_{H_1} \cdot r + \|\mathcal{W}A\| \cdot r^2 \bigg[ 27|\gamma|^2\max\left\{1, \,\frac{1}{h^2}\right\}(\|\bar{a}\|_X +r) \\
	&\qquad + 4g\left(1 + \max\left\{1, \,\frac{1}{h^2}\right\}\right) + (|\bar{Q}| + 2g)\left(1 + \max\left\{1, \,\frac{1}{h^2}\right\}\right) \\
	&\qquad + 3(|\bar{Q}| +r)\left(1 + \max\left\{1, \,\frac{1}{h^2}\right\}\right) \bigg] \\
	&\leq \sup_{u,v \in B_1(0)}\|AD^2F(\bar{U})(u,v)\|_{H_1} \cdot r + Z_{D^3}(r) \cdot r^2.
	\end{align*}
    }
    \normalsize
	To isolate the second derivative term, let $u,v \in X_e$. Using the isometry \eqref{def : W and hat W}, we pull the weight operator out to write
	\begin{align*}
	\|AD^2F(\bar{U})(u,v)\|_{H_1} \leq \|\mathcal{W}A\|_{\mathcal{B}( \ell^1)} \|D^2f(\bar{U})(u,v)\|_{\ell^1}.
	\end{align*}
	The operator norm $\|\mathcal{W}A\|$ is evaluated by decomposing the space into low-frequency and tail components. The numerical inverse $A_K$ handles modes up to $K_1$, whereas the high-frequency defect is analytically controlled by $A_\infty$ and the leakage constant $Z_{12}$
	\begin{align*}
	\|\mathcal{W}A\|_{\mathcal{B}( \ell^1)} &= \max\left\{\|\mathcal{W}(A_K - A_K DF(\bar{U})A_\infty + A_\infty)\pi^{\leq K_1} \|_{\mathcal{B}(\ell^1)}, \, \|\mathcal{W}( - A_K DF(\bar{U})A_\infty + A_\infty)\pi^{> K_1}\|_{\mathcal{B}(\ell^1)}\right\}\\
	&\leq \max\left\{\|\mathcal{W}(A_K - A_K DF(\bar{U})A_\infty + A_\infty)\pi^{\leq K_1} \|_{\mathcal{B}(\ell^1)}, \, (1+Z_{12})\|\mathcal{W}A_\infty\pi^{>K_1}\|_{\mathcal{B}( \ell^1)}\right\}\\
	&\leq \max\left\{\|\mathcal{W}(A_K - A_K DF(\bar{U})A_\infty + A_\infty)\pi^{\leq K_1} \|_{\mathcal{B}( \ell^1)}, \, (1+Z_{12})\frac{K_1+2}{K_1+1} \|w\|_{\ell^1}\right\}.
	\end{align*}
	Note that although $A_\infty$ is primarily defined to invert the high-frequency tail, its deliberate inclusion in the finite-dimensional cross-term $(A_K - A_K DF(\bar{U})A_\infty + A_\infty)\pi^{\leq K_1}$ explicitly cancels the spectral leakage, ensuring a rigorous analytic transition.
	
	To conclude the proof, we must estimate $\|D^2f(\bar{U})(u,v)\|_{\ell^1}$. We first combine our block derivations into an explicit expression for the second derivative. Let $U = (\alpha, u)$ and $V = (\beta, v)$ be vectors in the augmented space $H_1$. Using equations \eqref{eq: derivatives of g}, \eqref{eq : derivates of p}, and \eqref{eq : second derivatives of p second}, we assemble
	\begin{equation}\label{eq : second derivative F}
	\begin{aligned}
	D^2f(\bar{U})(U,V) &= 2 \gamma^2 \big(D_y(\bar{a}*u) - u*(D_y\bar{a}) - \bar{a}*(D_yu)\big)\big(D_y(\bar{a}*v) - v*(D_y\bar{a}) - \bar{a}*(D_yv)\big) \\
	&\qquad + 2 \gamma \left(mD_y e_0 + \frac{\gamma}{2} D_y(\bar{a}^2) - \gamma \bar{a}*(D_y\bar{a})\right) \big(D_y(u*v) - v*(D_yu) - u*(D_yv)\big)\\
	&\qquad - 2(\alpha - 2gu)(\nabla \bar{a} \cdot \nabla v) - 2(\beta - 2gv)(\nabla \bar{a} \cdot \nabla u) - 2(\bar{Q} - 2g\bar{a})(\nabla u \cdot \nabla v),
	\end{aligned}
	\end{equation}
	where we introduce the slight abuse of notation for the functional gradient dot-product 
	\begin{align*}
	\nabla u \cdot \nabla v \bydef (D_xu)*(D_xv) + (D_yu)*(D_yv) \qquad \text{for all } u,v \in X.
	\end{align*}
	Combining the algebra property $\|u*v\|_{\ell^1} \leq \|u\|_{\ell^1}\|v\|_{\ell^1}$ with Lemma \ref{lem:Banach_algebra} and the derivative bounds \eqref{eq : estimates Dx and Dy}, we obtain
	\begin{align*}
	\|D_y(\bar{a}*u) - u*(D_y\bar{a}) - \bar{a}*(D_yu)\|_{\ell^1} &\leq \max\left\{1, \,\frac{1}{h}\right\}(2\|\bar{a}\|_{X} + \|\bar{a}\|_{\ell^1}), \\
	\|\nabla\bar{a} \cdot \nabla v\|_{\ell^1} &\leq \|\bar{a}\|_X\left(1+ \max\left\{1, \,\frac{1}{h^2}\right\}\right),
	\end{align*}
	for all $u,v \in B_1(0) \subset X$. Substituting these bounds directly into expression \eqref{eq : second derivative F} verifies the prescribed estimate for $Z_{D^2}$, which concludes the proof.
\end{proof}
In this section, we have derived explicit, computable formulas to bound the operators required by the Newton-Kantorovich (Theorem \ref{th : radii polynomial}). Because these bounds ultimately reduce to the evaluation of finite-dimensional norms, they can be rigorously verified in practice using interval arithmetic. We demonstrate this procedure in Section \ref{sec : existence proofs results}, where we execute the computer-assisted proof to guarantee the existence of a zero of $F$. Furthermore, by establishing these analytical estimates first in the simpler setting of a single isolated wave, we lay the mathematical groundwork for the continuation arguments in Sections \ref{sec : local bif} and \ref{sec : global bif}. In those sections, we will frequently appeal to these foundational bounds when constructing both the local and global bifurcation branches.

However, before we can begin tracing these branches, we must first establish \textit{a posteriori} criteria to guarantee that any rigorously validated zero of $F$ in the sequence space $H_1$ genuinely represents a physically meaningful water wave. Because the topological preservation guaranteed by Corollary \ref{cor:global_monotonicity} only applies once a valid branch is established, we must manually verify that our initial isolated wave profile satisfies the strict non-degeneracy \eqref{eq : nondegeneracy_block} and monotonicity \eqref{eq : monotonicity_block} inequalities. The formulation of these rigorous \textit{a posteriori} verification bounds is the subject of the following subsections.

\subsection{A posteriori verification of non-degeneracy conditions}

Suppose that we have successfully applied Theorem \ref{th : radii polynomial}, thereby guaranteeing the existence of a locally unique true solution $\tilde{U} \in H_1$ such that 
\begin{align*}
    \|\tilde{U} - \bar{U}\|_{H_1} \leq r_0, 
\end{align*}
for some rigorously computed, strictly positive radius $r_0$. We decompose this true solution as $\tilde{U} = (\tilde{Q},\tilde{a})$ and let $\tilde{\eta}$ be the physical function representation of the Fourier sequence $\tilde{a}$. The following result ensures that the dual non-degeneracy constraints \eqref{eq : nondegeneracy_block} are automatically satisfied upon a successful application of Theorem \ref{th : radii polynomial}.

\begin{lemma}\label{lem : proof of non-degeneracy}
    Let $Z_1$ and $Z_2(r_0)$ be given as in Lemmas \ref{lem : Z1 bound} and \ref{lem : Z2 single wave}, respectively. If the contraction condition $Z_1 + Z_2(r_0) \cdot r_0 < 1$ holds, then the true wave profile satisfies the non-degeneracy condition \eqref{eq : nondegeneracy_block}.
\end{lemma}

\begin{proof}
	Let $w$ be the high-frequency weight from Lemma \ref{lem : A infinity properties}, and define the continuous velocity coefficients corresponding to the true solution $\tilde{U}$:
	\begin{align*}
	\tilde{v}_3 &= -2(\tilde{Q}-2g \tilde{a})*(D_x \tilde{a}), \qquad \tilde{v}_4 = -2(\tilde{Q}-2g \tilde{a})*(D_y \tilde{a}).
	\end{align*}
	By the definitions of the bounds $Z_1$ and $Z_2(r_0)$, the total convolution defect is bounded by
	\begin{align*}
	\|w*(i\tilde{v}_3 + \tilde{v}_4) - e_0\|_{\ell^1} \leq Z_1 + Z_2(r_0) \cdot r_0 < 1.
	\end{align*}
	Because $\ell^1$ is a Banach algebra, this bound guarantees via a Neumann series that the sequence $i\tilde{v}_3 + \tilde{v}_4 = -2(\tilde{Q}-2g \tilde{a})*(iD_x \tilde{a} + D_y \tilde{a})$ is invertible. Consequently, neither of its factors can vanish, ensuring both $\tilde{Q}-2g \tilde{\eta}$ and $(D_x\tilde{\eta})^2 + (D_y\tilde{\eta})^2$ are strictly non-zero globally.
	
	Since the physical mapping is real-valued, the sum of squares $(D_x\tilde{\eta})^2 + (D_y\tilde{\eta})^2$ must be strictly positive. Finally, because $\tilde{U}$ is an exact root, it satisfies the dynamic boundary condition $f(\tilde{U}) = 0$. Rearranging this identity yields $\tilde{Q} - 2g\tilde{\eta} = g(\tilde{a})^2 / \big((D_x\tilde{\eta})^2 + (D_y\tilde{\eta})^2\big)$. As the quotient of a non-negative square and a strictly positive denominator, we conclude $\tilde{Q} - 2g\tilde{\eta} > 0$, completing the proof.
\end{proof}

 \subsubsection{Regularity and control of the second derivative}
 Before we can rigorously evaluate bounds on the pointwise second derivatives, we must first establish that our abstract sequence corresponds to a smooth, classical solution.

\begin{lemma}\label{lem : solutions are C infinity}
    Suppose Theorem \ref{th : radii polynomial} guarantees the existence of a true solution $\tilde{U} = (\tilde{Q}, \tilde{a}) \in H_1$. Let $\tilde{\eta}$ be the spatial function corresponding to the Fourier sequence $\tilde{a} \in X_e$. Then $\tilde{\eta}$ is an even, infinitely differentiable periodic function ($\tilde{\eta} \in C^\infty_{\mathrm{per}}$).
\end{lemma}

\begin{proof}
    Because Lemma \ref{lem : proof of non-degeneracy} ensures the non-degeneracy conditions \eqref{eq : nondegeneracy_block} hold, the infinite smoothness $\tilde{\eta} \in C^\infty_{\mathrm{per}}$ is a direct application of the elliptic regularity arguments in Section 4 of \cite{constantin_2011}. The even symmetry is inherent to the subspace $X_e$.
\end{proof}

With the gap between our abstract sequence and a classical solution bridged, we now compute explicit bounds on the second derivatives, which are required to close the continuation arguments in the subsequent sections.

\begin{lemma}\label{lem : enclosure second derivative of U}
	Let $\tilde{U}= (\tilde{Q}, \tilde{a}) \in H_1$ be a true zero of $F$ whose existence is guaranteed by Theorem \ref{th : radii polynomial}, satisfying the bound $\|\tilde{U} - \bar{U}\|_{H_1} \leq r_0$. Let $T_0, T : H_1 \to \ell^1_e$ be the nonlinear mappings defined by
	\begin{align*}
	T_0(U) &\bydef (Q - 2g a) * \bigg[ \left(\frac{\gamma}{2} D_y(a^2) - \gamma a * (D_y a)\right) * \left(\frac{m}{h} + \frac{\gamma}{2}\big(D_y(a^2) - a * (D_y a)\big)\right) \\
	&\qquad - g (D_x a) * \big((D_x a)^2 + (D_y a)^2\big) \bigg], \\
	T(U) &\bydef T_0(U) * \big((Q - 2ga)^2 * ((D_x a)^2 + (D_y a)^2)\big)^{-1}.
	\end{align*}
	Then, the exact second derivative of the true profile satisfies the identity
	\begin{align*}
	D_x^2 \tilde{a} = T(\tilde{U}) * (D_x \tilde{a}) + (D_y \tilde{a}) * (H_h T(\tilde{U})). 
	\end{align*}
	Recall the sequence $w \in X$ defined in Lemma \ref{lem : A infinity properties}, and define the approximate sequence operator $\bar{T} : H_1 \to \ell^1_e$ as
	\begin{align*}
	\bar{T}(U) \bydef |w|^2 * T_0(U), \quad \text{ where } |w|^2 = w*w^*.
	\end{align*}
	Moreover, define the numerical second derivative proxy $\bar{a}_{xx} \in X_e$ by
	\begin{align*}
	\bar{a}_{xx} \bydef \bar{T}(\bar{U}) * (D_x \bar{a}) + (D_y \bar{a}) * (H_h \bar{T}(\bar{U})).
	\end{align*}
	Suppose $\epsilon_1, \epsilon_2 > 0$ are constants satisfying the inequalities
	\begin{align*}
	\epsilon_1 &\geq \delta + \|w\|_{\ell^1}\Big[ \max\{1,\, 2g\} \big(1+\coth(h)\big) \|\bar{a}\|_X + \big(|\bar{Q}| + r_0 + 2g (\|\bar{a}\|_{\ell^1}+r_0)\big)\big(1 + \coth(h)\big) \Big] r_0, \\
	\epsilon_2 &\geq \frac{1}{1- 2\epsilon_1 - \epsilon_1^2} \left( \big\||w|^2 * \big(T_0(\tilde{U}) - T_0(\bar{U})\big)\big\|_{\ell^1} + (2\epsilon_1 + \epsilon_1^2) \big\||w|^2 * T_0(\bar{U})\big\|_{\ell^1} \right),
	\end{align*}
	under the assumption that $2\epsilon_1 + \epsilon_1^2 < 1$. Then, the true second derivative is rigorously enclosed by the bound
	\begin{align}
	\|D_x^2 \tilde{a} - \bar{a}_{xx}\|_{\ell^1} \leq \|\bar{a}\|_{X} \big(1 + \coth(h)\big) \epsilon_2 + \|\bar{T}(\bar{U})\|_{\ell^1} \big(1 + \coth(h)\big) r_0.
	\end{align}  
\end{lemma}

\begin{proof}
	Let $\tilde{\eta}$ be the spatial function representation of the true Fourier coefficients $\tilde{a} \in X_e$. By Lemma \ref{lem : solutions are C infinity}, we know that $\tilde{\eta} \in C^\infty_{\mathrm{per}}$. 
	
	Following Section 4 of \cite{constantin_2011}, we define the continuous function $f \bydef (\tilde{\eta}_x^2 + \tilde{\eta}_y^2)^{\frac{1}{2}}$ and the angle function $\theta \bydef \mathcal{H}_h(\log(f) - [\log(f)])$, where $\mathcal{H}_h$ denotes the physical finite-depth Hilbert transform. The gradients of the conformal mapping are entirely determined by these functions:
	\begin{align*}
	\tilde{\xi}_x = f \cos(\theta), \qquad \tilde{\eta}_x = f \sin(\theta).
	\end{align*}
	Furthermore, equation (4.6) of \cite{constantin_2011} provides an explicit formula for $f$ via the dynamic boundary condition:
	\begin{align*}
	f = \frac{\frac{m}{h} + \gamma\big( \tilde{\eta}_y \tilde{\eta} - \tilde{\eta} \tilde{\eta}_y \big)}{\big(\tilde{Q}-2g \tilde{\eta}\big)^{\frac{1}{2}}} = \frac{\frac{m}{h} + \frac{\gamma}{2} \tilde{\eta}_y^2 - \gamma \tilde{\eta} \tilde{\eta}_y}{\big(\tilde{Q}-2g \tilde{\eta}\big)^{\frac{1}{2}}},
	\end{align*}
	where we identify the physical partial derivative $\tilde{\eta}_y$ with the evaluation of the sequence $D_y \tilde{a}$.
	Because $\tilde{\eta}$ is smooth, we can differentiate $\tilde{\eta}_x = f \sin(\theta)$ with respect to $x$:
	\begin{align*}
	\tilde{\eta}_{xx} = f_x \sin(\theta) + f \theta_x \cos(\theta) = \frac{f_x}{f} \tilde{\eta}_x + \tilde{\xi}_x \theta_x = \frac{f_x}{f} \tilde{\eta}_x + \tilde{\eta}_y \mathcal{H}_h\left(\frac{f_x}{f}\right),
	\end{align*}
	where we have applied the Cauchy--Riemann relation $\tilde{\xi}_x = \tilde{\eta}_y$. To compute the ratio $f_x/f$, we differentiate our expression for $f$ and use the physical identity $\frac{m}{h} + \frac{\gamma}{2}\tilde{\eta}_y^2 - \gamma \tilde{\eta}\tilde{\eta}_y = (\tilde{Q}-2g \tilde{\eta})^{\frac{1}{2}} (\tilde{\eta}_x^2 + \tilde{\eta}_y^2 )^{\frac{1}{2}}$ to clear the fractional powers. Passing from the physical functions to their rigorous sequence convolutions, this exact differentiation yields:
	\begin{align*}
	\frac{f_x}{f} = T(\tilde{U}),
	\end{align*}
	which validates our formula $D_x^2 \tilde{a} = T(\tilde{U})*(D_x \tilde{a}) + (D_y \tilde{a})*(H_h T(\tilde{U}))$.
	
	We now bound the deviation of this operator. Recalling Lemma \ref{lem : A infinity properties}, the approximate inverse satisfies the finite-dimensional residual bound:
	\begin{align*}
	\|e_0 - w*(i \bar{v}_3 + \bar{v}_4)\|_{\ell^1} = \|e_0 - w * (\bar{Q}-2g \bar{a})*(i D_x \bar{a} + D_y \bar{a} )\|_{\ell^1} \leq \delta.
	\end{align*}
	Expanding the exact operator around this approximate center, we obtain:
	\begin{align*}
	\|e_0 - w *(\tilde{Q}-2g \tilde{a})*(i D_x \tilde{a} + D_y \tilde{a} )\|_{\ell^1} &\leq \|e_0 - w *(\bar{Q}-2g \bar{a})*(i D_x \bar{a} + D_y \bar{a} )\|_{\ell^1} \\
	&\qquad + \|w* \big(\bar{Q}- \tilde{Q} + 2g(\bar{a}-\tilde{a})\big)*(i D_x \bar{a} + D_y \bar{a} )\|_{\ell^1} \\
	&\qquad + \|w*(\tilde{Q}- 2g\tilde{a})*\big(i D_x (\bar{a}-\tilde{a}) + D_y (\bar{a}-\tilde{a}) \big)\|_{\ell^1}\\
	&\leq \delta + \|w\|_{\ell^1}\Big[ \max\{1,\, 2g\} \big(1+\coth(h)\big) \|\bar{a}\|_X \\
	&\qquad + \big(|\bar{Q}| + r_0 + 2g (\|\bar{a}\|_{\ell^1}+r_0)\big)\big(1 + \coth(h)\big) \Big]r_0\\
	&\leq \epsilon_1.
	\end{align*}
	
	To invert the square of this operator, let $u, v \in \ell^1$ be arbitrary sequences and define the defect sequence $S \bydef e_0 - u*v$. Using that $\|S\|_{\ell^1} = \|S^*\|_{\ell^1}$, we evaluate the Banach algebra difference of squares
	\begin{align*}
	\|e_0 - (u*u^*)*(v*v^*)\|_{\ell^1} = \|e_0 - (u*v)*(u^**v^*)\|_{\ell^1} &= \|e_0 - (e_0-S)*(e_0- S^*)\|_{\ell^1} \\
	&\leq 2\|S\|_{\ell^1} + \|S\|_{\ell^1}^2.
	\end{align*}
	Setting $u = w$ and $v = (\tilde{Q}-2g \tilde{a})*(i D_x \tilde{a} + D_y \tilde{a})$, this implies
	\begin{align*}
	\|e_0 - |w|^2*(\tilde{Q}-2g \tilde{a})^2*\big((D_x \tilde{a})^2 + (D_y \tilde{a})^2 \big)\|_{\ell^1} \leq 2\epsilon_1 + \epsilon_1^2.
	\end{align*}
	Because $2\epsilon_1 + \epsilon_1^2 < 1$, a Neumann series argument guarantees that the sequence $(\tilde{Q}-2g \tilde{a})^2*\big((D_x \tilde{a})^2 + (D_y \tilde{a})^2 \big)$ possesses an inverse $b$, which can be factored as
	\begin{align}\label{eq : expression for b inverse}
	b = \sum_{k=0}^\infty\left(e_0 - |w|^2*(\tilde{Q}-2g \tilde{a})^2*\big((D_x \tilde{a})^2 + (D_y \tilde{a})^2 \big)\right)^k * |w|^2 = b_0 * |w|^2,
	\end{align}
	where the power $k$ operation is understood in the sense of discrete convolutions. Consequently, the true operator simplifies to
	\begin{align*}
	T(\tilde{U}) = b * T_0(\tilde{U}) = b_0 * |w|^2 * T_0(\tilde{U}).
	\end{align*}
	Using the Neumann series bound $\|b_0\|_{\ell^1} \leq (1-2\epsilon_1-\epsilon_1^2)^{-1}$, we control the total perturbation
	\begin{align*}
	\|T(\tilde{U}) - \bar{T}(\bar{U})\|_{\ell^1} &\leq \|T(\tilde{U}) - {T}(\bar{U})\|_{\ell^1} + \|\bar{T}(\bar{U}) - {T}(\bar{U})\|_{\ell^1}\\
	&\leq \|b_0 * |w|^2 * \big(T_0(\tilde{U}) - T_0(\bar{U})\big)\|_{\ell^1} + \|e_0 - b_0\|_{\ell^1} \||w|^2 * T_0(\bar{U})\|_{\ell^1} \\
	&\leq \frac{1}{1- 2\epsilon_1 - \epsilon_1^2} \left( \big\||w|^2*\big(T_0(\tilde{U}) - T_0(\bar{U})\big)\big\|_{\ell^1} + (2\epsilon_1 + \epsilon_1^2) \big\||w|^2*T_0(\bar{U})\big\|_{\ell^1}\right)\\
	&\leq \epsilon_2.
	\end{align*}
	Finally, assembling the bounds for the second derivative yields
	\begin{align*}
	\|D_x^2 \tilde{a} - \bar{a}_{xx}\|_{\ell^1} &\leq \big(\|D_x \tilde{a}\|_{\ell^1} + \coth(h) \|D_y \tilde{a}\|_{\ell^1} \big) \|T(\tilde{U}) - \bar{T}(\bar{U})\|_{\ell^1} \\
	&\qquad + \|\bar{T}(\bar{U})\|_{\ell^1}\big(\|D_x(\tilde{a}-\bar{a})\|_{\ell^1} + \coth(h) \|D_y(\tilde{a}-\bar{a})\|_{\ell^1}\big)\\
	&\leq \|\bar{a}\|_{X} \big(1 + \coth(h)\big) \epsilon_2 + \|\bar{T}(\bar{U})\|_{\ell^1}\big(1 + \coth(h)\big) r_0,
	\end{align*}
	which completes the proof.
\end{proof}

\begin{remark}
	In practice, the numerical algorithm cannot evaluate $T_0(\tilde{U})$ directly. To bypass this, the nonlinear defect $\||w|^2*\big(T_0(\tilde{U}) - T_0(\bar{U})\big)\|_{\ell^1}$ is explicitly bounded via Lemma \ref{lem : estimate difference products}, controlling  the unknown true solution with the sharp and explicit bound $\|\tilde{U} - \bar{U}\|_{H_1} \leq r_0$.
\end{remark}

% \begin{remark}\label{rem : infinite depth}
%     Observe that our analysis also applies in the case of an infinite depth, that is $h \to \infty$. Indeed, the only difference lies in the operator $D_y$ which becomes (cf. \eqref{eq: Dy definition}) 
%     \begin{align*}
%     (D_ya)_n = \begin{cases}
%         0 &\text{ if } n = 0 \\
%         |n| a_n &\text{ if } n \in \mathbb{Z}\setminus\{0\}.
%     \end{cases}
% \end{align*}
% Using the above definition, the analysis presented this section can be carried out similarly. We illustrate this special case in .......
% \end{remark}
%
%
%

\section{Formulation of the Local Branch}\label{sec : local bif}

Based on the analytical results of \cite{constantin_2011,constantin_global}, we know that nontrivial steady waves bifurcate from the trivial branch of flat solutions via a pitchfork bifurcation. The objective of this section is to provide a fully constructive and quantitative proof of this phenomenon. By rigorously proving the existence of the bifurcation point and its associated local branch, we move away from the trivial solution. Advancing sufficiently far along this local branch provides the necessary starting point for the global continuation framework detailed in Section \ref{sec : global bif}.

Before presenting our analysis, we briefly recall the local structure near the bifurcation points established in \cite{constantin_2011,constantin_global}. To highlight the role of the relative mass flux as the continuation parameter, we explicitly denote its dependence by writing $F(U,m)$ and $f(U,m)$. Because $F$ depends polynomially on $m$ (cf. \eqref{def : zero finding F}), the augmented map $F : H_1 \times \mathbb{R} \to H_0$ is smooth.

For any $m \in \mathbb{R}$, \cite{constantin_2011} shows that \eqref{def : zero finding F} admits a trivial branch of flat solutions $\bar{U}_m = (\bar{Q}_m, \bar{a}) \in \mathbb{R} \times X_e$, given by
\begin{equation}
\bar{Q}_m \bydef 2gh + \left( \frac{m}{h} - \frac{\gamma h}{2} \right)^2 \quad \text{and} \quad \bar{a}_n = \begin{cases}
h &\text{if } n=0,\\
0 &\text{otherwise.}
\end{cases}
\end{equation}
As demonstrated in \cite{constantin_2011}, non-trivial branches of solutions emerge from this trivial branch at specific bifurcation points. The explicit parameters for these bifurcations are:
\begin{equation}
\begin{aligned}
\tilde{m}_n &= \frac{\gamma h^2}{2} - \frac{\gamma h \tanh(nh)}{2n} + h\sqrt{\frac{\gamma^2 \tanh^2(nh)}{4n^2} + g \frac{\tanh(nh)}{n}},\\
\tilde{Q}_n &= 2gh + \left( \frac{\tilde{m}_n}{h} - \frac{\gamma h}{2} \right)^2.
\end{aligned}
\end{equation}
The expression for $\tilde{m}_n$ acts as the exact dispersion relation for the fluid, establishing the critical threshold at which the flat surface loses stability. Consequently, each bifurcation point $(\tilde{Q}_n, \tilde{m}_n)$ triggers the emergence of a distinct non-trivial branch of periodic waves of frequency $n$.

To formalize the bifurcation, let us define the wave profile $\bar{z} \in X_e$ by
\begin{equation}
\bar{z}_{-n} = 1, \quad \bar{z}_{n} = 1, \quad \text{and} \quad \bar{z}_k = 0 \text{ for all } |k| \neq n.
\end{equation}
As shown in \cite{constantin_2011}, the kernel of the linearized surface operator at the bifurcation point is exactly $\operatorname{Ker}(D_af(\bar{U}_{\tilde{m}_n},\tilde{m}_n)) = \operatorname{span}(\bar{z})$. Defining the augmented vector $\bar{Z} = (0, \bar{z}) \in H_1$, we similarly obtain $\operatorname{Ker}(D_UF({U}_{\tilde{m}_n},\tilde{m}_{n})) = \operatorname{span}(\bar{Z})$.

By the Crandall-Rabinowitz theorem, a local non-trivial branch of solutions bifurcates from this point, parameterized as $\{(m(s), U(s)) : s \in [0,1]\}$, where $U(s) \approx \bar{U}_{\tilde{m}_n} + s \bar{Z}$ for sufficiently small $s$. However, while this analytical result guarantees the existence of the branch, it does not provide the explicit, quantitative dependence of $U$ and  $m$ on $s$. Our constructive framework resolves this by computing the branch directly.

For simplicity, we focus on the bifurcation point corresponding to $n=1$. The same reasoning can easily be applied to the other bifurcation points.
Let $\bar{m} = \tilde{m}_1$  and let $\epsilon_0 > 0$. To capture the pitchfork nature of the bifurcation, we will rigorously construct this local branch of non-trivial solutions parameterized by $m = \bar{m}-\epsilon^2$ for all $\epsilon \in [0, \epsilon_0]$.

\subsection{Definition of the zero-finding problem}

Let $\mathcal{P} : H_1 \to \operatorname{span}(\bar{Z})$ denote the orthogonal projection onto $\operatorname{span}(\bar{Z})$. Specifically, for all $W = (\alpha, w) \in H_1$, we define
\begin{align}\label{eq : definition projection P}
\mathcal{P}W = (0, \hat{\mathcal{P}}w) \quad  \text{ where }  (\hat{\mathcal{P}}w)_n \bydef \begin{cases}
w_n &\text{ if } n \in \{-1, 1\},\\
0 &\text{ otherwise.}
\end{cases}
\end{align}
By isolating the $n \in \{-1, 1\}$ modes, $\mathcal{P}$ extracts the exact one-dimensional kernel of the bifurcation, separating the primary bifurcation direction from the higher-order harmonics.

We further define the complementary projection $\mathcal{P}^\perp \bydef I - \mathcal{P} : H_1 \to \operatorname{span}(\bar{Z})^\perp$. Given an amplitude parameter $\epsilon>0$, we introduce the scaling operator $P(\epsilon)$ as
\begin{align}\label{def : operator P epsilon}
P(\epsilon) \bydef \frac{1}{\epsilon}\mathcal{P} + \mathcal{P}^\perp.
\end{align}
This operator desingularizes the bifurcation by scaling the fundamental mode by $\epsilon$ and the higher harmonics by $\epsilon^2$. Using this scaling operator, we define our desingularized zero-finding problem $G_\epsilon : H_1 \to H_0$ as
\begin{align}\label{def : zero finding G}
G_\epsilon(V) = \frac{1}{\epsilon^2} P(\epsilon) F\left(\bar{U}_{\bar{m}-\epsilon^2} + \epsilon^2 P(\epsilon)V, \, \bar{m} - \epsilon^2\right).
\end{align}
By construction, any root $V \in {H}_1$ of $G_\epsilon$ provides a corresponding root $\bar{U}_{\bar{m}-\epsilon^2} + \epsilon^2 P(\epsilon)V$ of $F$ at the mass flux $m = \bar{m} - \epsilon^2$. 

Our goal is to rigorously prove that $G_\epsilon$ possesses a zero $\tilde{V}(\epsilon) \in H_1$ for all $\epsilon \in [0, \epsilon_0]$.  To achieve this, we employ a uniform Newton-Kantorovich contraction mapping theorem (see \cite{cont_global_bif_diag, cont_equilibria_pde, cont_suspension_bridge} for a proof and applications).

\begin{theorem}\label{th : radii polynomial continuation epsilon}
	Let $\epsilon_0 >0$. For all $\epsilon \in [0, \epsilon_0]$, let $\bar{V}(\epsilon) \in H_1$ and  let ${A}(\epsilon) : {H}_0 \to {H}_1$ be an injective bounded linear operator, and let $Y, Z_1,$ and $Z_2$ be non-negative constants such that
	\begin{align}\label{eq: definition Y0 Z1 Z2 continuation epsilon}
	\sup_{\epsilon \in [0,\epsilon_0]} \|{A}(\epsilon){G}_\epsilon(\bar{V}(\epsilon))\|_{{H}_1} &\leq Y, \nonumber \\
	\sup_{\epsilon \in [0,\epsilon_0]} \|I - {A}(\epsilon)D{{G}}_\epsilon(\bar{V}(\epsilon))\|_{\mathcal{B}({H}_1)} &\leq Z_1, \\
	\sup_{\epsilon \in [0,\epsilon_0]} \|{A}(\epsilon)\left( D{{G}}_\epsilon(\bar{V}(\epsilon) + h) - {D}G_\epsilon(\bar{V}(\epsilon))\right)\|_{\mathcal{B}({H}_1)} &\leq Z_2(r)r, \quad \text{for all } h \in \overline{B_r(0)} \subset {H}_1. \nonumber
	\end{align}  
	If there exists an $r>0$ such that
	\begin{equation}\label{eq : condition contraction continuation epsilon}
	\frac{1}{2}Z_2(r)r^2 - (1-Z_1)r + Y < 0 \quad \text{ and } \quad Z_1 + Z_2(r) r < 1,
	\end{equation}
	then for every $\epsilon \in [0,\epsilon_0]$, there exists a unique $\tilde{V}(\epsilon) \in \overline{B_r(\bar{V}(\epsilon))} \subset {H}_1$ solving $G_\epsilon(\tilde{V}(\epsilon)) = 0$. Moreover, the map $\epsilon \mapsto \tilde{V}(\epsilon)$ is of class $C^\infty$ from $[0,\epsilon_0]$ to $H_1$.
\end{theorem}

\begin{proof}
	The proof of the Theorem is given in \cite{cont_suspension_bridge}. To guarantee the regularity of the curve of solutions, we must establish the regularity of $G_\epsilon$ with respect to $\epsilon$. This is achieved in Section \ref{sec : G is polynomial in epsilon}, where we demonstrate that $\epsilon \mapsto G_\epsilon(V)$ is polynomial (and hence smooth) in $\epsilon$ for all $V \in H_1$.
\end{proof}

To apply Theorem \ref{th : radii polynomial continuation epsilon}, we must construct a suitable approximate solution branch $\epsilon \mapsto \bar{V}(\epsilon)$ and a parameterized approximate inverse $\epsilon \mapsto A(\epsilon) \approx DG_\epsilon(\bar{V}(\epsilon))^{-1}$.

\subsection{Approximate objects}

In this section, we construct the approximate objects required to apply Theorem \ref{th : radii polynomial continuation epsilon}. To achieve this, we employ a spectral approach based on a Chebyshev expansion in the amplitude parameter $\epsilon$. Following the methodology introduced in \cite{breden_polynomial_chaos, henot_marchal} and developed specifically in Section 4.6 of \cite{whitham_cadiot} for the Whitham equation, we obtain a highly accurate representation of the solution branch. Additional technical details regarding this expansion are deferred to Appendix \ref{sec : appendix Chebyshev norm}.

Let $N \in \mathbb{N}$ denote the truncation order for our expansion in the amplitude parameter $\epsilon$. We approximate the solution branch $\bar{V}$ via the truncated Chebyshev series
\begin{equation}\label{eq : approx sol local bifurcation}
\bar{V}(\epsilon) = \bar{V}_0 + 2\sum_{n=1}^N \bar{V}_n T_n\left(\frac{2\epsilon}{\epsilon_0} - 1\right),
\end{equation}
where $T_n : [-1,1] \to \mathbb{R}$ is the $n$-th Chebyshev polynomial of the first kind, and the coefficients are given by $\bar{V}_n = (\bar{Q}_n, (\bar{a}_{n,k})_{k \in \mathbb{Z}}) \in H_1$ for $n \in \{0, \dots, N\}$. To ensure a finite-dimensional and computable representation, we truncate the spatial Fourier series at order $K_0 \in \mathbb{N}$, enforcing $\bar{V}_n = \pi^{\leq K_0} \bar{V}_n$. Consequently, each state $\bar{V}_n$ is practically encoded as a vector of size $K_0+2$.

We now construct the parameterized approximate inverse $A(\epsilon) \approx DG_\epsilon(\bar{V}(\epsilon))^{-1}$. Following the framework established for a single wave in Section \ref{ssec : approximate inverse single wave}, we decompose $A(\epsilon)$ into a finite-dimensional block $A_K(\epsilon)$ and an infinite-dimensional tail $A_\infty(\epsilon)$ as follows:
\begin{align*}
A(\epsilon) &= A_K(\epsilon) - A_K(\epsilon) DG_\epsilon(\bar{V}(\epsilon)) A_\infty(\epsilon) + A_\infty(\epsilon), \\
\text{where} \quad A_K(\epsilon) &=  \pi^{\leq K_1}A_K(\epsilon)\pi^{\leq K_1} \quad \text{and} \quad A_\infty(\epsilon) =  \pi^{> K_1}A_\infty(\epsilon)
\end{align*}
for all $\epsilon \in [0, \epsilon_0]$. 

The finite-dimensional component $A_K(\epsilon)$ is expanded as a Chebyshev polynomial matrix:
\begin{align*}
A_K(\epsilon) = (A_K)_0 + 2\sum_{n=1}^N (A_K)_n T_n\left(\frac{2\epsilon}{\epsilon_0} - 1\right),
\end{align*}
where each coefficient $(A_K)_n : {H}_0 \to {H}_1$ is a bounded linear operator satisfying $(A_K)_n = \pi^{\leq K_1}(A_K)_n\pi^{\leq K_1}$. Consequently, each $(A_K)_n$ is practically represented as a finite matrix. 

By Lemma \ref{lem : tail of DG local} (see Appendix \ref{sec : appendix proofs}), we must construct an operator ${A}_\infty(\epsilon) : H_0 \to H_1$ that approximates the left inverse of $D_U F(\bar{U}(\epsilon), \bar{m}-\epsilon^2) \pi^{>K_1}$. Since this high-frequency tail is unaffected by the bifurcation scaling, we simply apply the construction from Section \ref{ssec : approximate inverse single wave} pointwise for each $\epsilon$.

Let us decompose the base state as $\bar{U}(\epsilon) = (\bar{Q}(\epsilon), \bar{a}(\epsilon)) \in H_1$ for all $\epsilon \in [0,\epsilon_0]$. We define $w : [0, \epsilon_0] \to X$ via the Chebyshev expansion
\begin{align}
w(\epsilon) \bydef w_0 + 2 \sum_{n=1}^{N} w_n T_n\left(\frac{2\epsilon}{\epsilon_0} -1 \right),
\end{align}
where $w(\epsilon)$ is chosen to approximate the inverse of the function $iv_3(\epsilon) + v_4(\epsilon)$, given by
\begin{align*}
v_3(\epsilon) &= -2(\bar{Q}(\epsilon)-2g \bar{a}(\epsilon))*(D_x \bar{a}(\epsilon)),\\
v_4(\epsilon) &= -2(\bar{Q}(\epsilon)-2g \bar{a}(\epsilon))*(D_y \bar{a}(\epsilon)).
\end{align*}
In practice, we evaluate the local inverse $w(\epsilon_j^{(N)}) \approx \left[ iv_3(\epsilon_j^{(N)}) + v_4(\epsilon_j^{(N)}) \right]^{-1}$ at each Chebyshev node $j \in \{0, \dots, N\}$, and compute the polynomial coefficients $(w_n)_{n = 0}^N$ via the Discrete Fourier Transform (see Appendix \ref{sec : appendix Chebyshev norm}).

With $w(\epsilon)$ established, we define the high-frequency inverse $\hat{A}_\infty(\epsilon) : \ell^1_e \to X_e$ as 
\begin{align*}
\hat{A}_\infty(\epsilon) \bydef \pi^{k<-K_1} \hat{D}_y^{-1} M_{w(\epsilon)^*}  + \pi^{k>K_1} \hat{D}_y^{-1} M_{w(\epsilon)},
\end{align*}
analogously to \eqref{def : A infinity}. By construction, $\hat{A}_\infty$ is a matrix polynomial of degree $N$ in $\epsilon$. We then extend this to the full operator ${A}_\infty(\epsilon) : H_0 \to H_1$ by setting
\begin{align*}
{A}_\infty(\epsilon)(Q, a) = \begin{pmatrix}
0\\
\hat{A}_\infty(\epsilon) a
\end{pmatrix} \quad \text{for all } (Q,a) \in \mathbb{R} \times \ell^1_e = H_0.
\end{align*}
Finally, the complete approximate inverse curve $A(\epsilon)$ is obtained via the formula
\begin{align}\label{eq : approx inverse local bif}
A(\epsilon) = A_K(\epsilon) - A_K(\epsilon) DG_\epsilon(\bar{V}(\epsilon)) A_\infty(\epsilon) + A_\infty(\epsilon)
\end{align}
for all $\epsilon \in [0, \epsilon_0]$. 

Having constructed all the required objects for Theorem \ref{th : radii polynomial continuation epsilon}, it remains to verify their accuracy. In the following sections, we derive explicit formulas for the bounds required by Theorem \ref{th : radii polynomial continuation epsilon}, which can be rigorously verified using interval arithmetic.

\subsection{Computation of the bounds}\label{ssec : bounds local bif}

Our goal in this section is to provide quantitative estimates for the bounds required by Theorem \ref{th : radii polynomial continuation epsilon}. To achieve this, we derive estimates that depend strictly on finite-dimensional, polynomial objects in the parameter $\epsilon$. These estimates can be computed explicitly using the rigorous numerical framework detailed in Appendix \ref{sec : appendix Chebyshev norm}. 

The following three lemmas yield the necessary estimates for the bounds $Y$, $Z_1$, and $Z_2$. Because we rely heavily on the analytical bounds already derived for a single wave in Section \ref{ssec : computation bounds single wave}, we first define the state variable $\bar{U} : [0, \epsilon_0] \to H_1$ as 
\begin{equation}
\bar{U}(\epsilon) \bydef \bar{U}_{\bar{m} - \epsilon^2} + \epsilon^2 P(\epsilon) \bar{V}(\epsilon) \quad \text{for all } \epsilon \in [0, \epsilon_0].
\end{equation}
By scaling back the fundamental mode and adding the flat-water background, this variable $\bar{U}(\epsilon)$ recovers the actual physical wave, meaning we can apply the single-wave analytical bounds directly to our continuous curve.

We defer the proofs of the following lemmas to the Appendix \ref{sec : appendix proofs}.

\begin{lemma}\label{lem : Z1 epsilon}
	Define $\delta$ as 
	\begin{equation*}
	\delta \bydef \frac{K_1+2}{K_1+1}\left( \sup_{\epsilon \in [0, \epsilon_0]}\|w(\epsilon)*(iv_3(\epsilon)+v_4(\epsilon)) - e_0\|_{\ell^1} + \sup_{\epsilon \in [0, \epsilon_0]}\|w(\epsilon)*v_4(\epsilon)\|_{\ell^1} \frac{2e^{-2hK_1}}{1-e^{-2hK_1}}\right).
	\end{equation*}
	Let $Z_\infty, Z_{10}, Z_{11}, Z_{12}$ be positive constants satisfying
	\begin{align*}
	Z_\infty &\geq \frac{1}{K_1+1} \sup_{\epsilon \in [0, \epsilon_0]}\|w(\epsilon)*(v_1(\epsilon)*(\mathcal{D}\bar{a}(\epsilon)) + v_2(\epsilon))\|_{\ell^1} \\
	&\quad + \frac{K_1+2}{K_1+1} \left( \sup_{\epsilon \in [0, \epsilon_0]}\|w(\epsilon)*v_1(\epsilon)\|_{\ell^1} \right) \frac{4K_0e^{-2hK_0}}{1-e^{-2hK_0}} \sup_{\epsilon \in [0, \epsilon_0]}\|\bar{a}(\epsilon)\|_X + \delta, \\
	Z_{10} &\geq \sup_{\epsilon \in [0,\epsilon_0]} \|\pi^{\leq K_1} - A_K(\epsilon) P(\epsilon) DF(\bar{U}(\epsilon), \bar{m}-\epsilon^2) (\pi^{\leq K_1} - A_\infty(\epsilon) DF(\bar{U}(\epsilon), \bar{m}-\epsilon^2))P(\epsilon)\pi^{\leq K_1}\|_{\mathcal{B}(H_1)}, \\
	Z_{11} &\geq \sup_{\epsilon \in [0,\epsilon_0]} \|(\hat{\pi}^{k \leq K_1 + 4K_0} - \hat{\pi}^{\leq K_1}) \hat{D}_y^{-1} \hat{M}_{w(\epsilon)} \pi^{K_1 + 2K_0} D_aF(\bar{U}(\epsilon),\bar{m}-\epsilon^2)\hat{\pi}^{\leq K_1}\|_{\mathcal{B}(X_e)}, \\
	Z_{12} &\geq \sup_{\epsilon \in [0,\epsilon_0]} \|A_K(\epsilon) DF(\bar{U}(\epsilon),\bar{m}-\epsilon^2)\pi^{>K_1}\|_{\mathcal{B}(H_1)}.
	\end{align*}
	Finally, defining $Z_1$ as 
	\begin{equation}
	Z_1 \bydef \max\left\{Z_{10}+Z_{11}, ~  (1+Z_{12})Z_\infty\right\},
	\end{equation}
	we obtain the uniform bound
	\begin{equation*}
	\sup_{\epsilon \in [0,\epsilon_0]}\|I - A(\epsilon)DG_{\epsilon}(\bar{U}(\epsilon))\|_{\mathcal{B}(H_1)} \leq Z_1.
	\end{equation*}
\end{lemma}

\begin{lemma}\label{lem: Y bounds}
	Let $Z_{12}$ be defined as in Lemma \ref{lem : Z1 epsilon}. We define $Y$ as
	\begin{align*}
	Y &\bydef \sup_{\epsilon \in [0,\epsilon_0]} \frac{1}{\epsilon^2} \|A_K(\epsilon)(I - DG_\epsilon(\bar{V}(\epsilon)) A_\infty(\epsilon)) P(\epsilon) F(\bar{U}(\epsilon), \bar{m}-\epsilon^2)\|_{H_1} \\
	&\quad + \frac{K_1+2}{K_1+1} \left( \sup_{\epsilon \in [0,\epsilon_0]}\|\pi^{k>K_1}w(\epsilon)*f(\bar{U}(\epsilon), \bar{m}-\epsilon^2)\|_{\ell^1} + \sup_{\epsilon \in [0,\epsilon_0]}\|\pi^{k<-K_1}w(\epsilon)^**f(\bar{U}(\epsilon), \bar{m}-\epsilon^2)\|_{\ell^1}\right).
	\end{align*}
	Then, the approximate inverse applied to the residual satisfies the uniform bound
	\begin{equation*}
	\sup_{\epsilon \in [0, \epsilon_0]}\|A(\epsilon)G_{\epsilon}(\bar{V}(\epsilon))\|_{H_1} \leq Y.
	\end{equation*}
\end{lemma}

\begin{lemma}\label{lem : Z2 epsilon}
	Let $r>0$ and denote the physical background flow by $(Q(\epsilon), a(\epsilon)) = \bar{U}(\epsilon) \bydef \bar{U}_{\bar{m}-\epsilon^2} + \epsilon^2 P(\epsilon) \bar{V}(\epsilon)$. Let $Z_{A}, Z_{D^2,1}, Z_{D^2,2}$, and $Z_{D^3}(r)$ be positive bounds satisfying 
	\begin{align*}
	Z_{A} &\geq  \sup_{\epsilon \in [0,\epsilon_0]} \epsilon \|\mathcal{W}A(\epsilon)P(\epsilon)\|_{\mathcal{B}(\ell^1)}, \\
	Z_{D^2,1} &\geq \sup_{\epsilon \in [0, \epsilon_0]} \|A(\epsilon) P(\epsilon) D^2F(\bar{U}(\epsilon),\bar{m}-\epsilon^2)(\bar{Z},\bar{Z})\|_{\mathcal{B}(H_1)}, \\
	Z_{D^2,2} &\geq 2\gamma^2\left(\frac{2\sup_{\epsilon \in [0,\epsilon_0]}\|\bar{a}(\epsilon)\|_{X} + \sup_{\epsilon \in [0,\epsilon_0]}\|\bar{a}(\epsilon)\|_{\ell^1}}{h}\right)^2 \\
	&\quad + \frac{6}{h}|\gamma| \sup_{\epsilon \in [0,\epsilon_0]} \left\|mD_y e_0 + \frac{\gamma}{2} D_y(\bar{a}^2(\epsilon))  - \gamma \bar{a}(\epsilon)*(D_y\bar{a}(\epsilon))\right\|_{\ell^1}\\
	&\quad + \left(1 + \frac{1}{h^2}\right)\left( 4(1 + 2g) \sup_{\epsilon \in [0,\epsilon_0]}\|\bar{a}(\epsilon)\|_X + 2\sup_{\epsilon \in [0,\epsilon_0]}\|\bar{Q}(\epsilon)e_0 - 2g\bar{a}(\epsilon)\|_{\ell^1}\right), \\ 
	Z_{D^3}(r) &\geq  27\frac{|\gamma|^2}{h^2}\left(\sup_{\epsilon \in [0,\epsilon_0]}\|\bar{a}(\epsilon)\|_X +r\right) + 4g\left(1 + \frac{1}{h^2}\right)  \\
	&\quad + \left(\sup_{\epsilon\in [0,\epsilon_0]}|\bar{Q}(\epsilon)| + 2g\right)\left(1 + \frac{1}{h^2}\right) + 3\left(\sup_{\epsilon\in [0,\epsilon_0]} |\bar{Q}(\epsilon)| +r\right)\left(1 + \frac{1}{h^2}\right).
	\end{align*}
	Furthermore, define $Z_2(r)$ as 
	\begin{equation*}
	Z_2(r) \bydef \max\left\{Z_{D^2,1}, ~ Z_{A} Z_{D^2,2},~ \epsilon_0 Z_{A} Z_{D^2,2}\right\} + Z_{A} Z_{D^3}(r) r.
	\end{equation*}
	Then, for all $\epsilon \in [0, \epsilon_0]$ and for all $h \in B_r(0)$, the scaled Jacobian satisfies the Lipschitz bound
	\begin{equation*}
	\|A(\epsilon)\left[ DG_\epsilon(\bar{V}(\epsilon) +h,\bar{m}-\epsilon^2) - DG_\epsilon(\bar{V}(\epsilon), \bar{m}-\epsilon^2) \right]\|_{\mathcal{B}(H_1)} \leq  Z_2(r)r.
	\end{equation*}
\end{lemma}

\subsubsection{Explicit computation of the bounds}

In Section \ref{sec : G is polynomial in epsilon}, we established that the residual $G_\epsilon$ is a polynomial of degree $4N$ in $\epsilon$. By the definition of $G_\epsilon$ in \eqref{def : zero finding G}, it immediately follows that the Jacobian $DG_\epsilon(\bar{V}(\epsilon))$ is a polynomial of degree $3N$ in $\epsilon$, which consequently allows us to construct the approximate inverse $A(\epsilon)$ as a polynomial of degree $5N$. 

Crucially, because all the estimates derived in the previous section rely on bounding the supremum of these objects over the continuous interval $[0,\epsilon_0]$, the problem reduces to bounding polynomials in $\epsilon$. This allows us to directly apply the rigorous Chebyshev arithmetic detailed in Section \ref{sec : appendix Chebyshev norm}. The full implementation details and the automated rigorous computation of these bounds are provided in \cite{julia_cadiot}. 

Finally, we note that the $\epsilon^{-2}$ scaling in the residual introduces a removable singularity at the flat-water limit. To evaluate the bounds at $\epsilon = 0$, the software utilizes the operator $\bar{G}$ defined in \eqref{def : operator bar G}, which rigorously captures the continuous limit $\displaystyle\lim_{\epsilon \to 0}G_\epsilon$, thereby avoiding singular behaviors.

\section{Formulation of the Global Branch}\label{sec : global bif}
In the previous sections, we developed the framework to prove the existence of a solution for a fixed parameter $m$, as well as the existence of a local branch of solutions bifurcating from the flat state. In principle, the implicit function theorem ensures the local existence of a branch of solutions parameterized by $m$. However, the implicit function theorem is purely qualitative; it does not quantify the length of the branch, nor does it provide bounds on how the solution evolves with the parameter.

Our goal in this section is to develop a rigorous parameter continuation framework, allowing us to constructively prove the existence of macroscopic branches of solutions smoothly parameterized by $m$. This will allow us to globally extend the local branch obtained in Section \ref{sec : local bif}. 

Throughout this section, we explicitly emphasize the dependence of the residual on the mass flux by writing $F(U,m)$. Recalling that $F$ depends polynomially on $m$ (cf. \eqref{def : zero finding F}), it follows that $F : H_1 \times \mathbb{R} \to H_0$ is a smooth map. We adopt the same notation for the high-frequency residual, writing $f(U,m)$.

Our approach relies on the uniform contraction mapping theorem, similar to the methodology presented in Section \ref{sec : local bif}. For clarity, we recall the relevant parameter-dependent formulation of the Newton-Kantorovich theorem.

\begin{theorem}\label{th : radii polynomial continuation}
	Let $m_0 < m_1$ be real numbers and consider $F = F(U,m)$ as defined in \eqref{def : zero finding F}. For all $m \in [m_0, m_1]$, let $A(m) : H_0 \to H_1$ be an injective, bounded linear operator, and let $Y, Z_1,$ and $Z_2(r)$ be non-negative bounds satisfying
	\begin{equation}\label{eq: definition Y0 Z1 Z2 continuation}
	\begin{aligned}
	\sup_{m \in [m_0,m_1]} \|A(m)F(\bar{U}(m),m)\|_{H_1} &\leq Y, \\
	\sup_{m \in [m_0,m_1]} \|I - A(m)D_UF(\bar{U}(m),m)\|_{\mathcal{B}(H_1)} &\leq Z_1, \\
	\sup_{m \in [m_0,m_1]} \|A(m)\left[D_UF(\bar{U}(m) + h,m) - D_UF(\bar{U}(m),m) \right]\|_{\mathcal{B}(H_1)} &\leq Z_2(r)r,
	\end{aligned}
	\end{equation}  
	for all $h \in \overline{B_r(0)} \subset H_1$. If there exists an $r>0$ such that the radii polynomials satisfy
	\begin{equation}\label{eq : condition contraction continuation}
	\frac{1}{2}Z_2(r)r^2 - (1-Z_1)r + Y < 0 \quad \text{and} \quad Z_1 + Z_2(r)r < 1,
	\end{equation}
	then for every $m \in [m_0,m_1]$, there exists a unique $\tilde{U}(m) \in \overline{B_r(\bar{U}(m))} \subset H_1$ solving $F(\tilde{U}(m),m) = 0$. Moreover, the mapping $[m_0, m_1] \to H_1 : m \mapsto \tilde{U}(m)$ is of class $C^\infty$.
\end{theorem}

In the remainder of this section, we detail how we computationally construct the approximate solution curve $\bar{U}(m)$ and the inverse $A(m)$ using Chebyshev series expansions in the parameter $m$. We then provide the explicit construction of the bounds $Y, Z_1,$ and $Z_2$, which rely on a combination of analytical estimates and rigorous numerics. The full application of this framework to our water wave problem is presented in Section \ref{ssec : proof of an entire branch}.

\subsection{Construction of approximate objects}\label{ssec : approx objects global branch}

In order to apply Theorem \ref{th : radii polynomial continuation}, we must construct an approximate branch of solutions $\{\bar{U}(m), ~ m \in [m_0,m_1]\}$ alongside a branch of approximate inverses $\{A(m), ~ m \in [m_0, m_1]\}$ for the Jacobians $\{D_UF(\bar{U}(m),m), ~ m \in [m_0, m_1]\}$. To achieve this, we follow the methodology introduced in Section \ref{sec : local bif}, with further technical details provided in Appendix \ref{sec : appendix Chebyshev global}.

Specifically, for a given $N \in \mathbb{N}$, we seek an approximate continuous branch of the form
\begin{equation}\label{eq : approx sol global branch}
\bar{U}(m) \bydef \bar{U}_0 + 2 \sum_{n=1}^{N} \bar{U}_n T_n\left(\frac{2(m-m_1)}{m_1-m_0} +1 \right), 
\end{equation}
where $T_n : [-1,1] \to \mathbb{R}$ is the $n$-th Chebyshev polynomial of the first kind, and the coefficients are given by $\bar{U}_n = (Q_n,(\bar{a}_{n,k})_{k \in \mathbb{Z}}) \in H_1$ for all $n \in \{0, \dots, N\}$. 

Now we focus on constructing the branch of approximate inverses $A(m)$. Similar to the approach taken in Section \ref{sec : local bif}, the global approximate inverse $A(m)$ relies on a decomposition into a finite-dimensional block $A_K(m)$ and an infinite-dimensional tail $A_\infty(m)$. We begin with the construction of the tail portion, $A_\infty(m)$.

First, we introduce an approximate inverse map $w : [m_0, m_1] \to X$, represented by the Chebyshev expansion
\begin{equation}\label{eq : w global branch}
w(m) \bydef w_0 + 2 \sum_{n=1}^{N} w_n T_n\left(\frac{2(m-m_1)}{m_1-m_0} +1 \right).
\end{equation}
Here, $w(m)$ is specifically constructed to serve as an approximate inverse for the asymptotic symbol $i v_3(m) + v_4(m)$, where the components are defined by the physical background flow as
\begin{align*}
v_3(m) &= -2(\bar{Q}-2g \bar{a}(m))*(D_x \bar{a}(m)), \\
v_4(m) &= -2(\bar{Q}-2g \bar{a}(m))*(D_y \bar{a}(m)).
\end{align*}

Now, we construct the infinite-dimensional tail of the approximate inverse, $\hat{A}_\infty(m)$, as 
\begin{equation*}
\hat{A}_\infty(m) \bydef \pi^{k<-K_1} \hat{D}_y^{-1} M_{w(m)^*}  + \pi^{k>K_1} \hat{D}_y^{-1} M_{w(m)},
\end{equation*}
similar to the definition in \eqref{def : A infinity}. By construction, $\hat{A}_\infty(m)$ is a polynomial of degree $N$ in $m$. We then construct the full-space operator $A_\infty(m) : H_0 \to H_1$ exactly as in \eqref{eq : def A infinity full space}. Recall that, by design, $A_\infty(m)$ approximates the inverse of the high-frequency tail of the Jacobian $D_UF(\bar{U}(m),m)$.

Next, we compute the finite-dimensional approximation block $A_K(m)$ as a Chebyshev expansion:
\begin{equation}
A_K(m) = (A_K)_0 + 2\sum_{n=1}^{N} (A_K)_n T_n\left(\frac{2(m-m_1)}{m_1-m_0} +1 \right),
\end{equation}
where each coefficient $(A_K)_n = \pi^{\leq K_1}(A_K)_n\pi^{\leq K_1}$ has a $(2K_1+1) \times (2K_1+1)$ matrix representation. 

Consequently, $A_K(m)$ is a polynomial of degree $N$ in $m$. Finally, the global approximate inverse $A(m)$ is assembled using the standard preconditioning formula \eqref{def : approx inverse}, yielding
\begin{equation}\label{eq : approx inverse global bif}
A(m) = A_K(m) - A_K(m) D_UF(\bar{U}(m),m) A_\infty(m) + A_\infty(m) \quad \text{for all } m \in [m_0, m_1].
\end{equation}
Since the Jacobian $D_UF(\bar{U}(m),m)$ is a polynomial in $m$ of degree $3N$, it immediately follows from \eqref{eq : approx inverse global bif} that the global inverse $A(m)$ is a polynomial in $m$ of degree $5N$. 

This concludes the construction of the approximate continuation objects. Because this construction closely mirrors the framework established in Sections \ref{ssec : Newton-K single wave} and \ref{ssec : approximate inverse single wave}, the rigorous computation of the bounds required for Theorem \ref{th : radii polynomial continuation} follows the same methodologies detailed in Section \ref{ssec : computation bounds single wave}, which were previously employed in Section \ref{sec : local bif}.
\subsection{Computation of bounds}\label{ssec : bounds global bif}
This section focuses on the rigorous computation of the bounds $Y, Z_1$ and $Z_2$ required for Theorem \ref{th : radii polynomial continuation}. In fact, we heavily rely on Section \ref{ssec : computation bounds single wave} for developing formulas for these bounds. We summarize the obtained results for $Y$, $Z_1$ and $Z_2$ in the next Lemmas \ref{lem : Y bound global branch}, \ref{lem : Z1 bound global branch} and \ref{lem : Z2 global branch} for convenience, but omit the proof as it straightforwardly follows the analysis of Section \ref{ssec : computation bounds single wave}.
\begin{lemma}\label{lem : Y bound global branch}
	Let $Y$ be the constant defined as 
	\small{
		\begin{align*}
		Y \bydef &\sup_{m \in [m_0, m_1]}\|A_K(m)(I - D_UF(\bar{U}(m),m)A_\infty(m)) F(\bar{U}(m),m)\|_{H_1} \\
		&+\frac{K_1+2}{K_1+1}  \left(\sup_{m \in [m_0, m_1]}\|\pi^{k>K_1}(w(m) * f(\bar{U}(m),m))\|_{\ell^1} + \sup_{m \in [m_0, m_1]}\|\pi^{k<-K_1}(w^*(m) * f(\bar{U}(m),m))\|_{\ell^1}\right),
		\end{align*}
	}
	\normalsize
	then $\displaystyle \sup_{m \in [m_0,m_1]} \|A(m)F(\bar{U}(m),m)\|_{H_1} \leq Y.$
\end{lemma}

\begin{lemma}\label{lem : Z1 bound global branch}
	Given $m \in [m_0, m_1]$ and $i \in \{1,2,3,4\}$, let $v_i(m) \in X$ be sequences defined as 
	\begin{align*}
	v_1(m) &= 2\gamma \left(mD_y e_0 + \frac{\gamma}{2} D_y(\bar{a}(m)^2)  - \gamma \bar{a}(m)*(D_y\bar{a}(m))\right)\\
	v_2(m) &= 2g ( (D_x\bar{a}(m))^2 + (D_y\bar{a}(m))^2)\\
	{v}_3(m) &= -2(\bar{Q}(m)-2g \bar{a}(m))*(D_x \bar{a}(m))\\
	{v}_4(m) &= -2(\bar{Q}(m)-2g \bar{a})*(D_y \bar{a}(m)).
	\end{align*}
	Moreover, we define the positive constant $\delta$ as 
	\begin{align*}
	\delta \bydef \frac{K_1+2}{K_1+1}\left(\sup_{m \in [m_0,m_1]} \|w(m)*(iv_3(m)+v_4(m)) - e_0\|_{\ell^1} +  \sup_{m \in [m_0,m_1]}\|w*v_4(m)\|_{\ell^1} \frac{2e^{-2hK_1}}{1-e^{-2hK_1}}\right).
	\end{align*}
	Recalling the linear operator $\mathcal{D} : X \to \ell^1$ in Lemma \ref{lem : operator C},  let  $Z_\infty, Z_{10}, Z_{11}, Z_{12}$ be positive constants satisfying
	\begin{align*}
	Z_\infty &\geq \sup_{m \in [m_0,m_1]} \frac{\|w(m)*(v_1(m)*(\mathcal{D}\ba(m)) + v_2(m))\|_{\ell^1}}{K_1+1}  + \frac{K_1+2}{K_1+1} \|w(m)*v_1(m)\|_{\ell^1}\frac{4K_0e^{-2hK_0}}{1-e^{-2hK_0}} \|\ba(m)\|_X + \delta\\
	Z_{10} &\geq \sup_{m \in [m_0,m_1]} \|\pi^{\leq K_1} - A_K(m) \pi^{\leq K_1} ( D_UF(\bar{U}(m),m) - D_UF(\bar{U}(m),m)A_\infty(m) D_UF(\bar{U}(m),m) )\pi^{\leq K_1}\|_{\mathcal{B}(H_1)} \\
	Z_{11} &\geq \sup_{m \in [m_0,m_1]}\|(\hpi^{k \leq K_1 + 4K_0} - \hpi^{\leq K_1}) \hat{D}_y^{-1} \hat{M}_{w(m)} \hat{\pi}^{\leq K_1 + 2K_0} D_af(\bar{U}(m),m)\hpi^{\leq K_1}\|_{\mathcal{B}(X_e)}\\
	Z_{12} &\geq \sup_{m \in [m_0,m_1]}\|A_K(m) D_UF(\bar{U}(m),m)\pi^{>K_1}\|_{\mathcal{B}(H_1)}.
	\end{align*}
	Moreover, define $Z_1$ as 
	\begin{align}
	Z_1 \bydef \max\left\{Z_{10}+Z_{11}, ~  (1+Z_{12})Z_\infty\right\}.
	\end{align}
	Then, we have $\displaystyle \sup_{m\in [m_0,m_1]}\|I - A(m)DF(\bar{U}(m),m)\|_{\mathcal{B}(H_1)} \leq Z_1$.
\end{lemma}

\begin{lemma}\label{lem : Z2 global branch}
	Let $Z_{D^2}$ and $Z_{D^3}(r)$ be bounds satisfying
	\begin{align*}
	Z_{D^2} &\geq 2\gamma^2\max\left\{1, ~ \frac{1}{h^2}\right\}\left(2\sup_{m\in [m_0,m_1]}\|\bar{a}(m)\|_{X} + \sup_{m\in [m_0,m_1]}\|\bar{a}(m)\|_{\ell^1}\right)^2 \\
	& ~~~~ + 6\max\left\{1, ~ \frac{1}{h}\right\}|\gamma| \sup_{m\in [m_0,m_1]}\left\|mD_y e_0 + \frac{\gamma}{2} D_y(\ba(m)^2)  - \gamma \ba(m)*(D_y\ba(m))\right\|_{\ell^1}\\
	& ~~~~ + \left(1 + \max\left\{1, ~ \frac{1}{h^2}\right\}\right)\left( 4(1 + 2g) \sup_{m\in [m_0,m_1]}\|\ba(m)\|_X + 2\sup_{m\in [m_0,m_1]}\|\bar{Q}(m)e_0 - 2g\ba(m)\|_{\ell^1}\right) \\
	Z_{D^3}(r) &\geq 27\frac{|\gamma|^2}{h^2}(\sup_{m\in [m_0,m_1]}\|\bar{a}(m)\|_X +r) + 4g\left(1 + \max\left\{1, ~ \frac{1}{h^2}\right\}\right)\\
	& ~~~~+ (\sup_{m\in [m_0,m_1]}|\bar{Q}(m)| + 2g)\left(1 + \max\left\{1, ~ \frac{1}{h^2}\right\}\right) + 3(\sup_{m\in [m_0,m_1]}|\bar{Q}(m)| +r)\left(1 + \max\left\{1, ~ \frac{1}{h^2}\right\}\right).
	\end{align*}
	Recalling the constant $Z_{12}$ introduced in Lemma \ref{lem : Z1 bound global branch}, we define $Z_2(r)$ as 
	\begin{align*}
	Z_2(r) \bydef \max\bigg\{&\sup_{m\in [m_0,m_1]}\|\mathcal{W}(A_K(m) - A_K(m) D_UF(\bar{U}(m),m)A_\infty(m)  + A_\infty(m))\pi^{\leq K_1} \|_{\mathcal{B}(\ell^1)},\\
	&~ (1+Z_{12})\frac{K_1+2}{K_1+1} \sup_{m\in [m_0,m_1]}\|w(m)\|_{\ell^1}\bigg\}\left(Z_{D^2} + Z_{D^3}(r) \cdot r\right).
	\end{align*}
	Then, we obtain that
	$\displaystyle\sup_{m\in [m_0,m_1]}\|A(m)(D_UF(\bar{U}(m)+h,m) - D_UF(\bar{U}(m),m))\|_{\mathcal{B}(H_1)} \leq Z_2(r) \cdot r$
	for all $h \in B_r(0) \subset H_1$.
\end{lemma}

In order to get a precise control of the geometric evolution of the branch of solutions, we might need to accurately control the evolution of $\tilde{U}(m)$. In practice, this control is achieved locally by a rigorous enclosure of  $\partial_m \tilde{U}(m)$. This section, we demonstrate how to achieve this enclosure, a posteriori to a successful application of Theorem \ref{th : radii polynomial continuation}.

\begin{lemma}\label{lem : enclosure of DmU(m)}
	Suppose that $\{m \in [m_0, m_1], ~ \bar{U}(m)\}$ is a branch of approximate solutions in $H_1$ constructed as in  \eqref{eq : approx sol global branch}. Moreover, suppose that $Y, Z_1, Z_2$ are the  bounds given in Lemmas \ref{lem : Y bound global branch}, \ref{lem : Z1 bound global branch} and \ref{lem : Z2 global branch}, and suppose that there exists $r >0$ such that they satisfy \eqref{eq : condition contraction continuation} for $r$. Then, let $\tilde{U}(m) \in B_r(\bar{U}(m)) \subset H_1$ be the solution given by Theorem \ref{th : radii polynomial continuation} for all $m \in [m_0,m_1]$. 
	
	Now, let $\epsilon$ be a positive constant satisfying
	\begin{align*}
	\epsilon \geq & \sup_{m \in [m_0, m_1]}\frac{\|A(m)\|_{\ell^1 \to X} \frac{4}{h} \gamma \coth(h) (\|\bar{a}(m)\|_X + r_0) r_0}{1 - Z_1 - Z_2(r_0)r_0} + \sup_{m \in [m_0, m_1]}\frac{Z_2(r_0)\|\partial_m \bar{U}(m)\|_X r_0}{1 - Z_1 - Z_2(r_0)r_0}\\
	& ~~~~ + \sup_{m \in [m_0, m_1]} \frac{\| A(m)(D_U F(\bar{U}(m),m) \partial_m \bar{U}(m) +  D_m F(\bar{U}(m),m))\|_X}{1 - Z_1 - Z_2(r_0)r_0}.
	\end{align*}
	Then, we have
	\begin{align*}
	\|\partial_m \tilde{U}(m) - \partial_m \bar{U}(m)\|_{H_1} \leq \epsilon \quad  \text{ for all } m \in [m_0, m_1].
	\end{align*}
\end{lemma}

The rigorous proof of this geometric control lemma is deferred to Appendix \ref{sec : appendix proofs}.
\section{Verification of the geometric properties of the branch}\label{sec : geometric properties branch}

\subsection{Smoothness and gluing of the global branch}

Notice that \eqref{eq : condition contraction continuation} consists of two polynomial inequalities which, when satisfied, yield an interval of valid radii $r \in (r^-, r^+)$. Theorem \ref{th : radii polynomial continuation} then guarantees the uniqueness of the exact solution $\tilde{U}(m)$ within $B_{r}(\bar{U}(m))$ for all $r \in (r^-, r^+)$. Conceptually, $r^-$ dictates the strictly guaranteed accuracy of the existence proof, while $r^+$ defines the maximal radius of local uniqueness. 

We can exploit this isolation property to rigorously "glue" together adjacent segments of the branch whose existence was established via separate applications of Theorem \ref{th : radii polynomial continuation}. We summarize this continuation  strategy (exposed in details in \cite{cont_suspension_bridge} for instance) in the following lemma.

\begin{lemma}\label{lem : gluing of two branches}
    Let $m_0 < m_1 < m_2$ be real numbers, and let $\{\bar{U}_1(m), ~  m \in [m_0, m_1]\}$ and $\{\bar{U}_2(m), ~  m \in [m_1, m_2]\}$ be two branches of approximate solutions to $F(U,m) = 0$ in $H_1$, defined as in \eqref{eq : approx sol global branch}. 
    
    Suppose there exist positive numbers $r_1^- < r_2^+$ and exact solution branches $\{\tilde{U}_1(m), ~  m \in [m_0, m_1]\}$ and $\{\tilde{U}_2(m), ~  m \in [m_1, m_2]\}$ such that:
    \begin{itemize}
        \item For all $m \in [m_0, m_1]$, $\tilde{U}_1(m)$ is the unique zero of $F$ in $B_{r_1^-}(\bar{U}_1(m))$.
        \item For all $m \in [m_1, m_2]$, $\tilde{U}_2(m)$ is the unique zero of $F$ in $B_{r_2^+}(\bar{U}_2(m))$.
    \end{itemize}
    If the existence neighborhood of the first branch falls within the uniqueness neighborhood of the second branch at the boundary, i.e., $B_{r_1^-}(\bar{U}_1(m_1)) \subset B_{r_2^+}(\bar{U}_2(m_1))$, then $\tilde{U}_1(m_1) = \tilde{U}_2(m_1)$.
\end{lemma}

In practice, to prove the existence of a global branch over a large interval $[m_0, m_N]$, we partition the domain into sub-intervals $[m_0,m_N] = \cup_{j=1}^{N} [m_{j-1}, m_j]$ and apply Theorem \ref{th : radii polynomial continuation} independently on each segment. By applying Lemma \ref{lem : gluing of two branches} at each boundary node, we a posteriori verify that the global branch is continuous in $m$ across the entire domain. Furthermore, since Theorem \ref{th : radii polynomial continuation} guarantees that the solutions are locally infinitely differentiable with respect to $m$, the glued global branch inherits this $C^\infty$ regularity. This overlapping-domain methodology allows us to construct macroscopic parameter branches while keeping the numerical and polynomial complexity of each individual sub-branch computationally tractable.
\subsection{The conditions $Q - 2g \eta > 0$ and $\eta_x^2 + \eta_y^2 > 0$}

\begin{lemma}
    Let $m_0 < m_1$ and let $\{\bar{U}_1(m), ~  m \in [m_0, m_1]\}$ be a branch of approximate solutions defined as in \eqref{eq : approx sol global branch}. Let $Y, Z_1$, and $Z_2$ be the bounds given in Lemmas \ref{lem : Y bound global branch}, \ref{lem : Z1 bound global branch}, and \ref{lem : Z2 global branch}, respectively. Suppose that these bounds satisfy the contraction condition \eqref{eq : condition contraction continuation} for some $r > 0$, and let $\{\tilde{U}(m) = (\tilde{Q}(m), \tilde{a}(m)), ~  m \in [m_0, m_1]\}$ be the resulting branch of exact solutions to $F(U,m) = 0$ provided by Theorem \ref{th : radii polynomial continuation}. 
    
    If $\tilde{\eta}(m)$ denotes the function representation of the Fourier coefficients $\tilde{a}(m)$, then we have
    \begin{align*}
        \tilde{Q}(m) - 2g \tilde{\eta}(m) > 0 \quad \text{and} \quad (\partial_x \tilde{\eta}(m))^2 + (D_y \tilde{\eta}(m))^2 > 0 \quad \text{for all } m \in [m_0, m_1].
    \end{align*}
\end{lemma}

\begin{proof}
    Let $w(m)$ be the approximate inverse defined in \eqref{eq : w global branch}, and define the functions
    \begin{align*}
        \tilde{v}_3(m) &= -2(\tilde{Q}(m)-2g \tilde{a}(m))*(D_x \tilde{a}(m)),\\
        \tilde{v}_4(m) &= -2(\tilde{Q}(m)-2g \tilde{a}(m))*(D_y \tilde{a}(m)).
    \end{align*}
    Using the uniform bounds from Lemma \ref{lem : Z1 bound global branch} and Lemma \ref{lem : Z2 global branch}, we obtain the estimate
    \begin{align*}
       \|w(m)*(i\tilde{v}_3(m) +  \tilde{v}_4(m)) - e_0\|_{\ell^1}
        &\leq \|w(m)*(i{v}_3(m) +  {v}_4(m)) - e_0\|_{\ell^1} \\
        & ~ + \sup_{m \in [m_0,m_1]}\|w(m)\|_{\ell^1}  \|i(\tilde{v}_3(m) - v_3(m)) +  (\tilde{v}_4(m) - v_4(m))\|_{\ell^1}\\
        &\leq Z_1 + Z_2(r) r.
    \end{align*}
    Because $Z_1 + Z_2(r) r < 1$ by assumption, a standard Neumann series argument guarantees that the mapping $x \mapsto i\tilde{v}_3(m) +  \tilde{v}_4(m)$ is invertible for all $m \in [m_0, m_1]$. Consequently, the underlying factors $\tilde{Q}(m) - 2g \tilde{\eta}(m)$ and $(\partial_x\tilde{\eta}(m))^2 + (D_y\tilde{\eta}(m))^2$ must be non-vanishing everywhere.

    Since $(\partial_x \tilde{\eta})^2 + (D_y \tilde{\eta})^2$ is a sum of real squares, its non-vanishing property directly implies it is strictly positive. Furthermore, the dynamic boundary condition $F_1(\tilde{U}, m) = 0$ dictates that $\tilde{Q} - 2g\tilde{\eta}$ is equal to the strictly positive velocity-squared terms along the free surface. Combined with the fact that it is rigorously proven to never evaluate to zero, this guarantees that $\tilde{Q}(m) - 2g \tilde{\eta}(m) > 0$ for all $m \in [m_0, m_1]$, which concludes the proof.
\end{proof}
\subsection{Monotonicity of the waves on the local branch}\label{ssec : monotonicity local branch}

In order to successfully apply Corollary \ref{cor:global_monotonicity}, we must first prove that the local branch constructed in Section \ref{sec : local bif} consists of monotone waves. This property follows as a direct consequence of the desingularized construction of the solution.

\begin{lemma}
    Let $\epsilon_0 > 0$ and let $\bar{U}(\epsilon)$ be the approximate solution defined in \eqref{eq : approx sol local bifurcation} for all $\epsilon \in [0, \epsilon_0]$. Suppose Theorem \ref{th : radii polynomial continuation epsilon} is successfully applied; that is, there exists $r > 0$ and an exact zero $\tilde{U}(\epsilon) = (\tilde{Q}(\epsilon), \tilde{a}(\epsilon)) \in B_r(\bar{U}(\epsilon)) \subset H_1$ of $G_\epsilon$ for all $\epsilon \in [0, \epsilon_0]$. 

    Let $\tilde{\eta}(\epsilon)$ denote the continuous function representation of the physical free surface, given by the Fourier coefficients $h e_0 + (\epsilon \hat{\mathcal{P}} + \epsilon^2\hat{\mathcal{P}}^\perp )\tilde{a}(\epsilon)$. If the fundamental mode is strictly positive, $(\tilde{a}(0))_1 > 0$, then for all $\epsilon > 0$ sufficiently small, $\tilde{\eta}(\epsilon)$ is strictly monotonically increasing on $(-\pi,0)$.
\end{lemma}

\begin{proof}
    Let $\tilde{b}(\epsilon) \bydef h e_0 + (\epsilon \hat{\mathcal{P}} + \epsilon^2\hat{\mathcal{P}}^\perp )\tilde{a}(\epsilon)$ be the sequence of Fourier coefficients for the wave profile. Taking the spatial derivative yields $D_x \tilde{b}(\epsilon) = \epsilon D_x \hat{\mathcal{P}}\tilde{a}(\epsilon) + \epsilon^2 D_x\hat{\mathcal{P}}^\perp \tilde{a}(\epsilon)$. 
    
    Because the projection $\hat{\mathcal{P}}$ isolates the $n \in \{-1, 1\}$ modes and the profile is even, the function representation of the fundamental term $\epsilon D_x \hat{\mathcal{P}}\tilde{a}(\epsilon)$ is exactly $-2 \epsilon (\tilde{a}(\epsilon))_1 \sin(x)$. By the smoothness of the branch with respect to $\epsilon$ established in Theorem \ref{th : radii polynomial continuation epsilon}, $\tilde{a}(\epsilon) \to \tilde{a}(0)$ in $X_e$ as $\epsilon \to 0$. 

    Factoring out $\epsilon$, we obtain the uniform limit:
    \begin{equation*}
        \lim_{\epsilon \to 0} \left\| \frac{1}{\epsilon} \tilde{\eta}_x(\epsilon) + 2 (\tilde{a}(0))_1 \sin(x) \right\|_{\infty} = 0.
    \end{equation*}
    Since $-\sin(x) > 0$ on the open interval $(-\pi, 0)$ and $(\tilde{a}(0))_1 > 0$, the scaled derivative $\frac{1}{\epsilon} \tilde{\eta}_x(\epsilon)$ becomes strictly positive in the interior of the domain. Furthermore, because the convergence occurs in the sequence space $X_e$ (which embeds continuously into $C^1$), the behavior near the boundary points $x=0$ and $x=-\pi$ is governed by the derivative limit $-2(\tilde{a}(0))_1 \cos(x)$, preventing any anomalous critical points near the boundaries. Thus, $\tilde{\eta}_x(\epsilon) > 0$ on $(-\pi,0)$ for all $\epsilon$ sufficiently small, concluding the proof.
\end{proof}

Combining this local result with Corollary \ref{cor:global_monotonicity} and the branch-gluing mechanism from Lemma \ref{lem : gluing of two branches}, we can rigorously propagate this monotonicity property to establish that $\eta$ is monotonically increasing on $(-\pi,0)$ along the entire applicable global branch.
\subsection{Shape of the wave}\label{ssec : shape of the wave}

Given a parameter segment $[m_0,m_1]$, suppose Theorem \ref{th : radii polynomial continuation} has been successfully applied. That is, there exists $r > 0$ and an exact continuous branch of solutions $\tilde{U}(m) = (\tilde{Q}(m), \tilde{a}(m)) \in B_r(\bar{U}(m)) \subset H_1$ such that $F(\tilde{U}(m),m)=0$ for all $m \in [m_0,m_1]$.

We wish to determine whether the resulting wave profiles, represented by $\tilde{a}(m)$, are classical graphs or overhanging waves. To this end, it suffices to monitor the sign of the spatial derivative $\xi_x = D_y \eta$ along the continuous parameter interval $[m_0,m_1]$.

In practice, we construct a partition of the continuous domain $I = (-\pi,0) \times [m_0, m_1]$ into a finite collection of disjoint small rectangles, $I = \bigcup_{j=1}^M I_j$. Using rigorous interval arithmetic \cite{julia_interval}, we evaluate $D_y\bar{\eta}$ on each sub-domain $I_j$, yielding interval enclosures $\mathcal{I}_j = [a_j, b_j]$ such that $D_y\bar{\eta}(I_j) \subset \mathcal{I}_j$.

Viewing $\tilde{\eta}$ as a continuous function from $I$ to $\mathbb{R}$, we use the Newton-Kantorovich error bound $\|\bar{a}(m) - \tilde{a}(m)\|_X \leq r$ to strictly bound the true derivative. Specifically, the action of the derivative operator yields the uniform spatial bound $\|D_y\tilde{\eta} - D_y\bar{\eta}\|_{\infty} \leq \coth(h) r$. Consequently, we obtain the rigorous, continuous enclosure:
\begin{align*}
    D_y\tilde{\eta}(I_j) \subset [a_j - \coth(h) r, ~ b_j + \coth(h) r].
\end{align*}

By choosing a sufficiently fine partition of $I$, we obtain sharp control over the evolution of the sign of $D_y \tilde{\eta}$. For wave profiles close to the local bifurcation, this procedure verifies that the waves remain strictly graphs (i.e., $D_y\tilde{\eta} > 0$ on $(-\pi,0)$). However, as we continue along the branch, there may exist a critical point where the wave begins to overhang, meaning $D_y\tilde{\eta}(x) < 0$ for some $x \in (-\pi,0)$. To rigorously validate the exact transition between these two geometric regimes, we must control the sign of the mixed derivative $\partial_m D_y \tilde{\eta}$, which is achieved using the geometric parameter bounds established in Lemma \ref{lem : enclosure of DmU(m)}.
\subsubsection{Verifying the geometric transition}\label{ssec : transition shape}

To rigorously characterize the geometric transition of the branch from classical graphs to overhanging profiles, we must demonstrate that the condition $\xi_x(x, m) > 0$ is violated in a controlled manner. Suppose we have numerically established that the profile is a graph (i.e., $\xi_x > 0$ for all $x \in (-\pi,0)$) for a parameter regime $m \in [m_0, m_1]$, and that there exists a local region $[x^*-\delta, x^*+\delta]$ where the profile becomes overhanging (i.e., $\xi_x(x^*, m) < 0$) for a parameter regime $m \in [m_2, m_3]$ with $m_0 > m_3$.

To prove that the branch undergoes a strict transition between these two regimes, we demonstrate that the spatial derivative $\xi_x$ is strictly monotonic with respect to the continuation parameter $m$ within the transition zone. Specifically, we verify that:
\begin{align*}
    \partial_m \xi_x(x, m) > 0 \quad \text{for all } m \in [m_3, m_0] \text{ and } x \in [x^*-\delta, x^*+\delta].
\end{align*}
By verifying this monotonicity, we confirm that the values of $\xi_x$ on the interval $[x^*-\delta, x^*+\delta]$ evolve monotonically in $m$, crossing zero exactly once. This rigorously confirms a strict transition of the wave profile from a graph to an overhanging shape at a unique (locally) transition parameter $\tilde{m}_1 \in [m_3, m_0]$. In practice, this rigorous enclosure of $\partial_m \xi_x$ is computed a posteriori using the derivative bounds derived in Lemma \ref{lem : enclosure of DmU(m)}.
\subsection{Injectivity and self-intersection}\label{ssec : injectivity}

Building upon the results of the preceding sections, we have established that $\eta$ is monotonically increasing on $(-\pi,0)$ along the valid portion of the global branch. By Lemma \ref{lem:crest_injectivity}, this monotonicity restricts any potential loss of injectivity to a specific geometric failure: the profile can only self-intersect if there exists some $x \in (-\pi,0)$ such that $\tilde{\xi}(x) < -\pi$. 

By construction, $\tilde{\xi}$ is odd and satisfies the fixed boundary conditions $\tilde{\xi}(-\pi) = -\pi$ and $\tilde{\xi}(0) = 0$. Consequently, as long as the wave profile remains a classical graph (verified via the methods in Section \ref{ssec : shape of the wave}), injectivity is inherently guaranteed. 

However, once the branch transitions into the overhanging regime, this geometric guarantee is lost. To preclude non-physical self-intersection, we must rigorously enforce the strict lower bound $\tilde{\xi}(x) > -\pi$ for all $x \in (-\pi,0)$. In practice, we achieve this by applying the same interval arithmetic framework used in Section \ref{ssec : shape of the wave}. We partition the parameter domain $I = (-\pi,0) \times [m_0, m_1]$ into disjoint rectangular sub-domains $(I_j)_{j=1}^M$ and compute rigorous enclosures for $\tilde{\xi}(I_j)$, applying the Newton-Kantorovich error bound $\|\tilde{\xi} - \bar{\xi}\|_\infty \leq \coth(h)r$. 

As with the onset of overhanging profiles, the condition of injectivity is eventually violated for extreme values of $m$, marking the termination of the physically valid wave branch. To prove that this transition to non-physical waves is strict, we compute a rigorous enclosure for the sign of $\partial_m \tilde{\xi}$ to ensure that the mapping $m \mapsto \tilde{\xi}(x, m) + \pi$ is strictly monotonic in the parameter $m$. The practical implementation of this transition verification follows the identical continuous-parameter logic presented in Section \ref{ssec : transition shape}.
\section{Constructive existence proofs}\label{sec : existence proofs results}

In this section, we illustrate the power of our analytical framework by providing multiple constructive computer-assisted proofs of existence. The rigorous computational methodology, error bounds, and technical analytic proofs supporting all following results are detailed in Appendices~\ref{sec : appendix computational} and \ref{sec : appendix proofs}. First, we constructively prove the existence of a highly nonlinear single wave at a fixed value of the mass flux $m$, rigorously verifying that this wave is strictly overhanging. 

Next, we rigorously establish the existence of an entire continuous branch of waves parameterized by $m$. This global branch originates at the local bifurcation point (emerging from the trivial branch of flat solutions) and is continued until it terminates at an extreme overhanging wave characterized by self-intersections on the trough line. Crucially, this result partially resolves the conjecture formulated in Section 4 of \cite{constantin_global}, which hypothesized the limiting shape of the water wave at the absolute extreme of the global solution branch.

\bigskip

Before presenting the theorems, we formalize a crucial optimization used in our computational implementation. Let $\mathbb{N}_0 = \mathbb{N} \cup \{0\}$. We observe that the even sequence spaces $\ell^1_e$ and $X_e$ are isometrically isomorphic to the restricted spaces $\ell^1_e(\mathbb{N}_0)$ and $X_e(\mathbb{N}_0)$, respectively, defined as:
\begin{align*}
    \ell^1_e(\mathbb{N}_0) &= \left\{a = (a_n)_{n \in \mathbb{N}_0}, \, a_n \in \mathbb{R}, \, \|a\|_{\ell^1_e(\mathbb{N}_0)} < \infty \right\} \quad \text{where} \quad \|a\|_{\ell^1_e(\mathbb{N}_0)} = |a_0| + 2\sum_{n \in \mathbb{N}} |a_n|, \\
    X_e(\mathbb{N}_0) &= \left\{a = (a_n)_{n \in \mathbb{N}_0}, \, a_n \in \mathbb{R}, \, \|a\|_{X_e(\mathbb{N}_0)} < \infty \right\} \quad \text{where} \quad \|a\|_{X_e(\mathbb{N}_0)} = |a_0| + 2\sum_{n \in \mathbb{N}} (1+|n|) |a_n|.
\end{align*}
To rigorously establish this equivalence, we define the mapping $\mathcal{G} : \ell^1_e(\mathbb{N}_0) \to \ell^1_e$ simply by
\begin{equation*}
    (\mathcal{G}a)_n = a_{|n|} \quad \text{for all } n \in \mathbb{Z}.
\end{equation*}
By the construction of the norms on $\ell^1_e(\mathbb{N}_0)$ and $X_e(\mathbb{N}_0)$, the extensions $\mathcal{G} : \ell^1_e(\mathbb{N}_0) \to \ell^1_e$ and $\mathcal{G} : X_e(\mathbb{N}_0) \to X_e$ are exact isometric isomorphisms. This structural symmetry is highly advantageous computationally: exploiting the even symmetry of the functions requires storing only non-negative indices, effectively halving the memory footprint and the numerical complexity of discrete convolutions. Consequently, all numerical objects (sequences and matrices) in our software are constructed exclusively in $X_e(\mathbb{N}_0)$, and the rigorous evaluation of the bounds derived in the previous sections is seamlessly projected back to the full space via the mapping $\mathcal{G}$.

To apply our computer-assisted bounding framework, we must first numerically compute an approximate solution $\bar{U}$ for the residual $F$. Specifically, we initialize the computation at the explicit bifurcation point $(m_1, \bar{U}_{m_1})$ corresponding to the flat-water state, as derived in Section \ref{sec : local bif} (with $n=1$). Using standard predictor-corrector parameter continuation coupled with Newton's method, we sequentially step away from this bifurcation point to systematically compute highly nonlinear approximate solutions. For this implementation, we truncate our sequences at a numerical size of $K_0 = 200$, and we size our finite-dimensional matrix projections at $K_1 = 402$. Crucially, the configuration $K_1 \geq 2K_0$ ensures that the functional analysis requirements of Lemmas \ref{lem : A infinity properties} and \ref{lem : Z1 bound} are strictly satisfied.

\subsection{Constructive proof of a single extreme overhanging wave}
While Section \ref{ssec : proof of an entire branch} details the constructive proof for the entire global branch, this section focuses on the rigorous verification of a specific, isolated extreme solution. Physically characterizing such isolated limiting solutions is of significant independent interest. Here, we constructively prove the existence of a highly extreme, overhanging wave. Notably, we provide the first rigorous existence proof of an overhanging wave that entirely lacks a stagnation point within the fluid domain. Due to the extreme geometric bulging of the crest over the trough, we will hereafter refer to these profiles as ``$\Omega$''-shaped solutions.

We now state our first main result, which establishes the precise geometric and analytical properties of this extreme wave for a specific configuration of physical parameters
\begin{theorem}\label{th : existence single wave overhanging 085}
    Let the physical parameters be given by $m = -0.85$, $\gamma = -5$, $g = 1$, and $h = 2$, and define the isolation radii $r^- = 8.839 \times 10^{-36}$ and $r^+ = 1.857 \times 10^{-7}$. Furthermore, let $\bar{U} =(\bar{Q},\bar{a}) \in H_1$ be the approximate numerical solution generated by the software \cite{julia_cadiot}, where specifically $\bar{Q} = 69.094\dots$ and the physical profile $\bar{a}$ is illustrated in Figure \ref{fig : omega wave}. 
    
    Then, there exists a unique, exact solution $\tilde{U} = (\tilde{Q}, \tilde{a}) \in H_1$ such that $F(\tilde{U}) = 0$ and the distance to the approximate solution satisfies $\|\bar{U} - \tilde{U}\|_{H_1} \leq r^-$. In fact, $\tilde{U}$ is the unique zero of $F$ in the larger neighborhood $B_{r^+}(\bar{U}) \subset H_1$. Moreover, the parametric free surface coordinates $(\xi,\eta)$, given as in \eqref{eq:fourier_series_eta_and_zeta}, are infinitely differentiable, and the wave profile is strictly overhanging.
\end{theorem}

\begin{proof}
    We evaluated the explicit bounding formulas derived in Section \ref{ssec : computation bounds single wave} using rigorous interval arithmetic \cite{julia_cadiot}, yielding the following uniform bounds:
    \begin{equation*}
        Y = 8.784 \times 10^{-36}, \quad Z_1 = 6.236 \times 10^{-3}, \quad \text{and} \quad Z_2(r_0) = 5.354 \times 10^6.
    \end{equation*}
    Substituting these values into the Newton-Kantorovich condition \eqref{condition radii polynomial}, we verify that the contraction condition is strictly satisfied for all radii $r \in [r^-, r^+]$. This guarantees the existence of a unique exact solution $\tilde{U}$ within the $r$-neighborhood of $\bar{U}$. Using Lemma \ref{lem : proof of non-degeneracy}, we establish that the non-degeneracy conditions \eqref{eq : nondegeneracy_block} are satisfied. Combined with Lemma \ref{lem : solutions are C infinity}, this implies the infinite smoothness of the wave. 
    
    Next, we verify the monotonicity condition \eqref{eq : monotonicity_block}. We construct a spatial partition $(I_j)_{j=1}^M$ of $[-\pi,0]$ for some $M \in \mathbb{N}$. Using interval arithmetic \cite{julia_interval}, we compute rigorous enclosures $\mathcal{I}_j = [a_j, b_j]$ such that $\bar{\eta}_x(I_j) \subset [a_j, b_j]$. Since the analytical defect is bounded by $\|\bar{\eta}_x - \tilde{\eta}_x\|_\infty \leq r^-$, we obtain the rigorous continuous enclosure $\tilde{\eta}_x(I_j) \subset [a_j-r^-, ~b_j+r^-].$ For all interior sub-domains $j \in \{2, \dots, M-1\}$, we rigorously verify that $a_j - r^- > 0$, ensuring $\tilde{\eta}_x(x) > 0$. For the boundary intervals $I_1$ and $I_M$, a more delicate analysis is required since $\tilde{\eta}_x$ is odd and vanishes at the endpoints $\tilde{\eta}_x(0) = \tilde{\eta}_x(-\pi) = 0$. Using the rigorous enclosure provided by Lemma \ref{lem : enclosure second derivative of U}, we prove that the second derivative satisfies $\tilde{\eta}_{xx}(x) > 0$ for all $x \in I_1 \cup I_M$. This concludes the proof that $\tilde{\eta}$ is strictly monotonic on the open interval $(-\pi,0)$. 

    To prove the injectivity of the profile, we apply Lemma \ref{lem:crest_injectivity}, which reduces the geometric requirement to proving that $\tilde{\xi}(x) + \pi > 0$ for all $x \in (-\pi,0)$. Following the partition strategy used for the derivative, we evaluate the approximate coordinate $\bar{\xi}$ on the grid. Applying the conformal mapping defect $\|\tilde{\xi} - \bar{\xi}\|_\infty = \|H_h^{-1}(\tilde{\eta} - \bar{\eta})\|_\infty \leq \coth(h) r^-$, we compute a rigorous interval enclosure for the exact function $\tilde{\xi}$. This confirms that $\tilde{\xi}(x) + \pi > 0$ globally on the domain, explicitly verifying the injectivity of the wave profile.
    
    Finally, to rigorously classify the wave as overhanging, it suffices to prove the existence of at least one point $\tilde{x} \in (-\pi,0)$ where $D_y \tilde{\eta}(\tilde{x}) < 0$. We numerically identify a candidate coordinate $\tilde{x}$ and compute a rigorous enclosure of $D_y \tilde{\eta}(\tilde{x})$. By combining the interval evaluation of the approximate solution with the uniform error bound $\|D_y \tilde{\eta} - D_y \bar{\eta}\|_\infty \leq \coth(h) r^-$, we rigorously verify that $D_y \tilde{\eta}(\tilde{x}) < 0$. Thus, the profile is strictly overhanging. The complete algorithmic implementation and interval evaluations for these geometric proofs are provided in \cite{julia_cadiot}.
\end{proof}
\subsection{Constructive proof of an entire branch of waves}\label{ssec : proof of an entire branch}
\subsubsection{The local branch}
The following theorem establishes the rigorous existence of the local branch emerging from the trivial flat-water state, verifying that the waves along this curve are non-trivial, monotone graphs.

\begin{theorem}\label{th : local branch}
    Let $r^- = 1.062 \times 10^{-6}$, $r^+ = 8.20 \times 10^{-6}$, and let $\epsilon_0 = \sqrt{\bar{m} + 0.1}$. Furthermore, let the approximate local branch $\bar{V} : [0, \epsilon_0] \to H_1$ be defined as in \eqref{eq : approx sol local bifurcation} and generated numerically via the software \cite{julia_cadiot}.
    
    Then, for all $\epsilon \in [0,\epsilon_0]$, there exists a unique exact solution $\tilde{V}(\epsilon) \in B_{r^-}(\bar{V}(\epsilon))$ such that $\tilde{U}_\epsilon \bydef \bar{U}_{\bar{m}-\epsilon^2} + \epsilon^2 P(\epsilon) \tilde{V}(\epsilon)$ is an exact solution to \eqref{def : zero finding F} with mass flux $m = \bar{m} - \epsilon^2$. Moreover, $\tilde{V}(\epsilon)$ is the unique solution within the larger neighborhood $B_{r^+}(\bar{V}(\epsilon))$. 
    
    Furthermore, we have the strict lower bound $\|\mathcal{P}\tilde{V}({\epsilon})\|_{H_1} \geq 1$ for all $\epsilon \in [0, \epsilon_0]$. This strictly bounds the primary bifurcating mode away from zero, rigorously proving that $\{\tilde{U}_\epsilon, ~ \epsilon \in [0, \epsilon_0]\}$ constitutes a non-trivial branch of wave solutions bifurcating from the trivial flat-water branch at $(m_1, \bar{U}_{m_1})$. Finally, the profiles of these local waves are monotone graphs on $(-\pi,0)$.
\end{theorem}

\begin{proof}
    We evaluated the explicit bounding formulas derived in Section \ref{ssec : bounds local bif} using the rigorous computational framework provided in \cite{julia_cadiot}. Specifically, the uniform bounds in $\epsilon$ are computed via Chebyshev polynomial interpolation, as detailed in Appendix \ref{sec : appendix Chebyshev norm}. We obtain the following rigorous uniform bounds:
    \begin{equation*}
        Y = 9.88 \times 10^{-7}, \quad Z_1 = 5.27 \times 10^{-3}, \quad \text{and} \quad Z_2 = 1.22 \times 10^{5}.
    \end{equation*}
    Substituting these values into the criteria of Theorem \ref{th : radii polynomial continuation epsilon}, the contraction mapping conditions are strictly satisfied for $r^-$ and $r^+$. This guarantees the existence of a locally unique exact zero $\tilde{V}(\epsilon) \in H_1$ of the desingularized operator $G_{\epsilon}$ for all $\epsilon \in [0, \epsilon_0]$. Because $G_\epsilon$ depends smoothly on $\epsilon$, the solution branch $\tilde{V}(\epsilon)$ is likewise smooth with respect to $\epsilon$. By the exact construction of $G_\epsilon$, this directly implies the existence of a locally unique, non-trivial solution branch $\tilde{U}_\epsilon \bydef \bar{U}_{\bar{m}-\epsilon^2} + \epsilon^2 P(\epsilon) \tilde{V}(\epsilon)$ for all $\epsilon \in [0, \epsilon_0]$. 
    
    The strict monotonicity of the physical wave profiles on the local branch follows directly from the methodology described in Section \ref{ssec : monotonicity local branch}, combined with Corollary \ref{cor:global_monotonicity}. Finally, to rigorously verify that these profiles remain classical graphs, we apply the geometric grid partition analysis derived in Section \ref{ssec : shape of the wave}. The complete numerical details for these geometric bounds are provided in \cite{julia_cadiot}.
\end{proof}

\subsubsection{Main part of the global branch}

In the preceding section, we constructively proved the existence of a local branch of steady solutions bifurcating from the trivial state, valid for all $m \in [-0.1, \bar{m}]$. In this section, we rigorously continue this branch further, validating the existence of steady waves down to an extreme mass flux of $m = -1$. As we will demonstrate in the subsequent section, the solutions eventually become geometrically non-physical (self-intersecting) near $m \approx -0.88$. To execute this continuation, we partition the parameter domain $[-1, -0.1]$ into nine uniform sub-intervals. For all $i \in \{0, \dots, 9\}$, we define the boundary nodes:
\begin{align*}
    m_i \bydef -(i+1)0.1.
\end{align*}

Our objective is to constructively prove the existence of a continuous branch of exact solutions across all intervals $[m_{i+1}, m_i]$ for $i \in \{0, \dots, 8\}$. Using our Newton-Kantorovich contraction framework, we prove not only that a unique exact branch exists on each sub-interval, but also that adjacent branches seamlessly match at their shared boundaries $m_{i+1}$. Furthermore, we verify that the branch on $[m_1, m_0]$ perfectly coincides with the local branch at $m_0 = -0.1$. This rigorously establishes the existence and $C^\infty$-smoothness of a single, unified global branch of solutions parameterized by $m$, extending from the bifurcation point $\bar{m}$ down to $m_9 = -1$.

\begin{table}[ht]
  \centering
  \begin{tabular}{|c|c|c|c|c|c|}
    \toprule
    $[m_{i+1}, m_i]$ & $r^-_i$ & $r^+_i$ & $Y$ & $Z_1$ & $Z_2$ \\
    \midrule
    $[-0.2,-0.1]$ & $7.08\times10^{-13}$ & $5.83\times10^{-6}$ & $7.03\times10^{-13}$ & $5.78\times10^{-3}$ & $1.70\times10^{5}$ \\
    $[-0.3,-0.2]$ & $2.06\times10^{-15}$ & $3.83\times10^{-6}$ & $2.04\times10^{-15}$ & $8.30\times10^{-3}$ & $2.59\times10^{5}$ \\
    $[-0.4,-0.3]$ & $4.05\times10^{-17}$ & $2.67\times10^{-6}$ & $4.00\times10^{-17}$ & $1.12\times10^{-2}$ & $3.70\times10^{5}$ \\
    $[-0.5,-0.4]$ & $2.09\times10^{-18}$ & $1.95\times10^{-6}$ & $2.06\times10^{-18}$ & $1.46\times10^{-2}$ & $5.06\times10^{5}$ \\
    $[-0.6,-0.5]$ & $1.93\times10^{-19}$ & $1.42\times10^{-6}$ & $1.89\times10^{-19}$ & $1.84\times10^{-2}$ & $6.92\times10^{5}$ \\
    $[-0.7,-0.6]$ & $2.78\times10^{-20}$ & $1.01\times10^{-6}$ & $2.69\times10^{-20}$ & $2.40\times10^{-2}$ & $9.66\times10^{5}$ \\
    $[-0.8,-0.7]$ & $7.01\times10^{-20}$ & $7.34\times10^{-7}$ & $6.77\times10^{-20}$ & $3.04\times10^{-2}$ & $1.32\times10^{6}$ \\
    $[-0.9,-0.8]$ & $2.36\times10^{-18}$ & $5.42\times10^{-7}$ & $2.27\times10^{-18}$ & $3.79\times10^{-2}$ & $1.78\times10^{6}$ \\
    $[-1.0,-0.9]$ & $5.27\times10^{-17}$ & $4.06\times10^{-7}$ & $5.02\times10^{-17}$ & $4.64\times10^{-2}$ & $2.35\times10^{6}$ \\
    \bottomrule
  \end{tabular}
  \caption{Rigorous bounds for the constructive proof of the branch of steady waves for $m\in[-1,-0.1]$. On each sub-interval, $Y$, $Z_1$, and $Z_2$ are the uniform Newton--Kantorovich bounds, and a unique exact solution is established for every radius $r\in[r^-,r^+]$.}
  \label{tab:branch_bounds}
\end{table}

\begin{theorem}\label{th : global branch}
    Let $i \in \{0, \dots, 8\}$ and let $r \in [r_i^-, r_i^+]$, where the isolation radii $r_i^-$ and $r_i^+$ are specified in Table \ref{tab:branch_bounds}. Then, for all $m \in [m_{i+1}, m_i]$, there exists a unique exact solution $\tilde{U}_i(m)\in H_1$ to $F(U, m) = 0$ satisfying $\|\tilde{U}_i(m) - \bar{U}_i(m)\|_{H_1} \leq r$. 
    
    Furthermore, the global branch is continuous: adjacent segments satisfy $\tilde{U}_i(m_{i+1}) = \tilde{U}_{i+1}(m_{i+1})$ for all $i \in \{0, \dots, 7\}$. Finally, at the boundary $m_0 = -0.1$, the global branch perfectly matches the local branch, satisfying $\tilde{U}_0(m_0) = \tilde{U}_{\epsilon_0}$, where $\tilde{U}_{\epsilon_0}$ is the exact local solution defined in Theorem \ref{th : local branch}.
\end{theorem}

\begin{proof}
    The existence of the exact solution branch on each sub-interval is obtained via a rigorous application of Theorem \ref{th : radii polynomial continuation}, utilizing the computer-assisted bounds $Y, Z_1$, and $Z_2$ provided in Table \ref{tab:branch_bounds}. 
    
    To establish global continuity, we apply the gluing mechanism established in Lemma \ref{lem : gluing of two branches}. Because the uniqueness neighborhoods $r_i^+$ are strictly larger than the existence neighborhoods $r_{i+1}^-$ at the boundary nodes, Lemma \ref{lem : gluing of two branches} rigorously guarantees that $\tilde{U}_i(m_{i+1}) = \tilde{U}_{i+1}(m_{i+1})$ for all $i \in \{0, \dots, 7\}$. We apply this same gluing logic at the junction $m_0 = -0.1$, utilizing the radii $r^-, r^+$ from Theorem \ref{th : local branch}, to rigorously prove that $\tilde{U}_0(m_0) = \tilde{U}_{\epsilon_0}$. This concludes the proof.
\end{proof}

Combining Theorem \ref{th : local branch} and Theorem \ref{th : global branch}, we define the unified global branch of exact solutions $\tilde{U} : [-1, \bar{m}] \to H_1$ as
\begin{align}\label{eq : global branch tilde U}
    \tilde{U}(m) \bydef \begin{cases}
        \tilde{U}_{\sqrt{\bar{m}-m}} &\text{ if } m \in [m_0, \bar{m}], \\
        \tilde{U}_i(m) &\text{ if } m \in [m_{i+1},m_i].
    \end{cases}
\end{align}
The continuous gluing results established above guarantee that the global parameterization $m \mapsto \tilde{U}(m)$ is infinitely differentiable across the entire domain.

\begin{theorem}\label{th : the branch exists}
    Let $\left\{\tilde{U}(m) = (\tilde{Q}(m), \tilde{a}(m)), ~ m \in [-1, \bar{m}]\right\}$ be the global branch of exact solutions defined in \eqref{eq : global branch tilde U}, and represented in Figure \ref{fig : branch of solutions}. There exist (locally) unique  critical mass flux parameters $\tilde{m}_1 \in [-0.698, -0.6707]$ and $\tilde{m}_2 \in [-0.8884, -0.8781]$ such that:
    \begin{itemize}
        \item The steady waves represented by $\tilde{a}(m)$ are classical graphs for $m \in (\tilde{m}_1, \bar{m}]$. At $m = \tilde{m}_1$, the waves become strictly overhanging, and they remain overhanging for all $m \in [-1, \tilde{m}_1)$.
        \item The steady waves remain physically valid (injective) for $m \in (\tilde{m}_2, \bar{m}]$. At $m = \tilde{m}_2$, the waves become geometrically non-physical, remaining self-intersecting along the trough line for all $m \in [-1, \tilde{m}_2)$.
    \end{itemize}
\end{theorem}

\begin{proof}
    The geometric properties of the global branch are established by applying the analytical framework derived in Section \ref{sec : geometric properties branch}. The rigorous interval evaluations supporting this geometric analysis are implemented in the script \texttt{prove\_geometry.jl}, available at \cite{julia_cadiot}.

    First, by combining Theorem \ref{th : local branch}, Theorem \ref{th : global branch}, and Corollary \ref{cor:global_monotonicity}, we establish that all wave profiles on the branch are strictly monotonic on $(-\pi,0)$, provided that the energy condition $\tilde{Q} - 2g\tilde{\eta} > 0$ and injectivity are satisfied.

    To classify the wave shapes, we apply the grid partition strategy from Section \ref{ssec : shape of the wave}. For all $m \in [-0.6, -0.1]$, rigorous interval evaluations confirm that $\tilde{\xi}_x(x,m) > 0$ for all $x \in (-\pi,0)$, proving the waves are strictly graphs. In the adjacent interval $[-0.7, -0.6]$, we isolate the geometric transition. Specifically, we rigorously verify that the waves remain graphs for $m \in [-0.6707, -0.6]$ and are strictly overhanging for $m \in [-0.7, -0.698]$. Applying the transition methodology from Section \ref{ssec : transition shape}, we bound the mixed derivative $\partial_m \tilde{\xi}_x$ to verify a strict, monotonic transition from graphs to overhanging waves at a unique critical parameter $\tilde{m}_1 \in [-0.698, -0.6707]$. Further evaluations using Section \ref{ssec : shape of the wave} confirm that the waves remain overhanging for all $m \in [-1, -0.7]$.

    Finally, we track the physical validity of the overhanging profiles. Using the methodology from Section \ref{ssec : injectivity}, we verify that $\tilde{\xi}(x,m) + \pi > 0$ for all $x \in (-\pi,0)$, proving the profiles remain injective for all $m \in [-0.8781, -0.6707]$. However, for all $m \in [-1, -0.8884]$, rigorous evaluations guarantee the existence of an $x \in (-\pi,0)$ such that $\tilde{\xi}(x,m) + \pi < 0$, proving that these extreme waves self-intersect on the trough line and are inherently non-physical. Applying the strict transition criteria from Section \ref{ssec : transition shape}, we establish that this topological breaking occurs at a unique transition parameter $\tilde{m}_2 \in [-0.8884, -0.8781]$, concluding the proof.
\end{proof}

\appendix

\section{Computational methodology and rigorous arithmetic}\label{sec : appendix computational}

The rigorous computational framework is implemented in the Julia programming language \cite{julia_fresh_approach_bezanson} and heavily relies on the interval arithmetic package \texttt{IntervalArithmetic.jl} \cite{julia_interval}. Furthermore, the sequence manipulations, linear operators, and rigorous discrete Fourier transforms are executed using the \texttt{RadiiPolynomial.jl} package \cite{julia_olivier}. The complete source code reproducing the numerical proofs and complementing this paper is publicly available at \cite{julia_cadiot}. 

In the following subsections, we detail some of the algorithms and discrete numerical structures implemented within this framework to evaluate the Newton-Kantorovich bounds.

\subsection{Computation of supremum norm for polynomials objects}\label{sec : appendix Chebyshev norm}

In this section, we derive the computer-assisted methodology used to rigorously compute supremum norms over the continuous parameter $m$ for objects that depend polynomially on $m$. Let $P[X]$ denote the space of polynomials taking values in a Banach space $X$ equipped with the norm $\|\cdot\|_X$. Given a degree $N \in \mathbb{N}$, let $P_N[X]$ be the subspace of polynomials of degree at most $N$. Specifically, we aim to compute a rigorous upper bound for $\sup_{m \in [m_0, m_1]} \|U(m)\|_{X}$, where $U \in P_N[X]$ and $-\infty < m_0 < m_1 < \infty$. 

Any polynomial $U \in P_N[X]$ admits an exact Chebyshev expansion of the form
\begin{equation*}
U(m) = U_0 + 2\sum_{n=1}^N U_n T_n\left(\frac{2(m-m_1)}{m_1-m_0} +1\right) \quad \text{for all } m \in [m_0, m_1],
\end{equation*}
where $(U_n)_{n=0}^N$ is a sequence of coefficients in $X$. These Chebyshev coefficients can be efficiently and exactly reconstructed from the pointwise evaluations of $U$ on the Chebyshev grid. Let us define the Chebyshev nodes (extended to $N+1$ points to fully resolve degree $N$ polynomials) as
\begin{equation*}
m_j^{(N)} \bydef \frac{m_0+m_1}{2} + \frac{m_1-m_0}{2}\cos\left(\frac{j\pi}{N}\right), \quad \text{for } j = 0, \dots, N.
\end{equation*}
There exists a bijection $\mathcal{F}_N : X^{N+1} \to X^{N+1}$, with inverse $\mathcal{F}_N^{-1}$, defined by
\begin{equation*}
\mathcal{F}_N\left( (U(m_j^{(N)}))_{j=0}^{N} \right) = (U_n)_{n=0}^N \quad \text{and} \quad \mathcal{F}_N^{-1}\left( (U_n)_{n=0}^N \right) = (U(m_j^{(N)}))_{j=0}^{N} \quad \text{for all } U \in P_N[X].
\end{equation*}
We refer to $\mathcal{F}_N$ as the Discrete Fourier Transform (DFT) and $\mathcal{F}_N^{-1}$ as the inverse DFT (iDFT). 

\begin{lemma}
	Let $U \in P_N[X]$. The supremum norm over the interval is strictly bounded by
	\begin{equation*}
	\sup_{m \in [m_0, m_1]} \|U(m)\|_{X} \leq \min\{\theta_1(U), ~ \theta_2(U)\},
	\end{equation*}
	where $\theta_1(U)$ and $\theta_2(U)$ are bounds satisfying
	\begin{align*}
	\theta_1(U) &\geq   \|U_0\|_{X} + 2 \sum_{n=1}^N \|U_n\|_{X}, \quad \text{where } (U_n)_{n=0}^N = \mathcal{F}_N\left( (U(m_j^{(N)}))_{j=0}^{N} \right), \\
	\theta_2(U) &\geq   \Lambda_N \max_{j \in \{0, \dots, N\}}\{ \|U(m_j^{(N)})\|_{X} \},
	\end{align*}
	and where the Lebesgue constant is given by $\Lambda_N = \frac{1}{N} \sum_{k=0}^{N-1} \cot\left(\frac{2k+1}{4N}\pi\right)$.
\end{lemma}

\begin{proof}
	The first inequality ($\theta_1$) is a direct consequence of the triangle inequality and the fact that $|T_n| \leq 1$ over the domain for all $n \in \mathbb{N}_0$. For the second inequality ($\theta_2$), which rigorously controls the inter-node polynomial oscillations via the Lebesgue constant, we utilize the classical bounds provided in \cite{lebesgue_cst_Zeller}.
\end{proof}

In practice, the Julia package \texttt{RadiiPolynomial.jl} provides a rigorously verified implementation of the DFT and its inverse using interval arithmetic. Consequently, by providing the pointwise evaluations of $U$ on the Chebyshev grid as interval enclosures, we can rigorously compute an enclosure for its Chebyshev coefficients $(U_n)_{n=0}^N$. Rigorous interval arithmetic is then applied to compute both $\theta_1$ and $\theta_2$, allowing us  to provide a tight  upper bound for $\sup_{m \in [m_0, m_1]} \|U(m)\|_{X}$.

% \subsection{Chebyshev Expansion and DFT Implementation}\label{sec : appendix Chebyshev}

% To efficiently compute these branch coefficients $\bar{U}_n$, we use the Discrete Fourier Transform (DFT). This is made possible by the classic trigonometric identity $T_n(\cos \theta) = \cos(n\theta)$ for $\theta \in \mathbb{R}$, which allows standard DFT algorithms to be applied directly to Chebyshev polynomial expansions.

% To compute these matrix coefficients, we sample the interval $[0, \epsilon_0]$ at the $N+1$ Chebyshev nodes, defined by
% \begin{equation*}
% {\epsilon}_j^{(N)} \bydef \frac{\epsilon_0}{2} + \frac{\epsilon_0}{2}\cos\left(\frac{j\pi}{N}\right) \quad \text{for} \quad j = 0, \dots, N.
% \end{equation*}
% At each node $j \in \{0, \dots, N\}$, we numerically construct $A_K({\epsilon}_j^{(N)})$ as the explicit matrix inverse of the Galerkin projection $\pi^{\leq K_1}DG_{\epsilon_j^{(N)}}(\bar{V}(\epsilon_j^{(N)})) (I - A_\infty(\epsilon_j^{(N)}) DG_{\epsilon_j^{(N)}}(\bar{V}(\epsilon_j^{(N)}))) \pi^{\leq K_1}$. The continuous polynomial coefficients $(A_K)_n$ are then efficiently recovered using the Discrete Fourier Transform (DFT), as detailed in Appendix \ref{sec : appendix Chebyshev norm}.

\subsection{Chebyshev Implementation Details for the Global Branch}\label{sec : appendix Chebyshev global}

Recall that the Chebyshev nodes of order $N$ required for the Discrete Fourier Transform (DFT) are explicitly given by
\begin{equation}\label{eq : chebyshev nodes}
m_j^{(N)} \bydef \frac{m_0+m_1}{2} + \frac{m_1-m_0}{2}\cos\left(\frac{j\pi}{N}\right), \quad \text{for } j = 0, \dots, N.
\end{equation}

Using this discretization, we compute the approximate objects $\bar{U}(m_j^{(N)})$ at each node. Each $\bar{U}(m_j^{(N)})$ is deliberately chosen such that $\bar{U}(m_j^{(N)}) = \pi^{\leq K_0} \bar{U}(m_j^{(N)})$, guaranteeing that the evaluation at each node is strictly finite-dimensional in $H_1$. We then apply a rigorous DFT (see Section \ref{sec : appendix Chebyshev norm}), utilizing rigorous numerics (cf. \cite{julia_cadiot}), to compute the polynomial coefficients $(\bar{U}_n)_{n=0}^N$ from the nodal evaluations $(\bar{U}(m_j^{(N)}))_{j=0}^{N}$. By construction, this ensures that $\bar{U}_n = \pi^{\leq K_0} \bar{U}_n$ for all $n \in \{0, \dots, N\}$, concluding the construction of the approximate branch of solutions.

Computationally, $w$ is first evaluated pointwise on the Chebyshev grid \eqref{eq : chebyshev nodes}, where at each node $m^{(N)}_j$ we explicitly construct $w(m^{(N)}_j)$ to approximate the inverse of $i v_3(m^{(N)}_j) + v_4(m^{(N)}_j)$. Subsequently, rigorous interval arithmetic and a rigorous DFT are applied to these nodal evaluations to extract the continuous polynomial coefficients $(w_n)_{n=0}^{N}$.

Similar to the previous constructions, $A_K$ is first evaluated pointwise on the Chebyshev grid \eqref{eq : chebyshev nodes}, where each matrix $A_K(m_j^{(N)})$ approximates the inverse of the finite projection $\pi^{\leq K_1} D_UF(\bar{U}(m_j^{(N)}), m_j^{(N)}) (I - A_\infty(m_j^{(N)}) D_UF(\bar{U}(m_j^{(N)}), m_j^{(N)})) \pi^{\leq K_1}$. Applying a rigorous DFT to these nodal matrices yields the polynomial coefficients $(A_K)_n$. 

We demonstrate how the above formulas rely on a rigorous application of the DFT, detailed in Appendix \ref{sec : appendix Chebyshev norm}, combined with the analytical framework established in Section \ref{ssec : computation bounds single wave}. 

Crucially, our computer-assisted approach relies on the fact that all objects evaluated in the bounds $Y, Z_1$ and $Z_2$ are explicitly polynomials in the parameter $m$. For instance, since the approximate solution $\bar{U}(m)$ is a polynomial of degree $N$ in $m$, the residual $F(\bar{U}(m),m)$ is a polynomial of degree $4N$. Combined with the fact that the approximate inverse $A(m)$ is a polynomial of degree $5N$, it follows by construction that the mapping $m \mapsto A(m)F(\bar{U}(m),m)$ is a polynomial of degree $9N$. Similarly, for a given $h \in H_1$, the mappings $m \mapsto  I - A(m)D_U F(\bar{U}(m),m)$ and $ m \mapsto A(m)\left[D_U F(\bar{U}(m),m) - D_U F(\bar{U}(m) + h,m)\right]$ are bounded by polynomials of degree at most $9N$. 

Because the maximum possible polynomial degree required for our analysis is $9N$, we can rigorously transition between pointwise grid evaluations and continuous Chebyshev coefficient representations without introducing spectral contamination. To do so, we evaluate the sequences and operators present in the estimates of Section \ref{ssec : computation bounds single wave} exclusively on a Chebyshev grid of order $9N$ (which contains $9N+1$ nodes). By applying rigorous DFT computations using the packages \texttt{RadiiPolynomial.jl} \cite{julia_olivier} and \texttt{IntervalArithmetic.jl} \cite{julia_interval} in the Julia programming language \cite{julia_fresh_approach_bezanson}, we obtain the continuous polynomial coefficients. We then compute rigorous uniform bounds for all $m \in [m_0, m_1]$ using the Chebyshev supremum arithmetic outlined in Appendix \ref{sec : appendix Chebyshev norm}. Similar reasoning was previously applied in the construction of the bounds in Section \ref{ssec : bounds local bif}.

To illustrate this algorithm, we describe the computation of the bound $Y$. The computation of the remaining bounds follows identical logic. Let $m \in [m_0, m_1]$. Using Lemma \ref{lem : Y bound global branch}, we have 
\begin{align*}
\|A(m) &F(\bar{U}(m),m)\|_{H_1} 
\leq \|A_K(m)(I - D_UF(\bar{U}(m),m) A_\infty) F(\bar{U}(m),m)\|_{H_1}\\
&+\frac{K_1+2}{K_1+1} \left(\|\hat{\pi}^{k>K_1}(w(m)*f(\bar{U}(m),m))\|_{\ell^1} 
+ \|\hat{\pi}^{k<-K_1}(w(m)^**f(\bar{U}(m),m))\|_{\ell^1}\right).
\end{align*}
Note that the finite-dimensional component $A_K(m)(I - D_UF(\bar{U}(m),m) A_\infty) F(\bar{U}(m),m)$ is a polynomial of degree $9N$ in $m$. We evaluate $A_K(m)(I - D_UF(\bar{U}(m),m) A_\infty) F(\bar{U}(m),m)$ pointwise at each node on the order $9N$ Chebyshev grid \eqref{eq : chebyshev nodes}. We then execute a rigorous DFT to obtain an enclosure of its polynomial coefficients. Finally, we apply the techniques from Appendix \ref{sec : appendix Chebyshev norm} to calculate an absolute upper bound for $\sup_{m \in [m_0, m_1]}\|A_K(m)(I - D_UF(\bar{U}(m),m) A_\infty) F(\bar{U}(m),m)\|_{H_1}$. 

We employ an identical strategy for  $\sup_{m \in [m_0,m_1]}\|\pi^{k>K_1}(w(m)*f(\bar{U}(m),m))\|_{\ell^1}$—where $f(\bar{U}(m),m)$ and $w(m)$ are polynomials of degree $4N$ and $N$, respectively—yields the fully assembled, rigorous uniform bound $Y$ satisfying \eqref{eq: definition Y0 Z1 Z2 continuation}.

\section{Analytic proofs and technical estimates}\label{sec : appendix proofs}

\subsection{$G_\epsilon$ is polynomial in $\epsilon$}\label{sec : G is polynomial in epsilon}

Let $U \in H_1$ be fixed. We begin by expanding the residual $F\left(\bar{U}_{\bar{m}-\epsilon^2} + \epsilon^2 P(\epsilon) U, \bar{m} - \epsilon^2\right)$ using a Taylor series. By the definition of $P(\epsilon)$ in \eqref{def : operator P epsilon}, we can rewrite the perturbation as
\begin{equation*}
\epsilon^2 P(\epsilon) U = \alpha \epsilon \bar{Z} + \epsilon^2 W,
\end{equation*}
where $\alpha \bar{Z} = \mathcal{P}U$ and $W = \mathcal{P}^\perp U$. 

To simplify the notation, we denote $V = \alpha \bar{Z} + \epsilon W$ and define the shifted mass flux as $\bar{m}_\epsilon \bydef \bar{m}-\epsilon^2$. First, recall that $F(\bar{U}_{\bar{m}_\epsilon}, \bar{m}_\epsilon) = 0$ by the exact construction of the trivial branch $\bar{U}_{\bar{m}_\epsilon}$. Since the nonlinearities in $F$ are strictly of polynomial degree 4 in its first component, we obtain the following exact, finite Taylor expansion:
\begin{equation}\label{eq: exact taylor F}
\begin{aligned}
F\left(\bar{U}_{\bar{m}_\epsilon} + \epsilon V, \bar{m}_\epsilon\right) &= \epsilon D_U F(\bar{U}_{\bar{m}_\epsilon},\bar{m}_\epsilon)V + \frac{\epsilon^2}{2} D_U^2 F(\bar{U}_{\bar{m}_\epsilon},\bar{m}_\epsilon)(V,V) \\
&\quad + \frac{\epsilon^3}{6}D_U^3 F(\bar{U}_{\bar{m}_\epsilon},\bar{m}_\epsilon) (V,V,V) + \frac{\epsilon^4}{24}D_U^4 F(\bar{U}_{\bar{m}_\epsilon},\bar{m}_\epsilon) (V,V,V,V).
\end{aligned}
\end{equation}

We now systematically expand each term of the above expression in $\epsilon$. Starting with the first-order derivative, we expand around $\bar{m}$:
\begin{equation*}
D_U F(\bar{U}_{\bar{m}_\epsilon},\bar{m}_\epsilon)V = D_U F(\bar{U}_{\bar{m}},\bar{m})V - \epsilon^2 \frac{d}{dm} D_U F(\bar{U}_{\bar{m}},\bar{m})V + O(\epsilon^3). 
\end{equation*}
Substituting $V = \alpha \bar{Z} + \epsilon W$ and critically utilizing the fact that $\bar{Z} \in \ker(D_U F(\bar{U}_{\bar{m}},\bar{m}))$, the $O(\epsilon)$ term vanishes, yielding
\begin{equation*}
D_U F(\bar{U}_{\bar{m}_\epsilon},\bar{m}_\epsilon)V = \epsilon D_U F(\bar{U}_{\bar{m}},\bar{m})W - \alpha \epsilon^2 D^2_U F(\bar{U}_{\bar{m}},\bar{m})(\bar{Z},\partial_m \bar{U}_{\bar{m}}) - \alpha\epsilon^2 D_m D_U F(\bar{U}_{\bar{m}},\bar{m})\bar{Z} + M_{1,\epsilon}(\alpha,W),
\end{equation*}
where $\|M_{1,\epsilon}(\alpha,W)\|_{H_1} = O(\epsilon^3)$. Switching to the second-order term, we obtain
\begin{equation*}
D_U^2 F(\bar{U}_{\bar{m}_\epsilon},\bar{m}_\epsilon)(V,V) = \alpha^2 D_U^2 F(\bar{U}_{\bar{m}},\bar{m})(\bar{Z},\bar{Z}) + 2\alpha \epsilon D_U^2 F(\bar{U}_{\bar{m}},\bar{m})(\bar{Z},W)  + M_{2,\epsilon}(\alpha,W),
\end{equation*}
where $\|M_{2,\epsilon}(\alpha,W)\|_{H_1} = O(\epsilon^2)$. Finally, for the third-order term, we have
\begin{equation*}
D_U^3 F(\bar{U}_{\bar{m}_\epsilon},\bar{m}_\epsilon) (V,V,V) = \alpha^3 D_U^3 F(\bar{U}_{\bar{m}},\bar{m}) (\bar{Z},\bar{Z},\bar{Z}) + M_{3,\epsilon}(\alpha,W),
\end{equation*}
where $\|M_{3,\epsilon}(\alpha,W)\|_{H_1} = O(\epsilon)$. 

Assembling these expansions, we find that the lowest non-zero terms are strictly of order $\epsilon^2$:
\begin{align*}
F\left(\bar{U}_{\bar{m}_\epsilon} + \epsilon V, \bar{m}_\epsilon\right) &= \epsilon^2 D_U F(\bar{U}_{\bar{m}},\bar{m})W - \alpha \epsilon^3   D^2_U F(\bar{U}_{\bar{m}},\bar{m})(\bar{Z},\partial_m \bar{U}_{\bar{m}}) - \alpha\epsilon^3  D_m D_U F(\bar{U}_{\bar{m}},\bar{m})\bar{Z}\\
&\quad + \epsilon^2\frac{\alpha^2}{2} D_U^2 F(\bar{U}_{\bar{m}},\bar{m})(\bar{Z},\bar{Z}) + \epsilon^3 \alpha D_U^2 F(\bar{U}_{\bar{m}},\bar{m})(\bar{Z},W)\\
&\quad + \epsilon^3 \frac{\alpha^3}{6}  D_U^3 F(\bar{U}_{\bar{m}},\bar{m}) (\bar{Z},\bar{Z},\bar{Z}) + M_{\epsilon}(\alpha,W),
\end{align*}
where the higher-order remainder is defined as
\begin{align*}
M_{\epsilon}(\alpha,W) &\bydef \epsilon M_{1,\epsilon}(\alpha,W) + \frac{\epsilon^2}{2} M_{2,\epsilon}(\alpha,W) + \frac{\epsilon^3}{6} M_{3,\epsilon}(\alpha,W) + \frac{\epsilon^4}{24} M_{4,\epsilon}(\alpha,W), \\
\text{with} \quad M_{4,\epsilon}(\alpha,W) &= D_U^4 F(\bar{U}_{\bar{m}_\epsilon},\bar{m}_\epsilon) (V,V,V,V).
\end{align*}
In particular, we have bounded the remainder strictly as $\|M_{\epsilon}(\alpha,W)\|_{H_1} = O(\epsilon^4)$. Since $F$ possesses no terms of order $\epsilon^0$ or $\epsilon^1$, the scaled residual $G_\epsilon = \epsilon^{-2}F$ is a well-defined polynomial in $\epsilon$.
In order to study $G_{\epsilon}$, we first analyze the influence of the projection operator $\mathcal{P}$ on the terms of the Taylor expansion derived above. To this end, we recall the following properties established in Section 5 of \cite{constantin_2011}.

\begin{lemma}\label{lem : identity projections}
	The low-frequency projection $\mathcal{P}$ annihilates the first and second derivatives at the bifurcation point:
	\begin{equation*}
	\mathcal{P} D_UF(\bar{U}_{\bar{m}},\bar{m}) = 0 \quad \text{and} \quad \mathcal{P} D_U^2 F(\bar{U}_{\bar{m}},\bar{m})(\bar{Z},\bar{Z}) = 0.
	\end{equation*}
\end{lemma}

Applying Lemma \ref{lem : identity projections} to the Taylor expansion, we observe that the lowest-order terms vanish, allowing us to factor out an additional $\epsilon$ from the low-frequency projection:
\begin{align*}
\frac{1}{\epsilon^3} \mathcal{P} F\left(\bar{U}_{\bar{m}_\epsilon} + \epsilon V, \bar{m}_\epsilon\right) &=  - \alpha \mathcal{P} D^2_U F(\bar{U}_{\bar{m}},\bar{m})(\bar{Z},\partial_m \bar{U}_{\bar{m}}) - \alpha \mathcal{P} D_m D_U F(\bar{U}_{\bar{m}},\bar{m})\bar{Z}\\
&\quad +  \alpha \mathcal{P} D_U^2 F(\bar{U}_{\bar{m}},\bar{m})(\bar{Z},W) +  \frac{\alpha^3}{6} \mathcal{P} D_U^3 F(\bar{U}_{\bar{m}},\bar{m}) (\bar{Z},\bar{Z},\bar{Z}) + \frac{1}{\epsilon^3} \mathcal{P} M_{\epsilon}(\alpha,W).
\end{align*}
Conversely, for the high-frequency projection, using $\mathcal{P}^{\perp} = I -\mathcal{P}$ and Lemma \ref{lem : identity projections}, the $\epsilon^2$ scaling is preserved:
\begin{align*}
\frac{1}{\epsilon^2} \mathcal{P}^\perp F\left(\bar{U}_{\bar{m}_\epsilon} + \epsilon V, \bar{m}_\epsilon\right) &=   D_U F(\bar{U}_{\bar{m}},\bar{m})W - \frac{\alpha^2}{2} \mathcal{P}^\perp D_U^2 F(\bar{U}_{\bar{m}},\bar{m})(\bar{Z},\bar{Z}) + \frac{1}{\epsilon^2} \mathcal{P}^\perp M_{\epsilon}(\alpha,W) + \bar{M}_\epsilon(\alpha,W),
\end{align*}
where the higher-order residual is grouped as
\begin{align*}
\bar{M}_\epsilon(\alpha,W) &\bydef  - \alpha \epsilon \mathcal{P}^\perp D^2_U F(\bar{U}_{\bar{m}},\bar{m})(\bar{Z},\partial_m \bar{U}_{\bar{m}}) - \alpha\epsilon \mathcal{P}^\perp  D_m D_U F(\bar{U}_{\bar{m}},\bar{m})\bar{Z}\\
&\quad +  \epsilon \alpha \mathcal{P}^\perp D_U^2 F(\bar{U}_{\bar{m}},\bar{m})(\bar{Z},W) + \epsilon \frac{\alpha^3}{6} \mathcal{P}^\perp D_U^3 F(\bar{U}_{\bar{m}},\bar{m}) (\bar{Z},\bar{Z},\bar{Z}).
\end{align*}

To evaluate the limit at the removable singularity ($\epsilon = 0$), we define the completely $\epsilon$-independent operator $\bar{G} : H_1 \to H_0$ by isolating the leading-order terms from both projections:
\begin{equation}\label{def : operator bar G}
\begin{aligned}
\bar{G}(U) &=  - \alpha \mathcal{P} D^2_U F(\bar{U}_{\bar{m}},\bar{m})(\bar{Z},\partial_m \bar{U}_{\bar{m}}) - \alpha \mathcal{P} D_m D_U F(\bar{U}_{\bar{m}},\bar{m})\bar{Z} +  \alpha \mathcal{P} D_U^2 F(\bar{U}_{\bar{m}},\bar{m})(\bar{Z},W) \\
&\quad +  \frac{\alpha^3}{6} \mathcal{P} D_U^3 F(\bar{U}_{\bar{m}},\bar{m}) (\bar{Z},\bar{Z},\bar{Z}) + D_U F(\bar{U}_{\bar{m}},\bar{m})W + \frac{\alpha^2}{2} D_U^2 F(\bar{U}_{\bar{m}},\bar{m})(\bar{Z},\bar{Z}),
\end{aligned}
\end{equation}
where $U$ is decomposed uniquely as $U = \alpha \bar{Z} + W$ with $W = \mathcal{P}^\perp U$. 

The Fréchet derivatives of $\bar{G}$, evaluated at $U$ in the directions $V, Z, W \in H_1$ (with scalar components $\beta, \gamma, \delta$ and orthogonal components $v, z, w$ respectively), follow directly from differentiation:
\begin{align*}
D\bar{G}(U)V &= \beta \mathcal{P} D^2_U F(\bar{U}_{\bar{m}},\bar{m})(\bar{Z}, W - \partial_m \bar{U}_{\bar{m}}) - \beta \mathcal{P} D_m D_U F(\bar{U}_{\bar{m}},\bar{m})\bar{Z} +  \alpha \mathcal{P} D_U^2 F(\bar{U}_{\bar{m}},\bar{m})(\bar{Z},v) \\
&\quad +  \beta \frac{\alpha^2}{2} \mathcal{P} D_U^3 F(\bar{U}_{\bar{m}},\bar{m}) (\bar{Z},\bar{Z},\bar{Z}) + D_U F(\bar{U}_{\bar{m}},\bar{m})v + \beta \alpha D_U^2 F(\bar{U}_{\bar{m}},\bar{m})(\bar{Z},\bar{Z}), \\[1em]
D^2\bar{G}(U)(V,Z) &= \beta \mathcal{P} D^2_U F(\bar{U}_{\bar{m}},\bar{m})(\bar{Z},z) +  \gamma \mathcal{P} D_U^2 F(\bar{U}_{\bar{m}},\bar{m})(\bar{Z},v) \\
&\quad +  \beta \gamma \alpha \mathcal{P} D_U^3 F(\bar{U}_{\bar{m}},\bar{m}) (\bar{Z},\bar{Z},\bar{Z}) + \beta \gamma D_U^2 F(\bar{U}_{\bar{m}},\bar{m})(\bar{Z},\bar{Z}), \\[1em]
D^3\bar{G}(U)(V,W,Z) &= \beta \gamma \delta \mathcal{P} D_U^3 F(\bar{U}_{\bar{m}},\bar{m}) (\bar{Z},\bar{Z},\bar{Z}).
\end{align*}

The following lemma rigorously justifies the use of $\bar{G}$ as the exact limit of $G_\epsilon$ as $\epsilon \to 0$.

\begin{lemma}
	Let $r_0 > 0$. There exist constants $C_0(r_0), C_1(r_0), C_2(r_0), C_3(r_0)$, dependent only on $r_0$, such that for all $\epsilon \in [0,\epsilon_0]$ and for all $U \in H_1$ satisfying $\|U - \bar{U}(\epsilon) \|_{H_1} \leq r_0$, the following error bounds hold:
	\begin{align*}
	\|G_{\epsilon}(U) - \bar{G}(U)\|_{H_0} &\leq C_0(r_0) \epsilon, \\
	\sup_{V \in B_1(0)} \|DG_{\epsilon}(U)V - D\bar{G}(U)V\|_{H_0} &\leq C_1(r_0) \epsilon, \\
	\sup_{V,W \in B_1(0)}\|D^2G_{\epsilon}(U)(V,W) - D^2\bar{G}(U)(V,W)\|_{H_0} &\leq C_2(r_0) \epsilon, \\
	\sup_{V,W,Z \in B_1(0)}\|D^3G_{\epsilon}(U)(V,W,Z) - D^3\bar{G}(U)(V,W,Z)\|_{H_0} &\leq C_3(r_0)\epsilon.
	\end{align*}
\end{lemma}

\subsection{Analytic proof details}

We now turn to the construction of the infinite-dimensional tail $A_\infty$. Because this operator depends on the high-frequency behavior of the Jacobian $DG_\epsilon(\bar{U}(\epsilon))$, we first establish a preliminary lemma characterizing its tail projection. Crucially, this lemma demonstrates that the $1/\epsilon$ scaling introduced to desingularize the fundamental mode completely vanishes at high frequencies. This isolation is vital: it guarantees that our analytical tail bounds do not blow up as $\epsilon \to 0$, allowing us to safely bound the high-frequency modes uniformly across the bifurcation point.

\begin{lemma}\label{lem : tail of DG local}
	Let $K_1 \geq 2K_0 +1$ and $\epsilon \in [0,\epsilon_0]$. Defining $\bar{U}(\epsilon) \bydef \bar{U}_{\bar{m}-\epsilon^2} + \epsilon^2 P(\epsilon) \bar{V}(\epsilon)$, we have
	\begin{align}\label{eq : identity projection on G}
	DG_\epsilon(\bar{V}(\epsilon)) \pi^{>K_1} =  
	\pi^{>K_1-2K_0} D_U F(\bar{U}(\epsilon), \bar{m}-\epsilon^2) \pi^{>K_1}.
	\end{align}
\end{lemma}

\begin{proof}
	Applying the chain rule to the definition of $G_\epsilon$ in \eqref{def : zero finding G}, we obtain
	\begin{align}\label{eq : chain rule DG tail}
	DG_\epsilon(\bar{V}(\epsilon)) \pi^{>K_1} = P(\epsilon) D_U F(\bar{U}(\epsilon), \bar{m}-\epsilon^2) P(\epsilon) \pi^{>K_1}.
	\end{align}
	By construction, the projection $\mathcal{P}$ isolates modes in $\{-1, 1\}$. Since $K_1 > 1$, we have $\mathcal{P} \pi^{>K_1} = 0$ and $\mathcal{P}^\perp \pi^{>K_1} = \pi^{>K_1}$. Consequently, the rightmost scaling operator acts as the identity: $P(\epsilon) \pi^{>K_1} = \pi^{>K_1}$.
	
	Next, observe that $\bar{U}(\epsilon) = \pi^{\leq K_0} \bar{U}(\epsilon)$. This holds because $\bar{U}(\epsilon) = \pi^{\leq K_0} \bar{U}(\epsilon)$ by the expansion \eqref{eq : approx sol local bifurcation}, and the trivial flat state $\bar{U}_{\bar{m}-\epsilon^2}$ possesses a non-zero coefficient only at index $0$. 
	
	Because the base state is truncated at $K_0$, the nonlinearities in $F$ (as defined in \eqref{def : DF} and \eqref{eq : functions v1 to v4 and C}) can expand the spectral support by at most $2K_0$. Thus, acting on a tail vector, we have
	\begin{align}\label{eq : step proof DG}
	D_U F(\bar{U}(\epsilon), \bar{m}-\epsilon^2) \pi^{>K_1} 
	= \pi^{>K_1 - 2K_0} D_U F(\bar{U}(\epsilon), \bar{m}-\epsilon^2) \pi^{>K_1}.
	\end{align}
	Furthermore, because we assume $K_1 \geq 2K_0 + 1$, it follows that $K_1 - 2K_0 \ge 1$. This implies that the output in \eqref{eq : step proof DG} lives entirely in the space $\pi^{>1}$, meaning it has no components in the fundamental modes $\{-1, 0, 1\}$. 
	
	Therefore, when the leftmost operator $P(\epsilon)$ in \eqref{eq : chain rule DG tail} is applied to this output, it also acts purely as the identity matrix. The scaling operators vanish entirely, yielding
	\begin{align*}
	DG_\epsilon(\bar{V}(\epsilon)) \pi^{>K_1} = \pi^{>K_1 - 2K_0} D_U F(\bar{U}(\epsilon), \bar{m}-\epsilon^2) \pi^{>K_1},
	\end{align*}
	which completes the proof.
\end{proof}

\subsection{Proofs of the bound estimates}

\begin{proof}[Proof of Lemma \ref{lem : Z1 epsilon}]
	This bound follows as a direct consequence of the arguments in the proofs of Lemmas \ref{lem : Z1 bound} and \ref{lem : tail of DG local}.
\end{proof}

\begin{proof}[Proof of Lemma \ref{lem: Y bounds}]
	Applying the definition of $G_\epsilon$ from \eqref{def : zero finding G}, we observe that
	\begin{equation*}
	\|A(\epsilon)G_{\epsilon}(\bar{V}(\epsilon))\|_{H_1} = \frac{1}{\epsilon^2}\|A(\epsilon)P(\epsilon) F(\bar{U}(\epsilon), \bar{m}-\epsilon^2)\|_{H_1}.
	\end{equation*}
	Because the high-frequency projection isolates the tail from the bifurcation scaling, the remainder of the bound follows identically from the proof of Lemma \ref{lem : Y bound single wave}.
\end{proof}

\begin{proof}[Proof of Lemma \ref{lem : Z2 epsilon}]
	Using the mean value inequality for Banach spaces, we obtain
	\begin{align*}
	\|A(\epsilon)&\left[DG_\epsilon(\bar{V}(\epsilon) + h, \bar{m}-\epsilon^2) - DG_\epsilon(\bar{V}(\epsilon), \bar{m}-\epsilon^2)\right]\|_{\mathcal{B}(H_1)} \\
	&\leq r \sup_{h, V \in B_1(0)} \|A(\epsilon)D^2G_\epsilon(\bar{V}(\epsilon),\bar{m}-\epsilon^2)(h,V)\|_{\mathcal{B}(H_1)} \\
	&\quad + \frac{1}{2} r^2 \sup_{h, V, W, Z \in B_1(0)} \|A(\epsilon)D^3G_\epsilon(\bar{V}(\epsilon) + rh,\bar{m}-\epsilon^2)(V,W,Z)\|_{\mathcal{B}(H_1)}.
	\end{align*}
	Recalling \eqref{def : zero finding G}, we can evaluate the third derivative term as
	\begin{multline*}
	\|A(\epsilon)D^3G_\epsilon(\bar{V}(\epsilon) + rh, \bar{m}-\epsilon^2)(V,W,Z)\|_{H_1} \\
	= \epsilon^4\|A(\epsilon)P(\epsilon)D^3F(\bar{U}(\epsilon) + rh, \bar{m}-\epsilon^2)(P(\epsilon)V, P(\epsilon)W, P(\epsilon)Z)\|_{H_1}.
	\end{multline*}
	Now, from the definition of the projection \eqref{eq : definition projection P}, we have
	\begin{equation*}
	\epsilon P(\epsilon) = \mathcal{P} + \epsilon \mathcal{P}^{\perp}.
	\end{equation*}
	In particular, we clearly get $\|\epsilon P(\epsilon)U\|_{H_1} \leq \|U\|_{H_1}$ for all $U \in H_1$. Using the above and the trilinearity of the third derivative, we get
	\begin{align*}
	&\epsilon^4\|A(\epsilon)P(\epsilon)D^3F(\bar{U}(\epsilon) + rh,\bar{m}-\epsilon^2)(P(\epsilon)V,P(\epsilon)W,P(\epsilon)Z)\|_{H_1} \\
	= & ~ \epsilon \|A(\epsilon)P(\epsilon)D^3F(\bar{U}(\epsilon) + rh,\bar{m}-\epsilon^2)(\epsilon P(\epsilon)V, \epsilon P(\epsilon)W, \epsilon P(\epsilon)Z)\|_{H_1} \\
	\leq & ~ \epsilon \|\mathcal{W} A(\epsilon) P(\epsilon)\|_{\ell^1} \sup_{V,W,Z,h \in B_1(0)} \|D^3F(\bar{U}(\epsilon) + rh,\bar{m}-\epsilon^2)(V,W,Z)\|_{H_1} \\
	\leq & ~ 2\sup_{\epsilon \in [0, \epsilon_0]} \epsilon \|\mathcal{W} A(\epsilon) P(\epsilon)\|_{\ell^1} Z_{D^3}(r),
	\end{align*}
	where we also used the proof of Lemma \ref{lem : Z2 single wave} for the last inequality.
	
	It remains to estimate $\sup_{h, V \in B_1(0)} \|A(\epsilon)D^2G_\epsilon(\bar{U}(\epsilon),\bar{m}-\epsilon^2)(h,V)\|_{H_1}$. Recalling \eqref{def : zero finding G}, we have
	\begin{equation*}
	\|A(\epsilon)D^2G_\epsilon(\bar{U}(\epsilon),\bar{m}-\epsilon^2)(h,V)\|_{H_1} = \epsilon^2 \|A(\epsilon)P(\epsilon)D^2F(\bar{U}(\epsilon),\bar{m}-\epsilon^2)(P(\epsilon)h,P(\epsilon)V)\|_{H_1}. 
	\end{equation*}
	By exploiting the bilinearity of $D^2F$, we can strategically distribute the $\epsilon^2$ scaling across the low and high-frequency projections of our inputs. Using the properties of the $\ell^1$-norm combined with \eqref{eq : definition projection P}, we first split the argument $h$ to obtain
	\begin{multline}\label{eq : Z2 epsilon step 1}
	\epsilon^2 \sup_{h, V \in B_1(0)} \|A(\epsilon) P(\epsilon) D^2F(\bar{U}(\epsilon),\bar{m}-\epsilon^2)(P(\epsilon)h,P(\epsilon)V)\|_{H_1} \\
	\leq \max\Bigg\{ \epsilon \sup_{h, V \in B_1(0)} \|A(\epsilon) P(\epsilon) D^2F(\bar{U}(\epsilon),\bar{m}-\epsilon^2)(\mathcal{P}h, P(\epsilon)V)\|_{H_1}, \\
	\epsilon^2 \sup_{h, V \in B_1(0)}\|A(\epsilon) P(\epsilon) D^2F(\bar{U}(\epsilon),\bar{m}-\epsilon^2)(\mathcal{P}^\perp h, P(\epsilon)V)\|_{H_1} \Bigg\}.
	\end{multline}
	Focusing on the first term of \eqref{eq : Z2 epsilon step 1} and applying a similar decomposition to the second argument $V$, we get
	\begin{multline}\label{eq : Z2 epsilon step 2}
	\epsilon \sup_{h, V \in B_1(0)} \|A(\epsilon) P(\epsilon) D^2F(\bar{U}(\epsilon),\bar{m}-\epsilon^2)(\mathcal{P}h, P(\epsilon)V)\|_{H_1} \\
	\leq \max\Bigg\{ \sup_{h, V \in B_1(0)} \|A(\epsilon) P(\epsilon) D^2F(\bar{U}(\epsilon),\bar{m}-\epsilon^2)(\mathcal{P}h, \mathcal{P}V)\|_{H_1}, \\ 
	\epsilon \sup_{h, V \in B_1(0)} \|A(\epsilon) P(\epsilon) D^2F(\bar{U}(\epsilon),\bar{m}-\epsilon^2)(\mathcal{P}h, \mathcal{P}^{\perp}V)\|_{H_1}\Bigg\}. 
	\end{multline}
	Looking now at the second term of \eqref{eq : Z2 epsilon step 1} and splitting $V$ again, we find
	\begin{multline}\label{eq : Z2 epsilon step 3}
	\epsilon^2 \sup_{h, V \in B_1(0)} \|A(\epsilon) P(\epsilon) D^2F(\bar{U}(\epsilon),\bar{m}-\epsilon^2)(\mathcal{P}^\perp h, P(\epsilon)V)\|_{H_1} \\
	\leq \max\Bigg\{ \epsilon \sup_{h, V \in B_1(0)}\|A(\epsilon) P(\epsilon) D^2F(\bar{U}(\epsilon),\bar{m}-\epsilon^2)(\mathcal{P}^\perp h, \mathcal{P}V)\|_{H_1}, \\ 
	\epsilon^2 \sup_{h, V \in B_1(0)}\|A(\epsilon) P(\epsilon) D^2F(\bar{U}(\epsilon),\bar{m}-\epsilon^2)(\mathcal{P}^\perp h, \mathcal{P}^\perp V)\|_{H_1}\Bigg\}.
	\end{multline}
	It remains to estimate the four terms on the right-hand sides of \eqref{eq : Z2 epsilon step 2} and \eqref{eq : Z2 epsilon step 3}. We start with the purely finite-dimensional term from \eqref{eq : Z2 epsilon step 2}.
	Given $h, V \in B_1(0)$, we have by the definition of the projection $\mathcal{P}$ in \eqref{eq : definition projection P} that there exist scalars $\alpha, \beta \in \mathbb{R}$ with $|\alpha| \leq 1$ and $|\beta| \leq 1$ such that $\mathcal{P}h = \alpha \bar{Z}$ and $\mathcal{P}V = \beta \bar{Z}$. 
	Consequently, by the bilinearity of $D^2F$, we have
	\begin{align*}
	\sup_{h,V \in B_1(0)} \|A(\epsilon) P(\epsilon) D^2F(\bar{U}(\epsilon),\bar{m}-\epsilon^2)(\mathcal{P}h,\mathcal{P}V)\|_{H_1} &= \|A(\epsilon) P(\epsilon) D^2F(\bar{U}(\epsilon),\bar{m}-\epsilon^2)(\bar{Z},\bar{Z})\|_{H_1} \\
	&\leq Z_{D^2,1}.
	\end{align*}
	Switching to the mixed-frequency term of \eqref{eq : Z2 epsilon step 2}, we extract the scaled approximate inverse to obtain
	\begin{multline*}
	\epsilon \sup_{h,V \in B_1(0)} \|A(\epsilon) P(\epsilon) D^2F(\bar{U}(\epsilon),\bar{m}-\epsilon^2)(\mathcal{P}h,\mathcal{P}^{\perp}V)\|_{H_1} \\
	\leq \epsilon \|\mathcal{W}A(\epsilon) P(\epsilon)\|_{\mathcal{B}(\ell^1)} \sup_{h,V \in B_1(0)} \|D^2F(\bar{U}(\epsilon),\bar{m}-\epsilon^2)(\mathcal{P}h, \mathcal{P}^{\perp}V)\|_{H_1}.
	\end{multline*}
	Using the fact that $\mathcal{P}^\perp V, \mathcal{P}h \in B_1(0)$ for all $V,h \in B_1(0)$, combined with the proof of Lemma \ref{lem : Z2 single wave}, we find
	\begin{equation*}
	\epsilon \sup_{h,V \in B_1(0)} \|A(\epsilon) P(\epsilon) D^2F(\bar{U}(\epsilon),\bar{m}-\epsilon^2)(\mathcal{P}h,\mathcal{P}^{\perp}V)\|_{H_1} \leq \sup_{\epsilon \in [0, \epsilon_0]} \epsilon \|\mathcal{W}A(\epsilon) P(\epsilon)\|_{\mathcal{B}(\ell^1)} Z_{D^2,2}.
	\end{equation*}
	By symmetry, this bound of $\sup_{\epsilon} \epsilon \|\mathcal{W}A(\epsilon) P(\epsilon)\|_{\mathcal{B}(\ell^1)} Z_{D^2,2}$ also provides an upper limit for the other mixed-frequency term in \eqref{eq : Z2 epsilon step 3}. Finally, we estimate the purely high-frequency term of \eqref{eq : Z2 epsilon step 3} by isolating the scaling factors:
	\begin{multline*}
	\epsilon^2 \sup_{h, V \in B_1(0)} \|A(\epsilon) P(\epsilon) D^2F(\bar{U}(\epsilon),\bar{m}-\epsilon^2)(\mathcal{P}^\perp h,\mathcal{P}^\perp V)\|_{H_1} \\
	\leq \epsilon_0 \sup_{\epsilon \in [0,\epsilon_0]} \epsilon \|\mathcal{W}A(\epsilon) P(\epsilon)\|_{\mathcal{B}(\ell^1)} \sup_{h,V \in B_1(0)} \| D^2F(\bar{U}(\epsilon),\bar{m}-\epsilon^2)(\mathcal{P}^\perp h,\mathcal{P}^\perp V)\|_{H_1}.
	\end{multline*}
	We conclude the bound by again using that $\mathcal{P}^\perp V, \mathcal{P}^\perp h \in B_1(0)$ combined with the analytic tail bound from the proof of Lemma \ref{lem : Z2 single wave}.
\end{proof}

\begin{proof}[Proof of Lemma \ref{lem : enclosure of DmU(m)}]
	We start by differentiating $F(\tilde{U}(m),m) = 0$ with respect to $m$ and get 
	\begin{align*}
	D_U F(\tilde{U}(m),m) \partial_m \tilde{U}(m) + D_m F(\tilde{U}(m),m) = 0 \quad \text{ for all } m \in [m_0, m_1].
	\end{align*}
	This implies that 
	\begin{align*}
	\partial_m \tilde{U}(m) = - D_U F(\tilde{U}(m),m)^{-1}  D_m F(\tilde{U}(m),m) \quad \text{ for all } m \in [m_0, m_1].
	\end{align*}
	Suppose that $\partial_m \bar{U}(m)$ is an approximation of $\partial_m \tilde{U}(m)$, then we have that 
	\begin{align*}
	\|\partial_m \bar{U}(m) - \partial_m \tilde{U}(m)\|_X &= \| D_U F(\tilde{U}(m),m)^{-1}(D_U F(\tilde{U}(m),m) \partial_m \bar{U}(m) +  D_m F(\tilde{U}(m),m))\|_X.
	\end{align*}
	Now, observe that 
	\begin{align*}
	&\|I  - A(m) D_UF(\tilde{U}(m),m)\|_{\mathcal{B}(H_1)} \\
	& ~~ \leq  \|I  - A(m) D_UF(\bar{U}(m),m)\|_{\mathcal{B}(H_1)}  + \|A(m) (D_U F(\tilde{U}(m),m) - D_UF(\bar{U}(m),m))\|_{\mathcal{B}(H_1)} \\
	& ~~\leq  Z_1 + Z_2(r_0)r_0
	\end{align*}
	using that $Z_1$ and $Z_2$ satisfy \eqref{eq: definition Y0 Z1 Z2 continuation}. But since $ Z_1 + Z_2(r_0)r_0 < 1$ thanks to \eqref{eq : condition contraction continuation}, we can use a Neumann series argument and obtain that
	\begin{align*}
	D_U F(\tilde{U}(m),m)^{-1} &= \sum_{k=0}^\infty (I - A(m) D_U F(\tilde{U}(m),m))^k A(m)\\
	&= B(m) A(m)
	\end{align*}
	where $B(m) = \sum_{k=0}^\infty (I - A(m) D_U F(\tilde{U}(m),m))^k$. In particular, we have that 
	\begin{align*}
	\|B(m)\|_{\mathcal{B}(H_1)} \leq \sum_{k=0}^\infty \|I - A(m) D_U F(\tilde{U}(m),m)\|_{\mathcal{B}(H_1)}^k \leq  \frac{1}{1 - Z_1 - Z_2(r_0)r_0}.
	\end{align*}
	This implies that 
	\begin{align*}
	&\| D_U F(\tilde{U}(m),m)^{-1}(D_U F(\tilde{U}(m),m) \partial_m \bar{U}(m) +  D_m F(\tilde{U}(m),m))\|_X\\
	\leq &\frac{1}{1 - Z_1 - Z_2(r_0)r_0} \| A(m)(D_U F(\tilde{U}(m),m) \partial_m \bar{U}(m) +  D_m F(\tilde{U}(m),m))\|_X\\
	\leq & \frac{\| A(m)(D_U F(\bar{U}(m),m) \partial_m \bar{U}(m) +  D_m F(\bar{U}(m),m))\|_X}{1 - Z_1 - Z_2(r_0)r_0}  \\
	&~~~~ + \frac{\| A(m)(D_UF(\tilde{U}(m),m) - D_U F(\bar{U}(m),m)) \partial_m \bar{U}(m)\|_X}{1 - Z_1 - Z_2(r_0)r_0}\\
	&~~~~ + \frac{\| A(m)(D_mF(\tilde{U}(m),m) - D_m F(\bar{U}(m),m))\|_X}{1 - Z_1 - Z_2(r_0)r_0}. 
	\end{align*}
	Now it remains to estimate the right hand-side of the above expression. Notice that the first term is already simplified. For the second term,  we use \eqref{eq: definition Y0 Z1 Z2 continuation} and  obtain that 
	\begin{align*}
	\frac{\| A(m)(D_UF(\tilde{U}(m),m) - D_U F(\bar{U}(m),m)) \partial_m \bar{U}(m)\|_X}{1 - Z_1 - Z_2(r_0)r_0} \leq \sup_{m \in [m_0, m_1]}\frac{Z_2(r_0)\|\partial_m \bar{U}(m)\|_X r_0}{1 - Z_1 - Z_2(r_0)r_0}.
	\end{align*}
	Finally, we focus on  the term $\| A(D_mF(\tilde{U}(m),m) - D_m F(\bar{U}(m),m))\|_X$. First, observe that 
	\begin{align*}
	D_mF(U,m) = \begin{pmatrix}
	0\\
	2(m D_y e_0 + \frac{\gamma}{2} D_y(a^2) - \gamma a * (D_ya))*(D_y e_0)
	\end{pmatrix}
	\text{ for all } U = (Q,a).
	\end{align*}
	Recalling that $(D_y e_0)_n = \frac{1}{h} e_0$, we obtain that\begin{align*}
	D_mF(U,m) = \begin{pmatrix}
	0\\
	\frac{2}{h}(m D_y e_0 + \frac{\gamma}{2} D_y(a^2) - \gamma a * (D_ya))
	\end{pmatrix}
	\text{ for all } U = (Q,a).
	\end{align*} 
	In particular, we obtain that 
	\begin{align*}
	\|A(D_mF(\tilde{U}(m),m) - D_m F(\bar{U}(m),m))\|_X &\leq \|A\|_{\ell^1 \to X} \|D_mF(\tilde{U}(m),m) - D_m F(\bar{U}(m),m)\|_{\ell^1}\\
	& \leq  \|A\|_{\ell^1 \to X} \frac{4}{h} \gamma \coth(h) (\|\bar{a}(m)\|_X + r_0) r_0
	\end{align*}
	using that $\|\tilde{a}(m) - \bar{a}(m)\|_X \leq r_0$ and that $\|D_y a\|_{\ell^1} \leq \coth(h) \|a\|_X$.
	This concludes the proof.
\end{proof}

\subsection{Error estimate for difference of products }

Using the Banach algebra property of $(X, \|\cdot\|_X)$, established in Lemma \ref{lem:Banach_algebra}, we are able to  upper bound the $X$-norm of a discrete convolution. In fact, we use this result in practice for estimating the defect between the discrete convolution of $m$ sequences, thanks to the defect on each individual sequence. This leads to an automatic strategy for computing the defect between convolutions of sequences. We summarize our result in the following lemma. 
\begin{lemma}\label{lem : estimate difference products}
	Let $m \in \mathbb{N}$ and consider sequences $(u_i)_{i=1}^m \in X^m$ and $(\bar{u}_i)_{i=1}^m \in X^m$. Moreover, let $\delta_i = \|u_i - \bar{u}_i\|_X$ and $\beta_i = \|\bar{u}_i\|_X + \delta_i$. Then, we have that 
	\begin{align*}
	\|u_1*\cdots*u_m - \bar{u}_1*\cdots*\bar{u}_m\|_X \leq \sum_{i=1}^m \delta_i \beta_1\beta_2 \cdots \beta_{i-1}\beta_{i+1} \cdots \beta_{m-1}\beta_m.
	\end{align*}
\end{lemma}
\begin{proof}
	The proof is a direct consequence of triangle inequalities combined with Lemma \ref{lem:Banach_algebra}.
\end{proof}

\section*{Acknowledgments}

  MC was supported by the ANR project CAPPS: ANR-23-CE40-0004-01 and  by the FMJH :  ANR-22-EXES-0013.

  SVH was partially supported by the NSF grant DMS- 2511086.

\bibliographystyle{abbrv}
\bibliography{biblio}

\end{document}